\documentclass[10pt,twoside, a4paper]{article}

\usepackage{nemish}

\usepackage{hyperref} 

\title{
Mesoscopic Spectral CLT for Block Correlated Random Matrices
}

\usepackage[auth-sc-lg,affil-sl]{authblk}
\author{Torben Kr\"{u}ger\footnote{\hspace{0.15cm} Partially supported by VILLUM FONDEN research grant no. 29369 and   Novo Nordisk Fonden Project Grant 0064428\newline  E-mail addresses: \href{mailto:torben.krueger@fau.de}{torben.krueger@fau.de} (T. Kr\"{u}ger), \href{mailto:ynemish@ucsd.edu}{ynemish@ucsd.edu} (Yu. Nemish)}}
\affil{Department of Mathematical Sciences, University of Copenhagen}
\affil{Department of Mathematics, FAU Erlangen-N\"{u}rnberg}
\author{Yuriy Nemish}
\affil{Department of Mathematics, University of California, San Diego}

\date{\today} 

\begin{document}

\maketitle

\thispagestyle{empty} 

\begin{abstract}
    For random matrices with block correlation structure we show that the fluctuations of linear eigenvalue statistics are Gaussian on all mesoscopic scales with universal variance which coincides with that of the Gaussian unitary or Gaussian orthogonal ensemble, depending on the symmetry class of the model.
   The main tool used for determining this variance is a two-point version of the matrix-valued Dyson equation, that  encodes the asymptotic behavior of the product of resolvents at different spectral parameters.
\end{abstract}


\section{Introduction}
\label{sec:intro}

The eigenvalues of Hermitian random matrices form a strongly correlated point process on the real line. To gain insights into the underlying correlations we study linear statistics of this process. 
More generally, for $N$ real random variables $\lambda_{1}, \dots, \lambda_{N}$ their linear statistics is  $L_{N}(f):=\sum_{i=1}^{N} f(\lambda_i)$ for some test function $f$. If $\frac{1}{N}L_{N}(f)$ converges as $N \to \infty$ and becomes nonrandom, we say that the law of large numbers holds. Such results are well know not only in the case when $\lambda_1, \dots, \lambda_n$ are independent, but also when they are the eigenvalues of an $N \times N$ random matrix ensemble. In this case $\lim_{N \to \infty}\frac{1}{N}L_{N}(f) = \int f(x) \rho(x) d x$, where $\rho$ is the asymptotic spectral density. 

The next important question to ask about $\frac{1}{N}L_{N}(f)$ concerns the size and distribution of its fluctuations around the limit. In the case of independent random variables the central limit theorem (CLT) states that these fluctuations are Gaussian of order $N^{-1/2}$. The strong correlation between the eigenvalues of random matrices, however,  typically reduces the fluctuation to the order $N^{-1}$, while the distribution remains Gaussian. 
Such results were proved, e.g., for invariant ensembles \cite{Joha98}, Wigner matrices with non-Gaussian i.i.d. entries \cite{KhorKhorPast96, BaiYao05}, sample covariance matrices \cite{BaiSilv04}, ensembles with external source \cite{JiLee20}, and polynomials of several independent random matrices \cite{BeliCapiDallFevr}.

The law of large numbers and CLT described above provide information about the eigenvalues on global scales $\eta =O(1)$. To resolve the eigenvalue distribution on mesoscopic scales $ N^{-1} \ll \eta \le 1$ above the typical spacing distance, or even microscopic scales of order $\eta = O(N^{-1})$, the compactly supported test function $f$ with $\int f(x) \, dx =1$ is rescaled  to capture an $\eta$-sized neighbourhood around $x_0$, i.e.,
\[
f_\eta(x) := f \bigg(\frac{x-x_0}{\eta}\bigg)\,.
\]
On mesoscopic scales the linear statistics for $f_\eta$ still involves a large number of  $O(N \eta)$ eigenvalues and thus  the law of large numbers now takes the form $\frac{1}{N \eta} L_{N}(f_\eta) \to \rho(x_0)$.  Such local laws have been established for a large variety of random matrix models, including Wigner matrices \cite{ErdoSchlYau09a}, deformed Wigner matrices \cite{LeeSchnStetYau16,HeKnowRose18}, matrices with variance profiles \cite{AjanErdoKrug17} and correlations \cite{AjanErdoKrug19,ErdoKrugSchr19}, invariant ensembles \cite{BourErdoYau14a}, and polynomials in several  random matrices \cite{Ande15,ErdoKrugNemi_Poly}. 
 
The fluctuation of linear eigenvalue statistics exhibits a striking universality phenomenon on mesoscopic scales. Not only its order  and  Gaussian distribution become universal, but also the corresponding variance, namely
\begin{equation}
  \label{mesoscopic CLT - Intro}
L_{N}(f_\eta)- \EE L_{N}(f_\eta) \to \mathcal{N} (0, v_f)\,, \quad v_f = \frac{1}{2\pi^2\beta} \int_{-\infty}^\infty\int_{-\infty}^\infty \bigg( \frac{f(x)-f(y)}{x-y}\bigg)^2 dx dy\,,
\end{equation}
is independent of $\eta$ and $x_0$. The only signature of the underlying random matrix ensemble still present in these fluctuations is its symmetry class $\beta$, where $\beta =1$ corresponds to real symmetric  and $\beta =2$ to complex Hermitian matrices. 
Versions of the mesoscopic CLT \eqref{mesoscopic CLT - Intro} have been established for a wide class of random matrix models. These include the classical Gaussian ensembles \cite{BoutKhor99a,FyodKhorSimm16},  Wigner matrices, for which the CLT has been shown covering an increasing range of mesoscopic regimes \cite{BoutKhor99b,HeKnow17,LodhSimm15}, Wigner-type matrices \cite{Riab}, invariant ensembles  \cite{BekeLodh18,Lamb18}, deformed Wigner matrices \cite{LiSchnXu21}, functions of Wigner matrices \cite{CipoErdoSchr23}, Dyson Brownian motion \cite{DuitJoha18,HuanLand19}, band matrices \cite{ErdoKnow15b}, and free sums \cite{BaoSchnXu22}.
The mesoscopic CLT is also used to establish fluctuations of individual eigenvalues \cite{LandLopaSoso23}.
For non-Hermitian random matrices with independent entries  and Coulomb gases in dimension two, mesoscopic CLTs have recently been proven in \cite{CipoErdoSchr} and \cite{BaueBourNikuYau19}, respectively. 

While the local law has been established for Hermitian and non-Hermitian models with very general correlation structure among the entries of the underlying random matrix, the current proofs of mesoscopic CLTs either assume independent entries or isotropic randomness in the matrix, which is invariant under either the unitary or orthogonal group. In this work we establish the CLT \eqref{mesoscopic CLT - Intro} for linear matrix pencils  in several random matrices, which do not satisfy either  condition. Instead they belong to a class of random matrix ensembles with non-isotropic  block correlations among their entries.

For matrices $\Xb_0,\Xb_1, \dots, \Xb_d \in \CC^{N \times N}$ a linear matrix pencil (LMP) is a linear combination of these matrices with matrix valued coefficients, i.e., a matrix in $\CC^{n \times n} \otimes \CC^{N \times N}$ of the form 
\begin{equation}
  \label{LMP - Intro}
\Hb = \sum_{i=0}^d L_i \otimes \Xb_i\,.
\end{equation}
These LMPs find applications in optimization \cite{WolkSaigVandBook}, systems engineering \cite{SkelIwasGrigBook}, theoretical computer science \cite{BersReutBook} (see, e.g., \cite{KlepPascVolc17,GargGurvOlivWigd20} and references therein for numerous other applications).

In this work we consider the case when $\Xb_{0}= \Ib_N$ and $\Xb_{1}, \dots, \Xb_{d}$ are Wigner matrices.
These random LMPs are of particular interest in the study of the evolution of ecosystems \cite{HastJuhaSchr92} and neural networks \cite{RajaAbbo06} (see also \cite{AljaRenfSter15,ErdoKrugRenf} and references therein).
Moreover, such LMPs appear as linearizations of non-commutative polynomials $P(\Xb_1, \dots, \Xb_{d})$ in the Wigner matrices. We use the result of the current work and this fact in the companion paper \cite{KrugNemi_mesoCLTpoly} to establish the mesoscopic CLT for linear eigenvalue statistics of such polynomials.
The linearization technique also extends to non-commutative rational functions of random matrices. These are used,  e.g.,  in the study of transport properties  through disordered quantum systems, such as quantum dots \cite{ErdoKrugNemi_qDot} and, therefore, the corresponding spectral CLTs can be used to determine fluctuations of transport eigenvalues.

Linearization techniques have been used extensively in free probability \cite{HaagThor05,Ande13,HeltMaiSpei18,MaiSpeiYin2}.
Various properties of the spectrum of the LMPs of free elements have been established, e.g. in \cite{HaagThor99,Thor00,BannCebr23}.
The study of the CLTs in the context of free probability uses the theory of  second order freeness (see, e.g., \cite{CollMingSniaSpei07}).

The matrix coefficients $L_i$ of the LMP $\Hb$ encode the correlation structure of the entries and determine the spectral density via a nonlinear matrix equation, the Matrix Dyson Equation (MDE), whose solution $M(z)$ is interpreted as the expectation value $\EE[\Gb(z)]$ of the resolvent  $\Gb(z):=(\Hb -z)^{-1}$ in the large $N$ limit. Incorporating this non-trivial structure of the resolvent in the calculation of the fluctuations is one of the main novelties in this work and can be extended to other models with decaying correlations, such as the Kronecker random matrices in \cite{AltErdoKrugNemi_Kronecker} or matrices with general  decaying correlations in \cite{AjanErdoKrug19,ErdoKrugSchr19}. To keep the presentation simple, however, we do not pursue this direction. Instead we show that the mesoscopic CLT \eqref{mesoscopic CLT - Intro} with $N$ replaced by $n N$ and in the limit $N \to \infty$ holds for matrices $\Hb$ of the form \eqref{LMP - Intro}, where the symmetry indicator $\beta \in \{1,2\}$ depends on whether the  Wigner matrices $\Xb_1, \dots,  \Xb_{d}$ and their coefficient matrices $L_i$ are real symmetric or complex Hermitian.
Our proof relies solely on  resolvent methods and does not involve the application of  Dyson Brownian motion.
This allows to obtain the CLT on \emph{all} mesoscopic scales for models with large zero blocks of size $O(N)$.

\section{Model and results}
\label{sec:results}

In this paper we study random matrix models having general block correlation structures with blocks drawn from  random matrix ensembles with independent identically distributed (\emph{i.i.d.}) entries.

\textbf{The model.}
Fix $d, n \in \NN$.
Let $K_0, L_1, \ldots, L_{d} \in \CC^{n\times n}$ be deterministic matrices, and let $\Xb_{1},\ldots, \Xb_{d} \in \CC^{N \times N}$, $\Xb_{\alpha} = \big( x_{ij}^{(\alpha)} \big)_{i,j=1}^{N}$, $1\leq \alpha \leq d$, be independent $N\times N$ random matrices with \emph{i.i.d.} entries.
Consider the random matrix model $\Hb^{(\beta)}\in \CC^{nN\times nN} =\CC^{n \times n} \otimes \CC^{N \times N} $ of the form
\begin{equation}
  \label{eq:1}
  \Hb^{(\beta)}
  =
  K_0 \otimes \Ib_N + \sum_{\alpha = 1}^{d} \Big(L_{\alpha} \otimes 
  \Xb_{\alpha} + L_{\alpha}^{*} \otimes 
  \Xb_{\alpha}^{*} \Big)
  ,
\end{equation}
where $\beta\in \{1,2\}$ denotes the symmetry class with $\beta = 1$ corresponding to the real-symmetric matrices and $\beta=2$ to the (complex) Hermitian matrices. We distinguish these two classes in our assumptions as follows:
\begin{itemize}[itemsep=0pt]
\item for $\beta = 1$, the structure matrices $K_{0}, L_{1},\ldots, L_{d}\in \RR^{n\times n}$ are deterministic real $n\times n$ matrices, and $\Xb_1,\ldots, \Xb_d \in \RR^{N\times N}$ are independent $N\times N$ real random matrices with \emph{i.i.d.} entries satisfying
  \begin{equation}
    \label{eq:2}
    \EE\big[ x_{11}^{(\alpha)} \big]
    =
    0
    , \quad
    \EE\big[ \big(x_{11}^{(\alpha)}\big)^2 \big]
    =
    \frac{1}{N}
  \end{equation}
  for all $1\leq \alpha \leq d$;
\item for $\beta = 2$, the structure matrices $K_{0}, L_{1},\ldots, L_{d}\in \CC^{n\times n}$ are deterministic complex $n\times n$ matrices, and $\Xb_1,\ldots, \Xb_d \in \CC^{N \times N}$ are independent $N\times N$ complex random  matrices with \emph{i.i.d.} entries satisfying
  \begin{equation}
\label{eq:3}
    \EE\big[ x_{11}^{(\alpha)} \big]
    =
    0
    , \quad
    \EE\big[ \big|x_{11}^{(\alpha)}\big|^2 \big]
    =
    \frac{1}{N}
  \end{equation}
  for all $1\leq \alpha \leq d$.
\end{itemize}
The dimension $N\in \NN$ of the second tensor factor is the large parameter that tends to infinity, $\Ib_N$ is the $N\times N$ identity matrix, and $\otimes$ denotes the tensor (or Kronecker) product.

We additionally assume that the entries of $\Xb_{\alpha}$ have bounded moments: for each $p \in \NN$, $p\geq 3$, there exists $c_{p}>0$ such that
\begin{equation}
  \label{eq:4}
  \max_{1\leq \alpha \leq d}\EE\Big[ \big|\sqrt{N} x_{11}^{(\alpha)}\big|^p \Big]
  \leq
  c_{p}
  ,
\end{equation}
and in the complex Hermitian case ($\beta = 2$) we assume that $\Re x_{11}^{(\alpha)}$ and $\Im x_{11}^{(\alpha)}$ are independent and
\begin{equation}
  \label{eq:5}
  \EE\Big[ \big( \Re x_{11}^{(\alpha)}\big)^2 \Big]
  =
  \frac{1}{2N}
  ,
  \quad
  \EE\Big[ \big( \Im x_{11}^{(\alpha)}\big)^2 \Big]
  =
  \frac{1}{2N}
  .
\end{equation}
We call $d,n, K_{0}, L_{1},\ldots, L_{d}$ and $c_3,c_4,\ldots$ the \emph{model parameters}.

Notice, that if $L_{\alpha} = L_{\alpha}^{*}$ for some $\alpha\in \NN$, then this gives rise to the term of the form
\begin{equation*}
  L_{\alpha} \otimes 
  \Xb_{\alpha} + L_{\alpha}^{*} \otimes 
  \Xb_{\alpha}^{*}
  =
  \sqrt{2} L_{\alpha} \otimes 
  \bigg( \frac{\Xb_{\alpha} + \Xb_{\alpha}^{*}}{\sqrt{2}}\bigg),
\end{equation*}
where $(\Xb_{\alpha} + \Xb_{\alpha}^{*})/\sqrt{2}$ is a (real or complex) Wigner matrix.
Whenever the symmetry class is irrelevant, we will suppress the parameter $\beta$ in the notation. 

\textbf{Preliminary results about Kronecker random matrices.} 
The model $\Hb$ defined in \eqref{eq:1} is a special case of the  Kronecker random matrices,  introduced in \cite{AltErdoKrugNemi_Kronecker} to denote a  model of the type \eqref{eq:1} in which matrices $\{\Xb_{\alpha}\}$ are assumed to be independent with independent but not necessarily identically distributed entries.
Below we collect several properties of $\Hb$ that are direct consequences of the corresponding results for general Kronecker random matrices from \cite{AltErdoKrugNemi_Kronecker} and \cite{AltErdoKrug20}.

We start by introducing the \emph{matrix Dyson equation} (\emph{MDE}), which, among other, characterizes the large-$N$ limit of the empirical spectral measure of $\Hb$.
In the case of the Kronecker random matrix $\Hb$ defined in \eqref{eq:1}, the MDE takes the form
\begin{equation}
  \label{eq:7}
  -\frac{1}{M(z)}
  =
  z I_n - K_{0} + \Gamma[M(z)]
\end{equation}
with unknown $M(z)\in \CC^{n\times n}$, $z\in \CC_{+}:= \{z \in \CC \, : \, \Im z > 0\}$, $I_n\in \CC^{n\times n}$ the identity matrix, and the operator $\Gamma: \CC^{n\times n} \to \CC^{n\times n}$ given by
\begin{equation}
  \label{eq:8}
  \Gamma[R]
  :=
  \sum_{\alpha=1}^{d} \Big(L_{\alpha} R L_{\alpha}^{*} + L_{\alpha}^{*} R L_{\alpha} \Big)
\end{equation}
for any $R\in \CC^{n\times n}$.
We call $\Gamma$ the \emph{self-energy} operator.
Notice that the operator $\Gamma$ maps the set of positive semidefinite matrices into itself.
Therefore, by \cite[Theorem 2.1]{HeltRaFaSpei07}, for any $z\in \CC_{+}$ there exists a unique solution $M=M(z)$ to the matrix Dyson equation \eqref{eq:7} with positive definite imaginary part $ \Im M = \frac{1}{2\imu}(M-M^*)>~0$.
The function $M: \CC_{+} \to \CC^{n\times n}$ is a matrix-valued Herglotz function (see, e.g., \cite[Section~5]{GeszTsek00}), depends analytically on $z$, and admits the representation
\begin{equation}
  \label{eq:9}
  M(z)
  =
  \int_{\RR} \frac{V(dx)}{x-z}
  ,
\end{equation}
where $V(dx)$ is a (positive semidefinite) matrix-valued measure on $\RR$ with compact support.
The Stieltjes transform representation \eqref{eq:9} readily follows from \cite[Proposition 2.1]{AltErdoKrug20} by taking $\Ac = \CC^{n\times n}$ and the self-energy operator $\Gamma$. 
In the following we will often assume that $\Gamma$ satisfies the following $L$-\emph{flatness} property:
\begin{itemize}
\item[\textbf{(A)}] there exist $L\in \NN$, a matrix $Z = (z_{kl})_{k,l=1}^{n} \in \{0,1\}^{n\times n}$, and a constant $\cstA>0$ such that $z_{kk} = 1$ for $1 \leq k \leq n$, the matrix $Z^L$ has all entries strictly positive, and for any positive semidefinite matrix $R = (r_{kl})_{k,l=1}^{n}\in \CC^{n\times n}$
  \begin{equation}
    \label{eq:10}
    \Gamma\big[R \,\big]
    \geq
    \cstA \cdot \sum_{l=1}^{n}\Big(\sum_{k=1}^{n} z_{kl} r_{kk} \Big) E_{ll}
    .
  \end{equation} 
\end{itemize}
The special case of the $L$-flatness with $L=1$ (in the literature often simply called the \emph{flatness property}) requires that all entries of $Z$ are equal to $1$, in which case the relation \eqref{eq:10} takes the form
\begin{equation}
  \label{eq:11}
    \Gamma\big[R \,\big]
    \geq
    \cstA \Tr\big(R \big) I_n
    .
  \end{equation}
  The $L$-flatness property was first introduced  in the context of the matrix Dyson equation in an early arXiv version of  \cite{AjanErdoKrug19}.
We collect certain important properties of the MDE \eqref{eq:7} with self-energy satisfying \textbf{(A)} in Proposition~\ref{pr:1}.
If $\Gamma$ satisfies $\textbf{(A)}$, then it follows from part (i) of Proposition~\ref{pr:1} that $\|M(z)\|$ is uniformly bounded on $\CC_{+}$.
This implies (see, e.g., \cite[Lemma~5.4 (vi) and Lemma~5.5 (i)]{GeszTsek00}) that the measure $V(dx)$ in \eqref{eq:9} is absolutely continuous with respect to the Lebesgue measure on $\RR$ and its density is given by the inverse Stieltjes transform of $M(z)$, i.e., 
\begin{equation}
  \label{eq:12}
  V(dx)
  =
  V(x)dx
  ,\quad
  \mbox{where}
  \quad
  V(x)
  :=
  \lim_{y\downarrow 0} \frac{1}{\pi} \Im M(x+\imu y)
  .
\end{equation}
 From part (iv) of Proposition~\ref{pr:1} we know that $\lim_{y \downarrow 0} M(x + \imu y)$ exists for all $x\in \RR$.
Using this we extend the function $M$ beyond the set $\CC_{+}$ by setting
\begin{equation}
  \label{eq:13}
  M(z)
  :=
  \left\{
    \begin{array}{ll}
      \big(M(\overline{z})\big)^{*}, & z\in \CC_{-}:= \{z \in \CC \, : \, \Im z < 0\},
      \\[10pt]
      \lim_{y\downarrow 0} M(x + \imu y), & z = x \in \RR. 
    \end{array}
  \right.
\end{equation}
With the above definition, $M(z)$ is continuous on $\CC_{+}\cup \RR$, and $\lim_{y \downarrow 0}(M(x + \imu y)-M(x - \imu y)) = 2 \imu \Im M(x) = 2 \pi  \imu V(x)$.

We define the empirical spectral measure of $\Hb$ by
\begin{equation*}
  \mu_{\Hb}
  :=
  \frac{1}{nN} \sum_{i=1}^{nN} \delta_{\lambda_i}
  ,
\end{equation*}
where $\lambda_1, \ldots, \lambda_{nN}\in \RR$ are the eigenvalues of $\Hb$ counted with multiplicities, and $\delta_{\lambda}$ denotes the Dirac measure at $\lambda \in \RR$.
Assume that the self-energy operator satisfies the $L$-flatness property \textbf{(A)}.
By specializing \cite[Theorem~2.7]{AltErdoKrugNemi_Kronecker} to $\Hb$ and using \eqref{eq:12} we obtain that, as $N\to \infty$, the empirical spectral measure $\mu_{\Hb}$ converges weakly in probability to $\rho(x)dx$, where
\begin{equation}
  \label{eq:15}
  \rho(x)
  :=\frac{1}{n}
  \Tr\big(V(x)\big)
  =
  \lim_{y\to 0} \frac{1}{n\pi} \Tr \big(\Im M(x+\imu y)\big)
  .
\end{equation}
We call the weak convergence $\mu_{\Hb}\Rightarrow \rho(x)dx$ the \emph{global law} for $\Hb$, and we call the function $\rho(x)$ the \emph{(self-consistent) density of states} for $\Hb$.
Notice that $\rho$ depends only on the model parameters.

In this paper we determine the fluctuations of the linear spectral statistics of $\Hb^{(\beta)}$ for $\beta \in \{1,2\}$ on \emph{mesoscopic} scales inside the bulk, i.e., for those $x\in \RR$ for which $0< \rho(x) < \infty$.
The following central limit theorem for linear spectral statistics is our main result.
\begin{thm} \label{thm:main}
  Let $\Hb=\Hb^{(\beta)}$ be as in \eqref{eq:1}, and suppose that the corresponding self-energy operator $\Gamma$ satisfies the $L$-flatness assumption \textbf{(A)} for some $L\in \NN$.
  Let $g \in \Cc^{2}_c(\RR)$ be a twice continuously differentiable test function  with compact support.
  Then for any $\gamma \in (0,1)$ and $E_0$ satisfying $0< \rho(E_0) < \infty$, the mesoscopic linear spectral statistic 
  \begin{equation}
    \label{eq:16}
    \sum_{i=1}^{nN} \big( f_N(\lambda_i) - \EE[ f_N(\lambda_i)]  \big) 
  \end{equation}
  with
  \begin{equation}
    \label{eq:17}
    f_N: \RR \to \RR,
    \qquad
    f_N(x) = g\Big(N^{\gamma}(x - E_0)\Big)
  \end{equation}
  converges in distribution to a centered Gaussian random variable with variance
  \begin{equation}
    \label{eq:18}
    V[g]
    :=
    \frac{1}{2 \beta \pi^2} \int_{\RR} \int_{\RR} \frac{(g(x) - g(y))^{2}}{(x-y)^2} dx dy
    .
  \end{equation}
\end{thm}

\begin{rem}[Comparison of $1$-flatness and $L$-flatness for $L>1$]
  The condition \eqref{eq:11} that defines the $1$-flatness does not allow any of the $N\times N$ blocks of matrix $\Hb$ to be constantly equal to $0$.
  Indeed, if there exist $1\leq i,j \leq n$ such that $\la e_{i}, L_{\alpha} e_{j} \ra = \la e_{i}, L_{\alpha}^{*} e_{j} \ra = 0$ for all $1\leq \alpha \leq d$, then
  \begin{equation*}
    \big\la e_{i}, \Gamma[E_{jj} ] e_{i} \big\ra
    =
    \sum_{\alpha = 1}^{d}\Big(\big\la e_{i}, L_{\alpha} e_{j} \big\ra \big\la e_j, L_{\alpha}^{*}e_{i} \big\ra + \big\la e_{i}, L_{\alpha}^{*} e_{j} \big\ra \big\la e_j, L_{\alpha}e_{i} \big\ra \Big) 
    =
0
    ,
  \end{equation*}
  while $\Tr(E_{jj} ) = 1$.
  This contradicts to \eqref{eq:11}.
  On the other hand, the $L$-flatness with $L>1$ gives the structure of $\Hb$ more flexibility, in particular in terms of the zero blocks.
  We illustrate this with the following example.
  Let $\Xb_i$, $1 \leq i \leq 7$, be independent (real or complex) random i.i.d. matrices.
Denote $\Yb_{i} := \frac{1}{\sqrt{2}}(\Xb_{i} + \Xb_{i}^{*})$ for $1 \leq i \leq 4$, and consider the random Kronecker matrix
\begin{equation*}
  \Hb
  =
  \left(
    \begin{array}{cccc}
      \Yb_1 & \Yb_1+\Xb_5 & 0 & 0
      \\
      \Yb_1+ \Xb_5^* & \Yb_2 & 0 & \Xb_6
      \\
      0 & 0 & \Yb_1 + \Yb_3 & \Xb_6 + \Xb_7
      \\
      0 & \Xb_6^* & \Xb_6^* + \Xb_7^* & \Yb_4
    \end{array}
  \right)
\end{equation*}
constructed using the structure matrices
\begin{align}
  \label{eq:20}
  L_{1}
  =
  \frac{1}{\sqrt{2}}\Big( E_{11} + E_{12} + E_{21} + E_{33} \Big)
  , \quad
  L_{2}
  =
  \frac{1}{\sqrt{2}} E_{22}
  ,\quad
  L_{3}
  =
  \frac{1}{\sqrt{2}} E_{33}
  ,
\end{align}
\begin{equation}
  \label{eq:21}
  L_{4} = \frac{1}{\sqrt{2}} E_{44}
  , \quad
  L_{5}
  =
  E_{12}
  ,\quad
  L_{6}
  =
  E_{24} + E_{34}
  ,\quad
  L_{7}
  =
  E_{34}
  .
\end{equation}
The matrix $Z \in \{0, 1\}^{4 \times 4}$ given by
\begin{equation}
  \label{eq:22}
  Z
  =
  \left(
    \begin{array}{cccc}
      1 & 1 & 0 & 0
      \\
      1 & 1 & 0 & 1
      \\
      0 & 0 & 1 & 1
      \\
      0 & 1 & 1 & 1
    \end{array}
  \right)
\end{equation}
had a strictly positive main diagonal, and $Z^3$ has all entries strictly positive.
Moreover, the operator $\Gamma$ defined through \eqref{eq:8} with $d=7$ and $L_\alpha$ in \eqref{eq:20}-\eqref{eq:21} satisfies the $3$-flatness property with $Z$ given in \eqref{eq:22} and $\cstA = 0.1$.
The possibility of having zero blocks plays important role in applying the random LMPs of the form \eqref{eq:1} to the study of the linearizations of polynomials and rational functions in random matrices.
\end{rem}

\vspace{.5cm}

\textbf{Structure of the proof}.
In Section~\ref{sec:proof} we provide the full proof of Theorem~\ref{thm:main} in the complex Hermitian case $\beta = 2$.
In Sections~\ref{sec:proof-1} and \ref{sec:proof-2} we derive a differential equation for the characteristic function of the mesoscopic linear spectral statistic \eqref{eq:16}.
The obtained equation contains a multiresolvent term.
In Section~\ref{sec:proof-3} we show that the multiresolvent term satisfies a certain self-consistent equation, which is analyzed in Section~\ref{sec:proof-4}.
This allows to approximate the multiresolvent term by a deterministic quantity and to compute the limiting behavior of the characteristic function of the statistic \eqref{eq:16} in Section~\ref{sec:proof-5}.

In Section~\ref{sec:proof-real} we prove Theorem~\ref{thm:main} in the real symmetric case $\beta=1$.
Section~\ref{sec:proof-r1} collects all the results that can be imported from Section~\ref{sec:proof} with minor changes or without changes.
The differential equation for the characteristic function of the linear statistic \eqref{eq:16} in the $\beta = 1$ case contains a multiresolvent term of a new type, which we then study in Section~\ref{sec:proof-r2}.
The limiting behavior of the characteristic function is calculated in Section~\ref{sec:proof-r3}, thus completing the proof of Theorem~\ref{thm:main}.

To streamline the presentation, the derivation of certain results, which are important, but whose proofs are not immediately relevant for establishing the main theorem, are postponed to the appendix.
In Appendix~\ref{sec:lflat} we list and prove the properties of the solution the the MDE \eqref{eq:7} with self-energy satisfying the general $L$-flatness property \textbf{(A)}.
Finally, in Appendix~\ref{sec:iid} we derive a convenient cumulant expansion formula for the resolvent matrix $(\Hb - z\Ib_{nN})^{-1}$, which is used extensively in Sections~\ref{sec:proof-2} and \ref{sec:proof-r2}.

\section{Complex Hermitian case}
\label{sec:proof}

The proof of Theorem~\ref{thm:main} presented in this section relies on the study of the characteristic function of the mesoscopic linear spectral statistic $\sum_{i} \big(f_N(\lambda_i) - \EE[f_N(\lambda_i)] \big) = \Tr f_N(\Hb) - \EE\big[\Tr f_N(\Hb)\big]$.
More precisely, we show that the characteristic function  $\EE\big[\exp\big\{\imu t \big(\Tr f_N(\Hb) - \EE\big[\Tr f_N(\Hb)\big]\big)\big\}\big]$ converges to $\exp\{-t^2V[g]/2\}$, the characteristic function of a Gaussian random variable with variance \eqref{eq:18}.
This is achieved through the detailed analysis of the resolvent of $\Hb$, which can be related to the linear spectral statistics \eqref{eq:16} via the Helffer-Sj\"{o}strand formula.
This relation will be explained below in \eqref{eq:28}.
We start by introducing the necessary notation.

\subsection{Notation and preliminary reductions}
\label{sec:proof-1}

Denote by $\eta_0(N) := N^{-\gamma}$ the (mesoscopic) scaling parameter.
We assume $\gamma \in (0,1)$, so that $N^{-1} \ll \eta_0(N) \ll 1$ and $f_N(x) = g\big((x-E_0)/\eta_0\big)$ (defined in \eqref{eq:17}) satisfies
\begin{equation}
  \label{eq:23}
  \|f_N\|_{1}
  =
  \|g\|_{1} \eta_0
  ,\quad
  \| f_N' \|_{1}
  =
  \|g'\|_{1}
  , \quad
  \| f_N''\|_{1}
  =
  \frac{\|g''\|_{1}}{\eta_0}
  .
\end{equation}
We will suppress the $N$-dependence in $\eta_0$ and $f$ for brevity.

Since the mesoscopic test function $f$ is localized around $E_0$, we can restrict our analysis to a small region around the support of $f$.
To this end, let $\sigma>0$ be a sufficiently large constant satisfying $\mathrm{supp}(g) \subset [-\sigma,\sigma]$, which, in turn, implies
\begin{equation*}
  \mathrm{supp}(f) \subset \big[E_{0}-\delta,E_{0}+\delta\big]
\end{equation*}
with $\delta = \delta_N:=\sigma \eta_{0}$.
From the continuity of $\rho$ and the fact that $0<\rho(E_0)<\infty$, there exists $\theta \in (0,1)$ such that
\begin{equation}
  \label{eq:25}
  \theta \leq \rho(x) \leq \theta^{-1}
  \quad
  \mbox{for all}
  \quad
  x \in (E_0-2 \delta,E_0+2\delta)
\end{equation}
for $N$ sufficiently large.
In the sequel, we will always assume that $N$ is large enough for \eqref{eq:25} to hold.

Denote by $\chi : \RR \to [0,1]$ a smooth cutoff function supported on $[-2\delta,2 \delta]$ and equal to $1$ on $[-\delta, \delta]$, and define the almost analytic extension of $f$ by
\begin{equation}
  \label{eq:26}
  \ft(z)
  :=
  (f(x) + \imu yf'(x))\chi(y)
\end{equation}
with $z=x+iy$.
An integral representation formula, used in the Helffer-Sj\"{o}strand functional calculus, expresses the value $f(\lambda)$ for any $\lambda \in \RR$ as
\begin{equation*}
  f(\lambda)
  =
  \frac{1}{\pi}\int_{\CC} \frac{\frac{\partial}{\partial\overline{z}} \ft (z)}{\lambda - z} \dz
  =
  \frac{1}{2 \pi}\int_{\CC} \frac{\imu y f''(x) \chi(y) +\imu(f(x) + \imu y f'(x))\chi'(y) }{\lambda - x - \imu y} dx dy
\end{equation*}
with the standard Lebesgue measure on $\CC$ denoted by  $\dz := d \Re z \, d \Im z$.
Therefore, we rewrite the fluctuations of the linear spectral statistics (\ref{eq:16}) as
\begin{equation}
  \label{eq:28}
  (1-\EE) \big[\Tr f(\Hb)\big]
  =
  \frac{1}{\pi} \int_{\CC}\frac{\partial \ft (z)}{\partial \overline{z}}  (1-\EE)\big[ \Tr \Gb(z)\big] \dz
  ,
\end{equation}
where
$\Gb(z) := (\Hb - z  \Ib_{nN} )^{-1}\in \CC^{nN\times nN}$ is the resolvent of $\Hb$, $\Ib_{nN}$ is the $nN\times nN$ identity matrix  and we used the shorthand notation  $(1-\EE) [X] := X - \EE[X]$ for any random variable $X$.
The above representation, with resolvent $\Gb(z)$ appearing on the right-hand side of \eqref{eq:28}, allows to exploit the properties of the resolvents of Kronecker random matrices, in particular, the local law.

To state the local law for Kronecker random matrices, we use the following notation.
For any matrix $\Rb \in \CC^{\,nN \times nN}$ we denote its (left) matrix coefficients $\{R_{ij}\}_{1\leq i,j \leq N}$ with respect to the standard basis of $\CC^{N\times N}$ via the identity
\begin{equation*}
  \Rb = \sum_{i,j=1}^{N} R_{ij} \otimes \Eb_{ij}
  ,\quad
  R_{ij} \in \CC^{n\times n}
  ,
\end{equation*}
where $\Eb_{ij} = \big(\delta_{k i}\delta_{j l }\big)_{k,l = 1}^{N} \in \CC^{N\times N}$, $\delta_{kl}$ is the Kronecker delta, and $\{\Eb_{ij}\}_{1 \leq i, j \leq N}$ is the standard basis of $\CC^{N\times N}$.
For example, the collection $\{G_{ij}(z)\}_{1\leq i,j \leq N}$ gives the matrix coefficients of the resolvent  $\Gb(z) \in \CC^{\,nN \times nN}$ through the identity
\begin{equation}
  \label{eq:30}
  \Gb(z) = \sum_{i,j=1}^{N} G_{ij}(z) \otimes \Eb_{ij}
  ,\quad
  G_{ij}(z) \in \CC^{n\times n}
  .
\end{equation}
We  also need the notion of   \emph{stochastic domination}.
For two sequences of nonnegative random variables $\bm{\Phi}:=(\Phi_{N})_{N\in \NN}$ and $\bm{\Psi}:=(\Psi_{N})_{N\in \NN}$ we say that $\bm{\Phi}$ is \emph{stochastically dominated} by $\bm{\Psi}$, denoted $\bm{\Phi} \prec \bm{\Psi}$, if for any  $\varepsilon > 0$ small and $D\in \NN$ there exists $C_{\varepsilon, D}>0$ such that
\begin{equation}
  \label{eq:31}
  \PP\Big[ \Phi_N \geq N^{\varepsilon} \Psi_N \Big]
  \leq
  \frac{C_{\varepsilon, D}}{N^{D}}
\end{equation}
holds for all $N\in \NN$. 
If $\bm{\Phi}$ and $\bm{\Psi}$ are deterministic, $\bm{\Phi}\prec \bm{\Psi}$ means that for any $\varepsilon>0$ there exists $C_{\varepsilon}>0$ such that $\bm{\Phi} \leq C_{\varepsilon}N^{\varepsilon} \bm{\Psi}$.
Finally, if there exists $C>0$ such that $\bm{\Phi} \leq C \bm{\Psi}$ uniformly for all $N$, then we denote this as  $\bm{\Phi} \lesssim \bm{\Psi}$.
We  write $\bm{\Phi} \sim \bm{\Psi}$ if $\bm{\Phi} \lesssim \bm{\Psi}$ and $\bm{\Psi} \lesssim \bm{\Phi}$.

We now state a result that will be used repeatedly throughout this paper to analyze the right-hand side of \eqref{eq:28}.
\begin{pr}[Optimal bulk local law for $\Hb$]
  Suppose that the self-energy operator $\Gamma$ satisfies the $L$-flatness property \textbf{(A)}.
  Then for any $\nu\in (0,1)$ and $\theta\in (0,1)$ 
\begin{equation}
  \label{eq:32}
  \max_{1\leq i,j \leq N}\|G_{ij}(z) - \delta_{ij}M(z)\|
  \prec
  \frac{1}{\sqrt{N \Im z}}
\end{equation}
and
\begin{equation}
  \label{eq:33}
  \Big\|\frac{1}{N}\sum_{i=1}^{N}G_{ii}(z) - M(z)\Big\|
  \prec
  \frac{1}{N \Im z}
\end{equation}
uniformly on the set $\{z \, : \, \theta < \rho\,(\Re z) < \theta^{-1}, \,\Im z \geq N^{-1+\nu}\}$, where $M(z)$ is the solution to the MDE \eqref{eq:7}, $\rho$ is the corresponding self-consistent density of states defined in \eqref{eq:15}, and $\| \cdot \|$ denotes the matrix norm of $G_{ij}(z) \in \CC^{n\times n}$ induced by the Euclidean norm on $\CC^n$.
\end{pr}
\begin{proof}
  The bounds \eqref{eq:32} and \eqref{eq:33} are an immediate consequence of \cite[Lemma~B.1]{AltErdoKrugNemi_Kronecker} and Proposition~\ref{pr:1}.
Indeed, it has been proven in \cite[Lemma~B.1]{AltErdoKrugNemi_Kronecker} that \eqref{eq:32} and \eqref{eq:33} hold for all $z\in \CC_{+}$ such that $\Im z \geq N^{-1 + \nu}$ and
\begin{itemize}[itemsep=0pt]
\item[(i)] $\|M(z)\|$ is bounded, and
\item[(ii)] the \emph{stability operator} defined for any $R\in \CC^{n\times n}$ by
  \begin{equation}
    \label{eq:34}
    R \mapsto  R-M(z) \Gamma[R] M(z)
  \end{equation}
  is invertible.
\end{itemize}
Parts (i) and (iii) of Proposition~\ref{pr:1} establish the boundedness of $\|M(z)\|$ (as well as $\|(M(z))^{-1}\|$) uniformly on $\CC_{+}$, and the boundedness of the inverse of the stability operator \eqref{eq:34} uniformly on the set $\{z \, : \, \theta < \rho\,(\Re z) < \theta^{-1}, \, \Im z \geq  0 \}$, which together imply~\eqref{eq:32} and \eqref{eq:33}. 
\end{proof}
Notice that the constants $C_{\varepsilon, D}$ in \eqref{eq:32} and \eqref{eq:33} that are hidden in the notion of stochastic domination depend on the model parameters and  additionally on $\nu$, $\theta$.
We call \eqref{eq:32} the \emph{entry-wise local law} and \eqref{eq:33} the \emph{averaged local law} for $\Hb$.
Local law bounds \eqref{eq:32} and \eqref{eq:33} provide us with the necessary  control of the resolvent $\Gb(z)$ for the spectral parameters $z$ very close to the real line, namely for  $|\Im z| \geq N^{-1+\nu}$ for arbitrarily small $\nu \in (0,1)$. 

In order to establish the weak convergence of the random variable \eqref{eq:16} to a centered Gaussian random variable with variance \eqref{eq:18}, we consider the characteristic function of $(1-\EE)[\Tr f(\Hb) ]$.
Define 
\begin{equation}
  \label{eq:35}
  e(t)
  :=
  e^{\imu t (1-\EE)[\Tr f(\Hb)]}
  =
  e^{\frac{\imu t}{\pi} \int_{\CC} \frac{\partial \ft (z)}{\partial \overline{z}} (1-\EE)[\Tr \Gb(z) ] \dz}
  ,
\end{equation}
where on the right-hand side we used the representation \eqref{eq:28}.
As the first step of the analysis of \eqref{eq:35} we show that it is enough to consider the integral over a small neighborhood around $E_0$.
Moreover, removing a sufficiently narrow strip around the real line in the integral in \eqref{eq:35} does not change the limiting value of $e(t)$ and its expectation.
More precisely, denote
\begin{equation*}
  \Omega
  =
  \Omega_N
  :=
  \{z \in \CC \, : \, |\Re z - E_0|<\delta, N^{-\tau}\eta_0<|\Im z |< 2\delta \}   ,
\end{equation*}
where $\tau \in \big(0,1 \big)$ is a sufficiently small constant.
For each statement in this and subsequent sections we will indicate how small $\tau$ should be compared to $\gamma$.
The following holds.
\begin{lem}
  Let $\gamma \in \big(0,1\big)$,  $\tau \in \big(0, (1-\gamma)\big)$ and 
  \begin{equation}
    \label{eq:37}
    \efr(t)
    : =
    e^{\frac{\imu t}{\pi}\int_{\Omega} \frac{\partial \ft (z)}{\partial \overline{z}}  (1-\EE)[\Tr \Gb(z) ] \dz}
    .
  \end{equation}
  Then
  \label{lem:1}
  \begin{equation*}
    \big|\EE[e(t)] - \EE[\efr(t)]\big|
    \prec
    |t| \|g''\|_{1} N^{-\tau} 
    .
  \end{equation*}
\end{lem}
\noindent The proof of this lemma follows the lines of the similar result from \cite[Section~4.2]{LandSoso20} with only minor changes, therefore, we omit it in the present work.
 From Lemma~\ref{lem:1} we see that $\EE[\efr(t)]$ and the characteristic function of the mesoscopic linear spectral statistic $\EE[e(t)]$ coincide in the limit $N\to \infty$.
In the remainder of this section we will study $\EE[\efr(t)]$.
Notice that the local laws \eqref{eq:32}-\eqref{eq:33} hold for $|\Im z| \gg N^{-1}$, therefore, working with $\efr(t)$ makes it possible to apply the local laws for $\Gb(z)$ for all $z$ in the domain of integration $\Omega$ in \eqref{eq:37}.

\subsection{Computing $\EE[\efr(t)]$  for $\beta = 2$}
\label{sec:proof-2}
In this section we obtain an approximate equation for $\EE[\efr(t)]$ in the complex Hermitian case.
The derivation is based on direct algebraic computations, and the obtained equation will be further analyzed and refined in the subsequent sections.

The main tool that will be used throughout this section is the cumulant expansion formula.
Denote by  $\Wb = \Hb - \EE[\Hb]$  the fluctuation matrix of $\Hb$ from \eqref{eq:1} with complex i.i.d. matrices $\Xb_{\alpha}$ in the second tensor factor, i.e. 
\begin{equation}
  \label{eq:39}
  \Wb
  :=
  \sum_{\alpha = 1}^{d} \bigg( L_{\alpha} \otimes 
  \Xb_{\alpha} + L_{\alpha}^{*} \otimes 
  \Xb_{\alpha}^{*} \bigg)
  .
\end{equation}
For any differentiable function $\Fcc: \CC^{nN\times nN} \to \CC^{nN \times nN}$ we define the directional derivative of $\Fcc(\Wb)$ with respect to $\Wb= (w_{ij})_{i,j=1}^{nN}$ in the direction $\Rb = (r_{ij})_{i,j=1}^{nN} \in \CC^{nN \times nN}$ as
\begin{equation}
  \label{eq:40}
  \nabla_{\Rb}\Fcc(\Wb)
  :=
  \sum_{i,j=1}^{nN} \frac{\partial \Fcc(\Wb) }{\partial \,w_{ij}}  \, r_{ij}
  .
\end{equation}
We also define the partial trace operator $\Id_n \otimes \Tr_{N}: \CC^{nN \times nN} \to \CC^{n\times n}$ acting as $(\Id_n \otimes \Tr_N)\big[A\otimes \Bb\big] = A \Tr\, (\Bb)$ for any $A\otimes \Bb \in \CC^{nN\times nN} = \CC^{n\times n} \otimes \CC^{N\times N}$.
Then the following holds. 
\begin{lem}
  \label{lem:b1}
  Let $\tau \in \big(0,\min\{\gamma,(1-\gamma)\}\big)$.
  Denote $\Fcc_1(\Wb):= \Gb(z)$ and $\Fcc_2(\Wb) := \Gb(z)\,\efr(t)$.
  Then for $\star \in \{1,2\}$ we have
  \begin{equation}
    \label{eq:41}
      \EE \Big[ \Wb \Fcc_{\star}(\Wb)\Big]
      =
      \EE \Big[ \Wbt \, \nabla_{\Wbt} \Fcc_{\star}(\Wb) \Big]
      +
      \errJ_{\star}(z)
      ,
    \end{equation}
     where $\widetilde{\Wb}$ is an independent copy of $\Wb$ and the error terms  $\errJ_{\star}(z)$ are analytic in $z$ and satisfy the bounds
    \begin{equation}
      \label{eq:42} 
      \Big\|\errJ_{\star}(z) \Big\|_{\max}
      =
      O_{\prec}\bigg(\big(1 + |t|^3\big)\frac{N^{5\tau/2}}{N\sqrt{\eta_{0}}} \bigg)
    \end{equation}
    uniformly for $z \in \Omega$.
  \end{lem}
The proof of Lemma~\ref{lem:b1} relies on the cumulant expansion formula for real random variables (see, e.g., \cite{Spee83}) and is presented in Appendix~\ref{sec:iid}.
Below we will see that applying formula \eqref{eq:41} gives rise to the operator $\Scc : \CC^{nN\times nN} \to \CC^{nN\times nN}$ acting as
\begin{equation} 
  \label{eq:43}
  \Scc[\Rb]
  :=
  \EE\big[ \Wb \Rb \, \Wb \big]
\end{equation}
for any matrix $\Rb \in \CC^{nN\times nN}$.
The operator $\Scc$ can be written in terms of the self-energy operator $\Gamma$ and the partial trace operator $\Id_n \otimes \Tr_{N}$.
Indeed, for any  $\Rb = \sum_{i,j=1}^{N} R_{ij}\otimes \Eb_{ij}$ with $R_{ij}\in \CC^{n\times n}$ we have that
\begin{align}
  \label{eq:44}
  \Scc[\Rb]
  =
  \EE\big[\Wb  \Rb \, \Wb \big]
  =
  \Gamma \Big[ \frac{1}{N}\sum_{j=1}^{N} R_{j j} \Big] \otimes  \Ib_N 
  =
  \frac{1}{N} \Gamma \Big[ (\Id_n \otimes \Tr_N)\big[\Rb\big] \Big] \otimes \Ib_N
  ,
\end{align}
where
\begin{equation*}
  (\Id_n \otimes \Tr_N) \big[\Rb\big]
  =
  \sum_{i,j=1}^{N} (\Id_n \otimes \Tr_N) \big[R_{ij}\otimes \Eb_{ij}\big]
  =
  \sum_{j=1}^{N} R_{j j}
  .
\end{equation*}
It will be often convenient to decompose taking the trace on $\CC^{nN\times nN}$ into two steps,  first taking the partial trace $(\Id_n \otimes \Tr_N)$, and then applying the trace on the smaller space $\CC^{n\times n}$.

Denote $\Mb(z) := M(z) \otimes \Ib_N \in \CC^{nN\times nN}$.
Then $\Scc[\Mb(z)] = \Gamma[M(z)]\otimes \Ib_{N}$, and thus for any $N\in \NN$ the function $\Mb(z)$ satisfies the equation
\begin{equation}
  \label{eq:46}
  -\frac{1}{\Mb(z)}
  =
  z \Ib_{nN} - \Kb_{0} + \Scc\big[ \Mb(z)\big]
\end{equation}
with $\Kb_0:=K_0\otimes \Ib_N$.
Equation \eqref{eq:46} is a Dyson equation with a positivity preserving self-energy operator \eqref{eq:44} that admits  the solution $\Mb(z) = M(z) \otimes \Ib_N$ for any $z\in \CC\setminus \RR$ and $N\in \NN$.

For any $r \in \NN$ and $T \in \CC^{r\times r}$, denote by $\Cc_{T}$ the operator acting on $\CC^{\,r\times r}$ as the multiplication by $T$ from the left and from the right
\begin{equation}
  \label{eq:47}
  \Cc_{T}\big[R\,\big]
  =
  T R \,T
  .
\end{equation}
If $T$ is invertible, then so is $\Cc_{T}$, and $\Cc_{T}^{-1} = \Cc_{T^{-1}}$.
For $z\in\CC\setminus \RR$ and $N \in \NN$, we define the operators  
\begin{equation}
  \label{eq:48}
  \Bcc_z:=\Cc_{\Mb(z)}^{-1} - \Scc
  ,\qquad
  \Bc_{z}:=\Cc_{M(z)}^{-1} - \Gamma
\end{equation}
acting on $ \CC^{nN\times nN }$ and  $\CC^{n\times n }$ correspondingly.
Notice that the composition $\Cc_{M(z)}\Bc_z$ gives the stability operator for the Dyson equation \eqref{eq:7} introduced in \eqref{eq:34}.
Similarly, the composition $\Cc_{\Mb(z)}\Bcc_z = \Id - \Cc_{\Mb(z)} \Scc$ is the stability operator for the Dyson equation \eqref{eq:46}.
As mentioned in Section~\ref{sec:proof-1}, the inverse of the stability operator \eqref{eq:34} is uniformly bounded in the operator norm on the set $\{z \, : \, \theta < \rho\,(\Re z) < \theta^{-1}, \, \Im z \geq 0 \}$ for any $\theta \in (0,1)$.
This, together with the boundedness of $\|M(z)\|$ and $\|M^{-1}(z)\|$ established in part (i) of Proposition~\ref{pr:1} and the extension of $M(z)$ to $\CC$ defined in \eqref{eq:13}, implies that for any $\theta,\nu \in (0,1)$ and any $C>0$
\begin{equation}
  \label{eq:49}
  \sup \{\|\Bc_{z}^{-1}\| \, : \, \theta < \rho\,(\Re z) < \theta^{-1},  \, |z| \leq C \}
  \lesssim
  1
  .
\end{equation}
From \eqref{eq:44} and \eqref{eq:48} we have that for any $\Rb \in \CC^{nN \times nN}$
\begin{equation}
  \label{eq:50}
  (\Id_n \otimes \Tr_N) \, \Bcc_z \big[\Rb\big]
  =
  \Bc_z \, (\Id_n \otimes \Tr_N)\big[\Rb\big]
  .
\end{equation}

Now we define the adjoints of the operators $\Bcc_{z}$ and $\Bc_{z}$ (introduced in \eqref{eq:34}).
To this end, for any $r\in \NN$ we equip $\CC^{r \times r}$ with the scalar product 
\begin{equation}
  \label{eq:51}
  \la S,T \ra
  =
  \frac{1}{r} \Tr \big(S^* T \big)
\end{equation}
for all $S,T \in \CC^{r\times r}$.
We denote by $\Bc^*$ and $\Bcc^*$ the adjoints of $\Bc$ and $\Bcc$ with respect to the corresponding scalar products \eqref{eq:51}.
We also introduce the notation for the normalized trace functional $\la T \ra := \la I, T \ra = \frac{1}{r} \Tr \big(T \big)$.

We proceed to the analysis of $\EE[\efr(t)]$.
\begin{lem}
  \label{lem:2}
  Let $\gamma \in (0,1)$ and $\tau \in \big(0, \min\{\gamma,(1 - \gamma)\}/ 7 \big)$. Then
  \begin{equation}
    \label{eq:52}
    \frac{d}{dt} \EE[\efr(t)]
    =
    - \frac{t}{\pi^2} \int_{\Omega \times \Omega}  \frac{\partial \ft(z)}{\partial \overline{z}}  \frac{\partial \ft(\zeta)}{\partial \overline{\zeta}} \frac{\partial}{\partial \zeta} \EE\Big[  \sum_{i,j=1}^{n}  \Tr \Big(E_{ji} \,S_{ij}(z,\zeta) \Big)  \efr(t) \Big] \dzeta  \dz + \errK
  \end{equation}
  where $|\errK|\prec N^{-\tau}$ and 
  \begin{equation*}
    S_{ij}(z,\zeta)
    :=
    \frac{1}{N} \sum_{k,l = 1}^{N}  G_{l k}(z)  \frac{1}{M(z)} \Bc_{z}^{-1}[I_n ] \, \Gamma [E_{ij} ]   G_{k l}(\zeta)
    .
  \end{equation*}
\end{lem}
\begin{proof}
  We start by differentiating the function $\efr(t)$ in \eqref{eq:37} with respect to $t$
  \begin{equation}
    \label{eq:54}
    \frac{d}{dt} \EE[\efr(t)]
    =
    \frac{\imu}{\pi}\int_{\Omega} \frac{\partial \ft(z)}{\partial \overline{z}}  \EE \big[ \efr(t) (1-\EE)[\Tr \Gb(z)] \big] \dz
    .
  \end{equation}
  Recall that $\Gb(z) = (\Wb + \Kb_0 - zI_{nN})^{-1}$ with $\Kb_0 = K_0 \otimes \Ib_N$ and $\Wb$ defined in \eqref{eq:40}, which implies the following trivial identities
  \begin{equation}
    \label{eq:55}
    (z \Ib_{nN} - \Kb_0) (1-\EE)[ \Gb(z)]
    =
    (1-\EE)\big[(z \Ib_{nN} - \Kb_0) \Gb(z)\big]
    =
    (1-\EE)\big[\Wb \Gb(z)\big]
    .
  \end{equation}
  The resolvent matrix $\Gb(z)$ is an analytic function of $\Wb$ satisfying $\|\Gb(z)\| \leq  |\Im z|^{-1} $ for all $z\in \CC\setminus \RR$.
  Moreover, for any $\Rb \in \CC^{nN\times nN}$
  \begin{equation}
    \label{eq:56}
    \nabla_{\Rb}\,\Gb(z)
    =
    -\Gb(z) \Rb \,\Gb(z)
    .
  \end{equation}
  We recall that $\nabla_{\Rb} = \sum_{i,j}r_{ij} \partial_{ij}$ denotes the directional derivative with respect to $\Wb$ in the direction $\Rb = (r_{ij})$.
  Therefore, by applying the cumulant expansion formula \eqref{eq:41} with $\star = 1$
  and taking separately the partial expectation with respect to $\widetilde{\Wb}$ we get
  \begin{equation}
    \label{eq:57}
    \EE[\Wb \Gb(z)]
    =
    -\EE[\widetilde{\Wb} \Gb(z) \widetilde{\Wb} \Gb(z)] + \errJ_{1}(z)
    =
    -\EE\big[\Scc[\Gb(z)]\,\Gb(z)\big] + \errJ_{1}(z)
    ,
  \end{equation}
  where the operator $\Scc : \CC^{nN\times nN} \to \CC^{nN\times nN}$ was defined in \eqref{eq:43}.
  Similarly, using \eqref{eq:41} for $\star = 2$ and \eqref{eq:56} we have that
  \begin{align}
    \label{eq:59}
    \EE\Big[ \Wb \Gb(z)\efr(t)\Big]
    &=
      \EE\Big[ \widetilde{\Wb} \nabla_{\widetilde{\Wb}} \big( \Gb(z)\efr(t)\big)\Big] + \errJ_{2}(z) 
      \\ \nonumber
    &  =
      \EE\Big[ \widetilde{\Wb} \nabla_{\widetilde{\Wb}} \big( \Gb(z) \big) \efr(t)\Big] + \EE\Big[ \widetilde{\Wb}  \Gb(z) \nabla_{\widetilde{\Wb}} \big(\efr(t)\big)\Big] + \errJ_{2}(z)
    \\ 
    \nonumber
    &     =
      -\EE\Big[\Scc[\Gb(z)]\Gb(z) \efr(t)\Big] + \EE\Big[ \widetilde{\Wb}  \Gb(z) \nabla_{\widetilde{\Wb}} \big(\efr(t)\big)\Big] + \errJ_{2}(z)
      .
  \end{align}
  Combining \eqref{eq:57} and \eqref{eq:59} yields
  \begin{equation}
    \label{eq:60}
    \EE\Big[\efr(t) \big(1-\EE\big)\big[\Wb \Gb\big]\Big]
    =
    -\EE\Big[\efr(t)\big(1-\EE\big)\big[\Scc[\Gb]\,\Gb\big] \Big] + \EE\Big[ \widetilde{\Wb}  \Gb\, \nabla_{\widetilde{\Wb}} \big(\efr(t)\big)\Big]
    + \errC
    ,
  \end{equation}
   where we suppressed the argument $z$ for brevity, and defined the error matrix  
   \begin{equation} \label{eq:def errJ2}
   \errC(z) := \errJ_{2}(z) - \EE[\efr(t)] \errJ_{1}(z)
   \end{equation}
   with $\errJ_{1},\errJ_{2}$ from \eqref{eq:41}. 
   We rewrite the first term on the right-hand side in \eqref{eq:60} as
  \begin{align}
    \label{eq:62}
    &-\EE\Big[\efr(t)\big(1-\EE\big)\big[\Scc[\Gb]\,\Gb\big] \Big]
      \\ \nonumber
    & 
      =
      -\EE\Big[\efr(t)\big(1-\EE\big)\big[\Scc[\Mb]\,\Gb\big] \Big]
     -\EE\Big[\efr(t)\Scc \big[\big(1-\EE\big)[\Gb]\big]\Mb \Big]
      + \errH 
    \\ 
    \nonumber
    & 
      =
      \EE\Big[\efr(t)\Big(\frac{1}{\Mb} + z\Ib_{nN} - \Kb_{0}\Big) \big(1-\EE\big)\big[\Gb\big] \Big]    
      -\EE\Big[\efr(t)\big(1-\EE\big)\big[\Scc[\Gb]\Mb\big] \Big]
      + \errH 
      ,
  \end{align}
  where we used that $\Mb$ satisfies \eqref{eq:46}, and introduced the error term
  \begin{equation}
    \label{eq:63}
    \errH(z)
    :=
    -\EE\Big[\efr(t)\big(1-\EE\big)\big[\Scc[\Mb(z)-\Gb(z)]\big( \Mb(z) -\Gb(z) \big)\big] \Big]
    .
  \end{equation} 
  The identities \eqref{eq:60} and \eqref{eq:62} substituted into \eqref{eq:55} after an elementary cancellation give
  \begin{equation*}
    \frac{1}{\Mb}\EE\Big[\efr(t)\big(1-\EE\big)\big[\Gb\big] \Big]
    =
    \EE\Big[\efr(t)\big(1-\EE\big)\big[\Scc[\Gb]\Mb\big] \Big]
    -\EE\Big[ \widetilde{\Wb}  \Gb \, \nabla_{\widetilde{\Wb}} \big(\efr(t)\big)\Big]
    - \errC  - \errH 
    .
  \end{equation*} 
  Multiplying the above equation by $\Mb^{-1}(z)$ from the right and rearranging the terms yields
  \begin{equation}
    \label{eq:65}
    \Bcc_z \, \EE\Big[\efr(t)\big(1-\EE\big)\big[\Gb\big] \Big]
    =
    -\EE\Big[ \widetilde{\Wb}  \Gb \, \nabla_{\widetilde{\Wb}} \big(\efr(t)\big)\Big] \frac{1}{\Mb }
    - \Big(\errC + \errH  \Big) \frac{1}{\Mb }
    ,
  \end{equation}
  where $\Bcc_z$ was defined in \eqref{eq:48}. 
  By applying $(\Id_n \otimes \Tr_N)$ on both sides of \eqref{eq:65} and using \eqref{eq:50} we get
  \begin{equation}
    \label{eq:66}
    (\Id_n\otimes \Tr_N)\EE\Big[\efr(t)\big(1-\EE\big)\Gb \Big]
    =
    -\Bc_{z}^{-1} (\Id_n\otimes \Tr_N) \Big[\EE\big[ \widetilde{\Wb}  \Gb \,\nabla_{\widetilde{\Wb}} \, \big(\efr(t)\big)\big]  \frac{1}{\Mb } \Big] + \errL 
    ,
  \end{equation}
  where
  \begin{equation}
    \label{eq:67}
    \errL (z)
    :=
    -\Bc_{z}^{-1} (\Id_n\otimes \Tr_N) \Big[ \Big( \errC(z) + \errH (z) \Big) \frac{1}{\Mb(z) } \Big]
    .
  \end{equation}
  Notice that $\efr(t)$ (see \eqref{eq:37}) is a bounded function of $\Wb$.
  Taking the directional derivative of $\efr(t)$ with respect to $\Wb$ in the direction $\Wbt$ in \eqref{eq:65} gives 
  \begin{align}
    \nonumber
    \EE\Big[ \widetilde{\Wb}  \Gb(z) \nabla_{\widetilde{\Wb}} \big(\efr(t)\big)\Big]
    &=
      \EE\Big[\widetilde{\Wb} \Gb(z) \efr(t) \nabla_{\widetilde{\Wb}}\Big( \frac{\imu t}{\pi} \int_{\Omega} \frac{\partial \ft(\zeta)}{\partial \overline{\zeta}}  (1-\EE)\big[\Tr \Gb(\zeta) \big] \dzeta \Big) \Big]
    \\ \label{eq:69}
    &=
      \frac{\imu t}{\pi} \EE\Big[\widetilde{\Wb} \Gb(z) \efr(t)    \int_{\Omega}   \frac{\partial \ft(\zeta)}{\partial \overline{\zeta}}    \Tr \Big(\nabla_{\widetilde{\Wb}} \Gb(\zeta)  \, \Big) \,\dzeta \, \Big]  
      .
  \end{align}
  Using the  relation $\frac{\partial}{\partial \zeta} \Gb(\zeta) = \Gb^2(\zeta)$, we get from \eqref{eq:56} and the cyclicity of the trace that
  \begin{equation*}
    \Tr \Big(\nabla_{\widetilde{\Wb}}\Gb(\zeta) \Big)
    =
    - \Tr \Big(\widetilde{\Wb} \frac{\partial}{\partial \zeta} \Gb(\zeta)\Big)
    .
  \end{equation*}
  Combining this with \eqref{eq:69}  gives
  \begin{equation}
    \label{eq:71}
    \EE\Big[ \widetilde{\Wb}  \Gb(z) \nabla_{\widetilde{\Wb}} \big(\efr(t)\big)\Big]
    =
    - \frac{\imu t}{\pi}  \EE\Big[\widetilde{\Wb} \Gb(z) \efr(t) \int_{\Omega}   \frac{\partial \ft(\zeta)}{\partial \overline{\zeta}} \Tr \Big(\widetilde{\Wb} \frac{\partial}{\partial \zeta}\Gb(\zeta) \Big) \,\dzeta  \Big] 
    .
  \end{equation}
  After plugging \eqref{eq:71} into \eqref{eq:66}, using the linearity of $\Bc^{-1}_{z}$ and taking the trace we have that 
  \begin{equation*}
    \EE\big[ \efr(t) (1-\EE)[\Tr \Gb(z)]\big]
    =
    \frac{\imu t}{\pi} \int\displaylimits_{\Omega}   \frac{\partial \ft (\zeta)}{\partial \overline{\zeta}} \frac{\partial}{\partial \zeta} \, h_{1}(z,\zeta) \, \dzeta + \Tr \errL
  \end{equation*}
    with
  \begin{equation*}
    h_{1}(z,\zeta)
    :=
    \EE\bigg[ \Tr \Big(  \Bc_{z}^{-1} (\Id_n\otimes \Tr_N)\Big[ \widetilde{\Wb} \Gb(z) \frac{1}{\Mb(z)}\Big] \Big) \Tr \Big(\widetilde{\Wb} \Gb(\zeta) \Big)  \efr(t) \bigg]
    .
   \end{equation*}
  Together with \eqref{eq:54} we obtain a formula for the derivative of $\EE[\efr(t)]$, namely, 
  \begin{align}
    \label{eq:73}
    \frac{d}{dt} \EE[\efr(t)]
     & =
      - \frac{t}{\pi^2} \int\displaylimits_{\Omega \times \Omega}  \frac{\partial \ft(z)}{\partial \overline{z}}  \frac{\partial \ft(\zeta)}{\partial \overline{\zeta}} \frac{\partial}{\partial \zeta} \, h_{1}(z,\zeta)\,  \dzeta \dz
      +
      \errK
      ,
  \end{align} 
   where
  \begin{equation}
    \label{eq:74}
    \errK
    :=
    \frac{\imu}{\pi} \int_{\Omega} \frac{\partial \ft(z)}{\partial \overline{z}}  \Tr \errL (z) \dz
    .
  \end{equation}
  With the Hilbert space structure on $\CC^{n \times n}$ introduced in \eqref{eq:51}, we find
  \begin{equation*}
   \Tr \Big(  \Bc_{z}^{-1} \big(\Id_n\otimes \Tr_N \big)\Big[\widetilde{\Wb} \Gb(z) \frac{1}{\Mb(z)} \Big] \Big)
    =
    \Tr \bigg(  \Big(\big(\Bcc_{z}^{\,-1}\big)^{*}\big[\Ib_{nN}\big] \Big)^{*} \, \widetilde{\Wb} \Gb(z) \frac{1}{\Mb(z)} \bigg)
    ,
  \end{equation*}
  where we used that 
  \begin{equation}
    \label{eq:76}
    (\Bcc^{*}_{z})^{-1}[\Ib_{nN}]
    =
    (\Bc^{*}_{z})^{-1}[I_{n}]\otimes \Ib_N
  \end{equation}
  following from the definitions \eqref{eq:48}.
  Moreover, from the cyclicity of the trace we have
  \begin{align}
    \label{eq:77}
     &
      \EE_{\Wbt}\Big[ \Tr \Big( \Bcc_{z}^{-1}\Big[ \widetilde{\Wb} \Gb(z)  \frac{1}{\Mb(z)} \Big] \Big) \Tr \Big(\widetilde{\Wb} \Gb(\zeta) \Big)\Big]
    \\ 
    & 
      =
      \EE_{\Wbt}\Big[ \Tr \Big(\Gb(z)  \frac{1}{\Mb(z)} \big((\Bcc_{z}^{*})^{-1}[\Ib_{nN} ]\big)^*  \widetilde{\Wb} \Big) \Tr \Big(\widetilde{\Wb} \Gb(\zeta) \Big)  \Big]
    \\ \label{eq:78} 
    &
      = \sum_{i,j=1}^{n}\sum_{p,q=1}^{N} \Tr \Big((E_{ji}\otimes \Eb_{q p}) \Gb(z)  \frac{1}{\Mb(z)} \big((\Bcc_{z}^{*})^{-1}[\Ib_{nN} ]\big)^* \Scc [E_{ij}\otimes \Eb_{p q}] \Gb(\zeta) \Big)
      ,
  \end{align}
  where $\{E_{ij}\otimes \Eb_{p q}\, : \, 1 \leq i,j \leq n, 1 \leq p, q \leq N\}$ is the standard basis in $\CC^{nN\times nN} = \CC^{n\times n} \otimes \CC^{N \times N}$, and we used that for any $\Sb,\Tb \in \CC^{n N \times n N}$
  \begin{equation}
    \label{eq:363}
    \Tr (\Sb) \Tr( \Tb) = \sum_{i,j,p,q} \Tr \Big( (E_{ji}\otimes \Eb_{qp})\, \Sb \, (E_{ij}\otimes \Eb_{p q}) \, \Tb \Big)
    .
  \end{equation}
  In \eqref{eq:78} the operator $\Scc$ appears as the result of taking the partial expectation with respect to $\Wbt$ (see \eqref{eq:43}).
  From \eqref{eq:44} we see that the operator $\Scc$ acts on the basis vectors as
  \begin{equation}
    \label{eq:79}
    \Scc [E_{ij}\otimes \Eb_{p q}]
    =
    \left\{
      \begin{array}{ll}
        0,& \mbox{ if }p \neq q,
        \\
        N^{-1}\Gamma[E_{ij}]\otimes \Ib_{N}, & \mbox{ if } p  = q.
      \end{array}
    \right.
  \end{equation}
  Combining \eqref{eq:79}, \eqref{eq:76} and \eqref{eq:78}, leads to the simplified expression of $h_{1}(z,\zeta)$,
  \begin{align}
    \nonumber
    h_{1}(z,\zeta)
    &=\EE\Big[ \Tr \Big( \Bcc_{z}^{-1}\Big[  \widetilde{\Wb} \Gb(z) \frac{1}{\Mb(z)} \Big] \Big) \Tr \Big(\widetilde{\Wb} \Gb(\zeta) \Big)  \efr(t) \Big]
    \\ \label{eq:80}
    &=
    \EE\Big[ \frac{1}{N} \sum_{i,j=1}^{n} \sum_{k,l = 1}^{N}  \Tr \Big(E_{ji} G_{l k}(z)  \frac{1}{M(z)} \big((\Bc_{z}^{*})^{-1}[I_n ]\big)^*   \Gamma [E_{ij} ]   G_{k l}(\zeta) \Big)  \efr(t) \Big]
    .
  \end{align} 
  Since $\big(\Bc_{z}[R]\big)^{*}=\Bc_{z}^{*}[R^{*}]$ for any $R \in \CC^{n\times n}$, we have that $\big((\Bc_{z}^{*})^{-1}[I_n ]\big)^* = \Bc_{z}^{-1}[I_n]$.
  Together with \eqref{eq:73} this gives the leading term in \eqref{eq:52}.

  It remains to show that $|\errK| \prec N^{-\tau}$, where $\errK$ is defined in terms of $\errL,\errH$ and $\errC$ through \eqref{eq:74}, \eqref{eq:67}, \eqref{eq:63} and \eqref{eq:def errJ2}.
  First, from the bounds of the error terms \eqref{eq:42} in the  cumulant expansion formula \eqref{eq:41}, we have   
  \begin{equation}\label{eq:bound on errC}
  \| \errC(z)\|_{\max} = O_{\prec}\big((1+|t|^{3})N^{5 \tau/2} N^{-1} \eta_0^{-1/2}\big)
  \end{equation}
  uniformly for $z\in \Omega$.
    Next we use the local laws \eqref{eq:32} and \eqref{eq:33} to estimate  the error term $\errL$ from \eqref{eq:67} and  $\errH$ from \eqref{eq:63}.
  By \eqref{eq:44} we see the identity
  \begin{equation*}
    \Scc[\Mb(z) - \Gb(z)]
    =
    \Gamma\Big[M(z) - \frac{1}{N}\sum_{j=1}^{N} G_{jj}(z) \Big]\otimes \Ib_N
    .
  \end{equation*}
  Therefore, we can rewrite the term appearing in the definition of $\errH (z)$ in \eqref{eq:63} as
  \begin{multline*}
    \Scc[\Mb(z)-\Gb(z)]\big( \Mb(z) -\Gb(z) \big)
    \\=
    \sum_{i,j=1}^{N} \Gamma\Big[M(z) - \frac{1}{N}\sum_{k=1}^{N} G_{kk}(z) \Big]\big( M(z)\,\delta_{ij} - G_{ij}(z) \big) \otimes \Eb_{ij}
    .
  \end{multline*}
  By applying the partial trace $(\Id_n\otimes \Tr_N) $ to the above identity and using the averaged local law  \eqref{eq:33}, we get
  \begin{equation}
    \label{eq:83}
    \Big\| (\Id_n\otimes \Tr_N) \Big[\Scc[\Mb(z)-\Gb(z)]\big( \Mb(z) -\Gb(z) \big) \Big] \Big\|
    \prec
    \frac{1}{N (\Im z)^{2}}
  \end{equation}
  uniformly on $\Omega$. 
  It follows from Proposition~\ref{pr:1} that $\|M^{-1}(z)\|$ and $\|\Bc_z^{-1}\|$ are uniformly bounded  on $\Omega$.
  Using this, the bound on $\errC(z)$ from \eqref{eq:bound on errC} and \eqref{eq:83} we have the estimate
  \begin{equation}
    \label{eq:84}
    \| \errL (z) \|
    \prec
    N^{-1}( \Im z)^{-2} + (1 + |t|^{3}) N^{5\tau/2}\eta_{0}^{-1/2}
  \end{equation}
  holding uniformly on $\Omega$.
  By Stokes' theorem, for any function $H:\Omega \to \CC$ with continuously differentiable real and imaginary parts
  \begin{equation}
    \label{eq:85}
    \int_{\Omega} \frac{\partial H(z) }{\partial \overline{z}} \dz 
    =
    \frac{-\imu}{2} \int_{\partial \Omega} H(z) dz
    .
  \end{equation}
  Using now \eqref{eq:85} and the analyticity of $\Tr \errL(z)$ on $\Omega$, we rewrite $\errK$ defined in \eqref{eq:74} as
  \begin{equation}
    \label{eq:86}
    \errK
    =
     \frac{\imu}{\pi} \int_{\Omega} \frac{\partial \ft(z)}{\partial \overline{z}}  \Tr \errL (z) \, \dz 
    =
     \frac{1}{2\pi} \int_{\partial \Omega} \ft(z)  \Tr \errL (z) \, dz
    .
  \end{equation}
  Since $\ft(z)$ vanishes everywhere on $\partial \Omega$ except the lines $|\Im z| = N^{-\tau} \eta_0$, we obtain from \eqref{eq:23}, \eqref{eq:26} and \eqref{eq:84} the estimate
  \begin{multline}
    \label{eq:87}
    |\errK|
    \prec
    \int_{E_0-\delta}^{E_0+\delta} \big(|f(x)| +  N^{-\tau}\eta_0 |f'(x)|\big) \bigg(\frac{1}{N(N^{-\tau} \eta_0)^{2}} + \frac{N^{5\tau/2}}{\sqrt{\eta_0}} \bigg)dx
    \\
    \leq
    \big(\|g\|_{1} + \|g'\|_{1}\big)\Big(\frac{N^{2\tau}}{N \eta_0} + N^{ 5\tau / 2}\sqrt{\eta_0} \Big)
    .
  \end{multline}
  From the assumption $\tau \in (0, \min\{\gamma,(1 - \gamma)\}/ 7)$ we have that $N\eta_0 = N^{1-\gamma} > N^{ 7\tau}$ and $\sqrt{\eta_0} = N^{-\gamma/2} < N^{- 7\tau/2 }$ , which together with \eqref{eq:87} establishes the estimate for $|\errK|$.
\end{proof}

\subsection{Equation for $S_{ij}(z,\zeta)$}
\label{sec:proof-3}

Our next step is to analyze the term
\begin{equation*}
  S_{ij}(z,\zeta)
  :=
  \frac{1}{N} \sum_{k,l = 1}^{N}  G_{l k}(z) \Bc_{z}^{-1}[I_n ] M(z)  \Gamma [E_{ij} ]   G_{k l}(\zeta)
\end{equation*}
appearing in \eqref{eq:52}.
It will be convenient to consider a more general quantity of the form
\begin{equation}
  \label{eq:89}
  G^{B}(z,\zeta)
  :=
  \frac{1}{N}\sum_{k,l = 1}^{N} G_{l k}(z) B G_{k l}(\zeta)
\end{equation}
with an arbitrary fixed deterministic $B\in\CC^{n\times n}$.
We can then specialize $G^{B}(z,\zeta)$ to $S_{ij}(z,\zeta)$ by taking $B = \Bc_{z}^{-1}[I_n ]  M(z)  \Gamma [E_{ij} ] $.
We start by showing that $G^{B}(z,\zeta)$ satisfies the following self-consistent equation with random error term.
The derivation is reminiscent of the proofs used in some earlier works on the mesoscopic CLT for random matrices with independent entries (e.g., \cite[Lemma~4.3]{LiXu21}, \cite[Proposition~5.1 and Lemma~5.2]{LandLopaSoso23} or \cite[Theorem~3.2]{Riab}). 
\begin{lem}
  \label{lem:3}
  Let $\gamma \in (0,1)$ and $\tau \in (0, 1-\gamma)$. Then for any $B\in \CC^{n\times n}$ 
  \begin{equation}
    \label{eq:90}
    G^{B}(z,\zeta)
    =
    M(z) B M(\zeta) + M(z) \Gamma\big[ G^{B}(z,\zeta) \big] M(\zeta) + \errE^{B}(z,\zeta)
  \end{equation}
  with the error term $\errE^{B}(z,\zeta)$ being analytic on $\Omega\times \Omega$ and satisfying
  \begin{equation}
    \label{eq:91}
    \EE[\| \errE^{B}(z,\zeta) \|]
    \prec
    \frac{1}{N^{1/2} (\min\{|\Im z|,|\Im \zeta|\})^{3/2}}
    .
  \end{equation}
  uniformly for $(z,\zeta)\in \Omega\times \Omega$.
\end{lem}
\begin{proof}
  Similarly as in \eqref{eq:30}, we define the (left) matrix coefficients $\{W_{ij}\}_{1\leq i,j\leq N}$ via $\Wb = \sum_{i,j=1}^{N} W_{ij} \otimes \Eb_{ij}$ with $W_{ij} \in \CC^{n\times n}$.
  Directly from the definition \eqref{eq:39} we get
  \begin{equation}
    \label{eq:93}
    W_{ij}
    =
    \sum_{\alpha=1}^{d}\bigg(L_{\alpha}  \, x^{(\alpha)}_{ij} + L_{\alpha}^{*}  \, \overline{x}^{\,(\alpha)}_{ji}\bigg)
    ,
  \end{equation}
  where $x^{(\alpha)}_{ij}$ are the entries of $\Xb_{\alpha}$.
  For any $i\in \{1,\ldots, N\}$, we denote by $\Xb_{\alpha}^{(i)}$ the random i.i.d. matrix $\Xb_{\alpha}$ with the $i$-th row and $i$-th column removed.
  To make the notation consistent and easier to follow, we index the rows and columns of $\Xb_{\alpha}^{(i)}$ by $\{1,\ldots, N\}\setminus \{i\}$.
  The resolvent of the model with removed $i$-th rows and columns in each $N\times N$ block is
  \begin{equation*}
    \Gb^{(i)}(z)
    :=
    \bigg(K_0\otimes \Ib_{N-1} + \sum_{\alpha=1}^{d}\Big( L_{\alpha} \otimes \Xb_{\alpha}^{(i)} + L_{\alpha}^{*} \otimes \big(\Xb_{\alpha}^{(i)}\big)^{*} \Big) - z \Ib_{n(N-1)} \Big)^{-1}
    ,
  \end{equation*}
  and the corresponding (left) coefficient matrices are
  \begin{equation*}
    \Gb^{(i)}(z)
    =
    \sum_{p,q \neq i} G^{(i)}_{p q}(z) \otimes \Eb_{p q}
    .
  \end{equation*}
  With this notation, the Schur complement formula yields
  \begin{align}
    \label{eq:96}
    &\frac{1}{G_{ii}(z)}
      =
      W_{ii} + K_0 - z - \sum_{p, q \neq i} W_{i p} \,G^{(i)}_{p q}(z) \, W_{q i}
      ,
    \\ \label{eq:97}
    &G_{ij}(z)
      =
      -G_{ii}(z) \sum_{p \neq i} W_{i p} \, G^{(i)}_{p j}(z)
      =
      -\sum_{p \neq j} G^{(j)}_{i p }(z) W_{p j }  \, G_{jj}(z)
  \end{align}
  for all $1\leq i,j\leq N$, $i\neq j$.
  Moreover, for $l \neq k$ we have
  \begin{align}
    \label{eq:98}
    G_{l k}
    =
    M\frac{1}{G_{ll}}G_{lk} + (G_{ll} - M) \frac{1}{M + (G_{ll} - M)}G_{lk}
    ,
  \end{align}
  where we dropped the spectral parameter $z$ for brevity.
  Denote $\eta := |\Im z|$.
  From the local law \eqref{eq:32}, the bounds $\|G_{ll} - M\| \prec (N\eta)^{-1/2}$ and  $ \|G_{lk} \| \prec (N\eta)^{-1/2}$ hold uniformly on $\Omega$.
  The functions $\|M\|$ and $\|M^{-1}\|$ are uniformly bounded on $\Omega$  (see Proposition~\ref{pr:1}), which, in turn, implies that $\|M\|\sim 1 $ uniformly on  $\Omega$.
  After applying the local law \eqref{eq:32} and the first equality in \eqref{eq:97} to the first term in \eqref{eq:98}, we arrive at
  \begin{equation}
    \label{eq:99}
    G_{lk}
    =
    - M \sum_{p \neq l} W_{ l p } G^{(l)}_{ p k}  + O_{\prec}\Big(\frac{1 }{N\eta}\Big)
    ,
  \end{equation}
  where for any $r\in \NN$, any (random) $\phi>0$ and  (random) $r\times r$ matrix $\Psi$ we write $\Psi = O_{\prec}(\phi)$ if $\| \Psi \| \prec \phi$.
  Similarly, by putting $M$ on the right-hand side in \eqref{eq:98}, we get that for $k\neq l$
  \begin{equation}
    \label{eq:100}
    G_{kl}
    =
    - \sum_{p \neq l} G^{(l)}_{k p } W_{ p l }  M  + O_{\prec}\Big(\frac{1}{N\eta}\Big)
    .
  \end{equation}
  By the local law \eqref{eq:32} we also see that
  \begin{equation}
    \label{eq:101}
    M\frac{1}{G_{ll}}G_{lk}
    =
    O_{\prec}\Big(\frac{1}{\sqrt{ N\eta}} \Big)
    ,\quad
    G_{kl}\frac{1}{G_{ll}} M
    =
    O_{\prec}\Big(\frac{1}{\sqrt{N\eta}} \Big)
  \end{equation}
  holds for $l\neq k$.
  Denoting $\etaH := \min\{|\Im z|, |\Im \zeta|\}$ with $N^{-\gamma - \tau} \lesssim \etaH \lesssim N^{-\gamma}$ on $\Omega \times \Omega$,  we conclude that for $l\neq k$ the identity
  \begin{align}
    \label{eq:102}
    G_{lk}(z) B G_{kl}(\zeta)
    &=
      M(z) \sum_{p \neq l} W_{ l p } G^{(l)}_{ p k}(z)   B \sum_{q \neq l} G^{(l)}_{kq }(\zeta)  W_{q  l}  M(\zeta)  + O_{\prec}\Big(\frac{1 }{(N\etaH)^{\,3/2}}\Big)
  \end{align}
  holds uniformly on $(z,\zeta)\in \Omega \times \Omega$.
  By construction, $G^{(l)}_{pk}$ and $G^{(l)}_{kq}$ are independent of $W_{lp}$ and $W_{ql}$.
  After taking the partial  expectation $\EE_{l}$ with respect to $\{W_{lp}, W_{ql} \, : \, 1\leq p,q \leq N\}$ (denoted below by $\EE_{l}$) we rewrite \eqref{eq:102} as
  \begin{align}
    \label{eq:103}
    G_{lk}(z) B G_{kl}(\zeta)
    &=
      M(z) \Gamma \Big[\frac{1}{N}  \sum_{p \neq l} G^{(l)}_{ p k}(z)   B  G^{(l)}_{kp }(\zeta)  \Big] M(\zeta) 
    \\ \label{eq:104}
    &\qquad
      +
      M(z) (1-\EE_{l})\Big[ G_{lk}(z) B G_{kl}(\zeta) \Big]  M(\zeta) 
      + O_{\prec}\Big(\frac{1 }{(N\etaH)^{\,3/2}}\Big)
      .
  \end{align}
  In \eqref{eq:103} the (linear) operator $\Gamma$ appears as a result of the structure of $W_{ij}$  (see \eqref{eq:93}) after applying $\EE_{l}$ and using that for the complex i.i.d. matrices $\Xb_{\alpha} = (x_{pq}^{(\alpha)})_{1\leq p,q \leq N}$ 
  \begin{equation} 
    \label{eq:105}
    \EE[x_{kp}^{(\alpha_1)}x_{ql}^{\,(\alpha_2)}]
    =
    0
    ,\quad
    \EE[x_{kp}^{(\alpha_1)}\overline{x}_{ql}^{\,(\alpha_2)}]
    =
    \delta_{\alpha_1 \alpha_2}\delta_{kq} \delta_{pl} \frac{1}{N}
    .
  \end{equation}
  From the standard identity (a consequence of the Woodbury formula)
  \begin{equation}
    \label{eq:106}
    G_{pk}^{(l)}
    =
    G_{pk} - G_{pl}\frac{1}{G_{ll}} G_{lk}
  \end{equation}
  holding for all $l\notin \{p,k\}$, together with \eqref{eq:32}, we deduce that
  \begin{equation}
    \label{eq:107}
    G_{pk}^{(l)}
    =
    G_{pk} + O_{\prec}\Big( \frac{1}{N\eta }\Big)
    .
  \end{equation}
  This implies that
  \begin{equation}
    \label{eq:108}
    G^{(l)}_{ p k}(z)   B  G^{(l)}_{kp }(\zeta)
    =
    G_{ p k}(z)  B  G_{kp }(\zeta) +
    \left\{
      \renewcommand\arraystretch{1.5}
      \begin{array}{ll} 
        O_{\prec}\big( (N\etaH)^{\,-1} \big),  & p = k,
        \\
        O_{\prec}\big( (N\etaH)^{\,-3/2} \big),  & p \neq k.
      \end{array}\right.    
  \end{equation}
  Recall that we are dealing with the case $k\neq l$, and  $\etaH \gg N^{-1} $ for $(z,\zeta)\in \Omega\times \Omega$.
  We replace the $G^{(l)}$ entries in \eqref{eq:103} using formula \eqref{eq:108} to get 
  \begin{align*}
    \frac{1}{N}  \sum_{p \neq l} G^{(l)}_{ p k}(z)   B  G^{(l)}_{k p }(\zeta)
    & =
      \frac{1}{N}  \sum_{p =1}^{N} G_{ p k}(z)   B  G_{k p }(\zeta)
      +
      O_{\prec}\Big( \frac{1}{(N\etaH)^{\,3/2}}\Big)
      .
  \end{align*}
  The boundedness of $M$ and $\Gamma$ (as an operator acting on $\CC^{n\times n}$) implies that 
  \begin{align*}
    G_{lk}(z) B G_{kl}(\zeta)
    &=
      \frac{1}{N}   M(z) \Gamma \Big[\sum_{p =1}^{N} G_{ p k}(z)   B  G_{kp }(\zeta)  \Big] M(\zeta) 
    \\
    &\qquad
      +
      M(z) (1-\EE_{l})\Big[ G_{lk}(z) B G_{kl}(\zeta) \Big] M(\zeta)
      +
      O_{\prec}\Big(\frac{1 }{(N\etaH)^{\,3/2}}\Big)
  \end{align*}
  holds for  $k\neq l$.
  By the local law \eqref{eq:32}, $\sum_{p =1}^{N} G_{ p k}(z)   B  G_{kp }(\zeta) = O_{\prec}(1 + \etaH^{\,-1})$ for any $k\in \{1,\ldots, N\}$.
  Therefore, summing the above equality over $N-1$ indices $l$ for $l\neq k$ gives 
  \begin{align}
    \label{eq:111}
    &\sum_{ l:\,l\neq k } G_{lk}(z) B G_{kl}(\zeta)
      =
      M(z)  \Gamma \Big[ \sum_{p =1}^{N} G_{ p k}(z)   B  G_{kp }(\zeta)  \Big] M(\zeta) 
    \\  \nonumber
    &\qquad\qquad\qquad
      +
      M(z) \sum_{l:\,l\neq k}(1-\EE_{l}) \Big[ G_{lk}(z) B G_{kl}(\zeta) \Big] M(\zeta)
      + O_{\prec}\Big(\frac{1 }{N^{1/2}\etaH^{\,3/2}}\Big)
  \end{align}
  where we used that $N^{-1} + (N \etaH)^{-1} \ll N^{-1/2}\etaH^{\,-3/2}$ for $(z,\zeta) \in \Omega \times \Omega$.
  Finally, in order to obtain the first term on the right hand side of \eqref{eq:90}, we use the local law \eqref{eq:32} and apply it to the diagonal terms $G_{kk}$
  \begin{equation}
    \label{eq:112}
    G_{kk}(z) B G_{kk}(\zeta)
    =
    M(z) B M(\zeta) + O_{\prec}\Big(\frac{1}{\sqrt{N \etaH}} \Big)
    ,
  \end{equation}
  which together with \eqref{eq:111} gives 
  \begin{align}
    \nonumber
    \sum_{l=1}^{N}G_{lk}(z) B G_{kl}(\zeta)
    &=
      M(z) B M(\zeta)+M(z) \Gamma \Big[  \sum_{p =1}^{N} G_{ p k}(z)   B  G_{k p }(\zeta)  \Big] M(\zeta) 
    \\ \label{eq:114}
    &\qquad
      +  M(z) \sum_{l\neq k}(1-\EE_{l})\Big[ G_{lk}(z) B G_{kl}(\zeta) \Big] M(\zeta)
      + O_{\prec}\Big(\frac{1}{N^{1/2}\etaH^{\,3/2}}\Big)
      .
  \end{align}
  It is left to control
  $ D_k:= \sum_{l\neq k} (1-\EE_{l})\big[  G_{lk}(z)   B  G_{k l }(\zeta)   \big] $.
  For this we use the  decomposition
  \begin{align}
    \label{eq:115}
    D_k D_k^{*}
    & =
      \sum_{l:\,l\neq k} (1-\EE_{l})\Big[  G_{lk}(z)   B  G_{k l }(\zeta)\Big]  \bigg((1-\EE_{l})\Big[  G_{lk}(z)   B  G_{k l }(\zeta)   \Big]\bigg)^*
    \\ \label{eq:116}
    & \quad
      +
      \sum_{\substack{l, j: \, l,j \neq k,\\  l\neq j}} (1-\EE_{l})\Big[  G_{lk}(z)   B  G_{k l }(\zeta)   \Big] \bigg(  (1-\EE_{j})\Big[  G_{jk}(z)   B  G_{k j }(\zeta)   \Big]\bigg)^*
    .
  \end{align}
  By \eqref{eq:32} the first term is of order $O_{\prec}(N^{-1}\etaH^{\,-2})$.
  For the second term we use \eqref{eq:106} to see
  \begin{align*} 
    &(1-\EE_{l})\Big[  G_{lk}(z)   B  G_{k l }(\zeta)   \Big] \bigg(  (1-\EE_{j})\Big[  G_{jk}(z)   B  G_{k j }(\zeta)   \Big]\bigg)^*
    \\
    &
      =
      (1-\EE_{l})\Big[  \Big(G_{lk}^{(j)}(z) + G_{lj}(z)\frac{1}{G_{jj}(z)}G_{jk}(z) \Big)   B  \Big(G_{kl}^{(j)}(\zeta) + G_{kj}(\zeta)\frac{1}{G_{jj}(\zeta)}G_{jl}(\zeta) \Big) \Big]
    \\
    &\quad
      \times \bigg(  (1-\EE_{j})\Big[  \Big(G_{jk}^{(l)}(z) + G_{jl}(z)\frac{1}{G_{ll}(z)}G_{lk}(z)   \Big)B  \Big(G_{k j }^{(l)}(\zeta) + G_{kl}(\zeta)\frac{1}{G_{ll}(\zeta)}G_{lj}(\zeta) \Big) \Big]\bigg)^*
      .
  \end{align*}
The expectation of the summands in \eqref{eq:116} vanishes, i.e.,
  \begin{equation*}
    \EE\Big[  (1-\EE_{l})\big[  G_{lk}^{(j)}(z)  B  G_{kl}^{(j)}(\zeta) \big] (1-\EE_{j})\big[  G_{jk}^{(l)}(z) B G_{k j }^{(l)}(\zeta)  \big]\Big]
    =
    0
    .
  \end{equation*}
  Therefore, the expectation of $D_kD_k^*$ can be written as a sum of at most $O(N^2)$ non-zero terms, each containing a product of at least 6 off-diagonal coordinate matrices of $\Gb$ or $\Gb^{(l)}$.
  By the local law \eqref{eq:32} each such product is of order $O_{\prec}((N\etaH)^{\,-3})$.
  All the terms in \eqref{eq:115} are also deterministically bounded by a sufficiently high power of $1/\etaH$, which, in particular, means that $\|D_k D_k^*\| \leq N^{C}$ for some $C>0$.
  Combining this with the stochastic domination estimate $D_k D_k^* \prec N^{-1} \etaH^{\, -3}$ (defined in \eqref{eq:31}), we obtain an estimate for the expectation $\EE[D_k D_k^{*}] = O_{\prec}(N^{-1} \etaH^{\,-3})$.
  Notice, that $\EE[\|D_k\|^2_{\mathrm{HS}}]=\EE[\Tr(D_k D_k^*)/n] = O_{\prec}(N^{-1} \etaH^{\,-3})$, where $\|\cdot \|_{\mathrm{HS}}$ is the Hilbert-Schmidt norm on $\CC^{n\times n}$ induced by the scalar product \eqref{eq:51}.
  
  The functions $M$, $G_{ij}$, $M^{-1}$, $G_{ii}^{-1}$ used in the above proof are analytic on $\Omega$.
  By collecting the error terms in \eqref{eq:114} ($O_{\prec}(N^{-1/2}\etaH^{\,-3/2})$ and $D_k$), taking the average over $k\in \{1,\ldots, N\}$, denoting the resulting error by $\errE^{B}(z,\zeta)$  and using the equivalence of norms on $\CC^{n\times n}$, we obtain \eqref{eq:91}, which finishes the proof.
\end{proof}

\subsection{Self-consistent equation for $G^{B}(z,\zeta)$}
\label{sec:proof-4}
In this section we study the limiting behavior of $G^{B}(z,\zeta)$ introduced in \eqref{eq:89} through the analysis of the self-consistent equation \eqref{eq:90}.
More precisely, for any $z,\zeta \in \Omega$ and $B \in \CC^{n\times n}$ we consider the equation
\begin{equation}
  \label{eq:119}
  M^{B}(z,\zeta)
  =
  M(z) B M(\zeta)  + M(z) \Gamma \big[M^{B}(z,\zeta)\big] M(\zeta)
\end{equation}
for $M^{B}(z,\zeta)$.
For any $Q_1,Q_2\in\CC^{n\times n}$ we denote by $\Cc_{Q_1,Q_2}$ an operator on $\CC^{n\times n}$ given by
\begin{equation}
  \label{eq:120}
  \Cc_{Q_1,Q_2}[R]
  :=
  Q_1 R Q_2
\end{equation}
for all $R\in\CC^{n\times n}$.
The operator $\Cc_{M(z),M(\zeta)}$ is invertible, and $\Cc_{M(z),M(\zeta)}^{-1} = \Cc_{\frac{1}{M(z)},\frac{1}{M(\zeta)}}$.
With this notation, equation \eqref{eq:119} reads
\begin{equation}
  \label{eq:121}
  \Bc_{z,\zeta}\big[M^{B}(z,\zeta) \big]
  =
  B
  ,
\end{equation}
where we defined $\Bc_{z,\zeta}:=\Cc_{M(z),M(\zeta)}^{-1} - \Gamma$.

If the operator $\Bc_{z,\zeta}$ is invertible, then $M^{B}(z,\zeta)$ is uniquely determined from \eqref{eq:121}.
Here it is sufficient to restrict the parameters $z$ and $\zeta$ to a small neighborhood of $E_0$.
Therefore, we show that $\Bc_{z,\zeta}$ is invertible in a sufficiently small region around the point $(E_0,E_0)\in \RR\times \RR$.
The bound for $\| \Bc_{z,\zeta}^{-1}\|$ depends on whether $z$ and $\zeta$ belong to the same half-plane or not.
If $z$ and $\zeta$ are in the same half-plane, then the $\Bc_{z,\zeta}^{-1}$ is bounded around $(E_{0},E_{0})$, while for $z$ and $\zeta$ in different half-planes the main contribution to $\Bc_{z,\zeta}^{-1}$ stems from its smallest isolated eigenvalue of $\Bc_{z,\zeta}$.
The latter is analyzed using analytic perturbation theory.
Similar analysis has also been performed to establish the mesoscopic spectral CLT for Wigner-type matrices in \cite{Riab}.
Denote $\Omega^{+}:=\Omega \cap \CC_{+}$ and $\Omega^{-}:=\Omega \cap \CC_{-}$.
\begin{lem}[Invertibility of $\Bc_{z,\zeta}$]
  \label{lem:4}
  Let $\gamma \in (0,1)$ and $\tau \in (0,\gamma/2)$.
  Then there exists $C>0$ such that the following holds
  \begin{itemize}
  \item[(i)] Uniformly for $(z,\zeta) \in \big(\Omega^{+}\times \Omega^{+} \big)\cup\big(\Omega^{-}\times \Omega^{-} \big)$
    \begin{equation}
      \label{eq:122}
      \big\|\,\Bc_{z,\zeta}^{\,-1}\,\big\|
      \leq
      C
      ;
    \end{equation}
  \item[(ii)] Uniformly for $(z,\zeta) \in \big(\Omega^{-}\times \Omega^{+} \big)\cup\big(\Omega^{+}\times \Omega^{-} \big)$
    \begin{equation}
      \label{eq:123}
      \big\|\,\Bc_{z,\zeta}^{\,-1}\,\big\|
      \leq
      \frac{C}{|z - \zeta|}
      ,
    \end{equation}
    and the operator $\Bc_{z,\zeta}^{\,-1}$ admits the decomposition
    \begin{equation}
      \label{eq:124}
      \Bc_{z,\zeta}^{\,-1}
      =
      \vartheta \,\frac{2\imu }{\la \Im M(E_{0}) \ra} \frac{1}{z - \zeta}  \Im M(E_{0}) \la \Im M(E_{0}),\, \cdot \, \ra + \Jc_{z,\zeta}
      ,
    \end{equation}
    where $\vartheta = 1$ if $(z,\zeta)\in \Omega^{+}\times \Omega^{-}$, $\vartheta = -1$ if  $(z,\zeta)\in \Omega^{-}\times \Omega^{+}$, and  $\|\Jc_{z,\zeta}\|$ is uniformly bounded for $(z,\zeta) \in \big(\Omega^{-}\times \Omega^{+} \big)\cup\big(\Omega^{+}\times \Omega^{-} \big)$.
  \end{itemize}
\end{lem}
\begin{proof}
  Consider first the case (i), and assume that both $z$ and $\zeta$ are in the upper half-plane, i.e.,  $z,\zeta \in \Omega^{+}$.
  The case $z,\zeta \in \Omega^{-}$ follows analogously.
  For $z,\zeta \in \Omega^{+}$ we write
  \begin{align}
    \label{eq:125}
    \Bc_{z,\zeta}
    =
    \Cc^{-1}_{M(z),M(\zeta)}-\Gamma
    =
    \Bc_{z} +\Cc^{-1}_{M(z),M(\zeta)} - \Cc^{-1}_{M(z),M(z)}
    ,
  \end{align}
  where $\Bc_{z}$ was introduced in \eqref{eq:48}.
  By the definition \eqref{eq:120}, we have that
  \begin{equation}
    \label{eq:126}
    \Cc^{-1}_{M(z),M(\zeta)} - \Cc^{-1}_{M(z),M(z)}
    =
    \Cc_{\frac{1}{M(z)},\frac{1}{M(\zeta)}-\frac{1}{M(z)}}
    =
    O(|z-\zeta|)
  \end{equation}
  from the analyticity of $M(z)^{-1}$ on $\CC_{+}$.
  Therefore, using the boundedness of $\Bc_{z}^{-1}$ from \eqref{eq:49}, we get \eqref{eq:122}.
  
  Now suppose that $z$ and $\zeta$ are in different half-planes, and consider first the case  $z\in \CC_{-}$ and $\zeta \in \CC_{+}$.
  Denote $M_0 :=\lim_{y\downarrow  0} M(E_0 + \imu y)$ for brevity.
  From part (ii) of Proposition~\ref{pr:1} we have that $\Im M(z) \sim \rho(z) I_n$ uniformly on $\CC_{+}$, which together with $\rho(E_0) \sim 1$ implies
    \begin{equation}
    \label{eq:127}
    \Im M_{0}
    \sim
    I_n
    .
  \end{equation}
  We see that $\lim_{y\to 0} M(E_{0}+\imu y) - M(E_{0}-\imu y) = 2 \imu \Im M_{0}$ is a non-zero matrix, and thus $\|\Cc_{\frac{1}{M(z)},\frac{1}{M(\zeta)}-\frac{1}{M(z)}}\| = O(1)$.
  Therefore, the perturbation argument used in \eqref{eq:125} and \eqref{eq:126} to control $\Bc_{z,\zeta}$ through comparison with $\Bc_{z}$ and the analyticity of $M(z)$ in the upper (lower) half-plane is not applicable anymore.

  Instead, we consider the perturbation of $\Bc_{z,\zeta}$ around the operator  $\Bc_{E_{0},E_{0}}$ given by
  \begin{equation*}
    \Bc_{E_{0},E_{0}}
    :=
    \Cc_{M_{0}^*,M_{0}}^{-1} - \Gamma
    .
  \end{equation*}
  For convenience, define the centered variables $w:=z - E_{0} \in \CC_{-}$, $\xi:= \zeta - E_{0} \in \CC_{+}$, so that
  \begin{equation*}
    \Bc_{z,\zeta}
    =
    \Bc_{E_0+w, E_0 + \xi}
    =
    \Bc_{E_0, E_0 }
    +
    \Cc_{M(E_0+w), M(E_0 + \xi)}^{-1} - \Cc_{M_{0}^*,M_{0}}^{-1}
    .
  \end{equation*}
  For $(z,\zeta) \in \Omega^{-}\times \Omega^{+}$ the variables $w$ and $\xi$ remain in their corresponding half-planes, the perturbation $\Cc_{M(E_0+w), M(E_0 + \xi)}^{-1} - \Cc_{(M(E_{0}))^*,M(E_{0})}^{-1}$ is an analytic function in $w$ and $\xi$, and thus we can apply the analytic perturbation theory to control the invertibility of $\Bc_{z,\zeta}$ on $\Omega^{-}\times \Omega^{+}$.
  We start by collecting the necessary properties of the spectrum of $\Bc_{E_0,E_0}$.

  Firstly, by taking the imaginary part of the MDE \eqref{eq:7} at $z = E_{0} \in \RR$ we find that
  \begin{equation}
    \label{eq:130}
    \Bc_{E_{0},E_{0}}\big[ \Im M_{0}\big]
    =
    0
    ,\quad
    \Bc_{E_{0},E_{0}}^{*}\big[ \Im M_{0}\big]
    =
    0
    ,
  \end{equation}
  where we used that the adjoint operator to $\Bc_{E_{0},E_{0}}$ with respect to the scalar product \eqref{eq:51} is
  \begin{equation*}
    \Bc_{E_{0},E_{0}}^{*}
    =
    \Cc_{M_{0},M_{0}^{*}}^{-1} - \Gamma
    .
  \end{equation*}
  This means that $\Bc_{E_{0},E_{0}}$ is not invertible, and the kernels of both $\Bc_{E_{0},E_{0}}$ and $\Bc_{E_{0},E_{0}}^{*}$ contain the eigenvector $\Im M_{0}$.
  We now show that
  \begin{equation}
    \label{eq:132}
    \mathrm{dim}\big(\mathrm{ker}\big(\Bc_{E_{0},E_{0}}\big)\big)
    =
    \mathrm{dim}\big(\mathrm{ker}\big(\Bc_{E_{0},E_{0}}^{*}\big)\big)
    =
    1
    .
  \end{equation}
  To this end, we use the balanced polar decomposition of $M_{0}$ from \cite[Eq.~(3.1)]{AltErdoKrug20}
  \begin{equation}
    \label{eq:133}
    M_{0}
    =
    Q^{*} U Q
  \end{equation}
  with unitary $U\in \CC^{n\times n}$ and invertible $Q\in \CC^{n\times n}$.
  Notice that the decomposition \eqref{eq:133} is well defined since $\Im M_{0}$ is (strictly) positive definite as shown in \eqref{eq:127}.
  The operator $\Bc_{E_{0},E_{0}}$ can be written in terms of $U$ and $Q$ as
  \begin{equation}
    \label{eq:134}
    \Bc_{E_{0},E_{0}}
    =
    \Cc_{Q,Q^*}^{-1}\Cc_{U^*,U}^{-1}\Cc_{Q^*,Q}^{-1} - \Gamma
    =
    \Cc_{Q,Q^*}^{-1}\Big(\Cc_{U^*,U}^{-1} - \Fc\Big)\Cc_{Q^*,Q}^{-1}
    ,
  \end{equation}
  where $\Fc := \Cc_{Q,Q^*}\Gamma \,\Cc_{Q^*,Q}$  is a self-adjoint and positivity preserving operator.
  Similarly, we can rewrite the MDE \eqref{eq:7} in terms of $U$, $Q$ and $\Fc$ as
  \begin{equation*}
    -\frac{1}{U}
    =
    E_{0} Q Q^{*} - Q K_0 Q^{*} + \Fc\big[U\big]
    .
  \end{equation*}
  Notice that by \eqref{eq:127} and \eqref{eq:133} we have that $\Im U \geq c I_n$ for some $c>0$.
  Since $U$ is unitary, the above equation yields
  \begin{equation}
    \label{eq:136}
    \Im U
    =
    \Fc \big[ \Im U\big]
    .
  \end{equation}
  Let $\|\Fc\|_{2}$ denote the operator norm of $\Fc$ induced by the Hilbert-Schmidt norm on $\CC^{n\times n}$ with the scalar product \eqref{eq:51}.
  It follows from \eqref{eq:136} and the properties \eqref{eq:318} and \eqref{eq:320} of $\Fc$ discussed in Appendix~\ref{sec:lflat} that $\|\Fc\|_{2} = 1$ is a simple eigenvalue of $\Fc$, and the operator $\Fc$ has a spectral gap
  \begin{equation}
    \label{eq:137}
    \mathrm{Spec}(\Fc)
    \subset
    [-1+\kappa,1-\kappa]\cup\{1\}
  \end{equation}
  for some $\kappa>0$ sufficiently small.
  Denote by $F:=\Im U / \|\Im U\|_{\mathrm{HS}}$ the normalized eigenvector of $\Fc$ corresponding to the eigenvalue $\|\Fc\|_2 = 1$.
  Since $F$ commutes with $U$ and $U^{*}$, and $\Fc$ is self-adjoint, we get that 
  \begin{equation*}
    \Big(\Cc_{U^*,U}^{-1} - \Fc\Big)\big[F\big]
    =
    0
    ,\qquad
    \Big(\Cc_{U^*,U}^{-1} - \Fc\Big)^{*}\big[F\big]
    =
    0
    .
  \end{equation*}
  Moreover, if $Y\in \CC^{n\times n}$ is such that $\la Y , F \ra = 0$, then \eqref{eq:137} and the fact that $\|\Fc\|_{2}=1$ is a simple eigenvalue imply that
  \begin{equation*}
    \Big\|\Big(\Cc_{U^*,U}^{-1} - \Fc\Big)\big[Y\big]\Big\|_{\mathrm{HS}}
    \geq
    \big\|U Y U^*\big\|_{\mathrm{HS}} - \big\|\Fc\big[Y\big]\big\|_{\mathrm{HS}}
    \geq
    \kappa \|Y\|_{\mathrm{HS}}
    .
  \end{equation*}
  We conclude that $\mathrm{dim} \big( \mathrm{ker}\big(\Cc_{U^*,U}^{-1} - \Fc\big)\big) = 1$, which together with \eqref{eq:134} and the invertibility of $Q$ implies \eqref{eq:132}.

  From \eqref{eq:130} and \eqref{eq:132} we know that $\Bc_{E_{0}, E_{0}}$ has eigenvalue  $0$ with geometric multiplicity $1$, and the corresponding left and right eigenvectors coincide and are equal to $\Im M_0$.
  If we assume that the algebraic multiplicity of the eigenvalue at $0$ is greater than $1$, then the corresponding Jordan chain contains a generalized eigenvector $T \in \CC^{n\times n}$ such that $\Bc_{E_{0}, E_{0}}[T] = \Im M_0$.
  This implies that
  \begin{equation*}
    \la \Im M_0, \Im M_0 \ra
    =
    \la \Im M_0, \Bc_{E_{0}, E_{0}}[T] \ra
    =
    \la \Bc_{E_{0}, E_{0}}^{*}[\Im M_0], T \ra
    =
    0
    ,
  \end{equation*}
  which contradicts to $\|\Im M_{0}\|_{\mathrm{HS}} \geq \theta$.
  We, therefore, conclude that the algebraic multiplicity of the eigenvalue $0$ of $\Bc_{E_{0},E_{0}}$ is also equal to $1$.

  Since zero is a simple eigenvalue, we can apply the analytic perturbation theory of non-Hermitian operators to control the spectrum of $\Bc_{z,\zeta}$.
  At the same time, the dimension of the space $\CC^{n\times n}$ is a fixed model parameter independent of $N$.
  Therefore, we can find a sufficiently small $\varepsilon>0$ such that $\mathrm{Spec}(\Bc_{E_{0},E_{0}})\cap \{v\in \CC \, : \, |v|< \varepsilon\} = \{0\}$ and
  \begin{equation*}
    \sup_{|v|=\varepsilon} \Big(\big\|\big(\Bc_{E_{0},E_{0}}-v \Id\big)^{-1}\big\| + \big\|\big(\Bc_{E_{0},E_{0}}^{*}-v \Id\big)^{-1}\big\|\Big)
    \lesssim
    1
    .
  \end{equation*}

  Recall that the function $M(z)$ defined in \eqref{eq:13} has a jump at $z = E_0 \in \RR$.
  In order to apply the analytic perturbation theory, we will restrict $M(z)$ to the set $\CC_{+}$ or $\CC_{-}$ depending on whether $z \in \Omega^{+}$ or $z\in \Omega^{-}$, and then use part (v) of Proposition~\ref{pr:1} to extend it analytically to a neighborhood of $E_0 \in \RR$ containing $\Omega$.
  Thus, for sufficiently small $\varepsilon>0$ we define $M_{+}: \CC_{+}\cup \{z\, : \, |z - E_0|< \varepsilon\} \to \CC^{n\times n}$ such that $M_{+}(z) = M(z)$ on $\CC_{+}$ and $M_{+}$ is analytic on $\CC_{+}\cup \{z\, : \, |z - E_0|< \varepsilon\}$.
  Similarly, we define $M_{-}: \CC_{-}\cup \{z\, : \, |z - E_0|< \varepsilon\} \to \CC^{n\times n}$ such that $M_{-}(z) = M(z)$ on $\CC_{-}$ and $M_{-}$ is analytic on $\CC_{-}\cup \{z\, : \, |z - E_0|< \varepsilon\}$.
  In particular, if we denote $M_0:= \lim_{y\downarrow 0} M(E_0 + \imu y)$, then $M_{+}(E_0) = M_0$ and $M_{-}(E_0) = M_0^{*}$.

  Consider the (analytic) perturbation of $\Bc_{E_{0},E_{0}}$ by $\Dc_{w,\xi}$, where
  \begin{equation*}
    \Dc_{w,\xi}
    :=
    \Cc_{M_{+}(E_0+w), M_{-}(E_0 + \xi)}^{-1} - \Cc_{(M(E_{0}))^*,M(E_{0})}^{-1}
    ,
  \end{equation*}
  so that $\Bc_{z,\zeta} = \Bc_{E_{0},E_{0}} + \Dc_{w,\xi}$.
  Denote by $L:=\Im M_{0}/\|\Im M_{0}\|_{\mathrm{HS}}$ the (left and right) normalized eigenvector of $\Bc_{E_{0},E_{0}}$ corresponding to the eigenvalue $0$.
  Then it follows from the analytic perturbation theory (see, e.g., \cite[Lemma~C.1]{AltErdoKrug20}) that $\Bc_{z,\zeta}$ has a unique eigenvalue inside the disk $\{v\, : \, |v|\leq \varepsilon\}$, which we denote by $\lambda_{z,\zeta}$, that satisfies
  \begin{equation}
    \label{eq:143}
    \lambda_{z,\zeta}
    =
    \la L,\Dc_{w,\xi}[L]\ra + O(|w|^{2}+ |\xi|^{2})
    .
  \end{equation}
  In order to separate the leading term in \eqref{eq:143}, we calculate the derivative $(\partial_w \Dc_{w,\xi}, \partial_\xi \Dc_{w,\xi})$ at $(w,\xi) = (0,0)$
  \begin{equation*}
    \big(\partial_w \Dc_{w,\xi}, \partial_\xi \Dc_{w,\xi}\big)\Big|_{(w,\xi)=(0,0)}
    =
    \Big( \Cc_{-\frac{1}{M_{0}^{*}} (M'_{0})^* \frac{1}{M_{0}^{*}},\frac{1}{M_{0}}} \, , \,  \Cc_{\frac{1}{M_{0}^*},-\frac{1}{M_{0}} M'_{0} \frac{1}{M_{0}}} \Big)
  \end{equation*}
  with $M'_0:=\lim_{y\downarrow 0} M'(E_0 + \imu y)$.
  The first order approximation of $\Dc_{w,\xi}$ gives
  \begin{equation}
    \label{eq:145}
    \lambda_{z,\zeta}
    =
    - \overline{\alpha}\,w  - \alpha \, \xi  + O(|w|^2+|\xi|^2)
    ,
  \end{equation}
  where
  \begin{equation*}
    \alpha
    :=
    \Big\la L , \frac{1}{M_{0}^*} L \frac{1}{M_{0}} M'_{0} \frac{1}{M_{0}} \Big\ra
    .
  \end{equation*}
  We now show that $\Re \alpha = 0$ and $\Im \alpha \geq c> 0$ for some $c>0$.
  For brevity, the $0$ subscript will be dropped, and we will write $M := M_{0}$ and $M':=M'_{0}$.

  Recall that $L = \Im M / \|\Im M\|_{\mathrm{HS}}$, $\Im M$ is symmetric and positive definite, and thus $\alpha \|\Im M\|_{\mathrm{HS}}^2$ is equal to  $\Big\la \big(\Im M\big) \frac{1}{M^*} \big(\Im M \big)\frac{1}{M} M' \frac{1}{M} \Big\ra$.
  After differentiating the Dyson equation \eqref{eq:7} at $z = E_{0}$, we obtain the relation
  \begin{equation*}
    \frac{1}{M} M' \frac{1}{M}
    =
    I_n + \Gamma[M']
    .
  \end{equation*}
  Multiplying the above equation by $\Im M$ on the left and taking the normalized trace yields
  \begin{align}
    \label{eq:148}
    \Big\la \Im M \frac{1}{M} M' \frac{1}{M} \Big\ra
    =
    \la \Im M \ra + \big\la \Im M \,\Gamma[ M'] \big\ra
    =
    \la \Im M \ra + \big\la \Gamma[\Im M]  M' \big\ra
    ,
  \end{align}
  where in the second equality we used that $\Gamma$ is self-adjoint with respect to the scalar product \eqref{eq:51}.
  From the imaginary part of the Dyson equation we know that
  \begin{equation}
    \label{eq:149}
    \Gamma[\Im M]
    =
    \frac{1}{M^*} \Im M \frac{1}{M} 
    .
  \end{equation}
  Plugging the above identity into \eqref{eq:148} gives
  \begin{align*}
    \Big\la \Im M \frac{1}{M} M' \frac{1}{M} \Big\ra
    =
    \la \Im M \ra + \Big\la \frac{1}{M^*} \Im M \frac{1}{M}   M' \Big\ra
    =
    \la \Im M \ra + \Big\la  \Im M \frac{1}{M} M' \frac{1}{M^*} \Big\ra
    .
  \end{align*}
  After moving the second term on the right-hand side to the left-hand side we get that 
  \begin{align}
    \label{eq:152}
    \la \Im M \ra
    &=
      \Big\la \Im M  \frac{1}{M}  M'\Big( \frac{1}{M} -  \frac{1}{M^*} \Big)\Big\ra
    \\  \nonumber
    & 
      =
      -2\imu  \Big\la \Im M \frac{1}{M} M' \Big(\frac{1}{M} \Im M  \frac{1}{M^*} \Big) \Big\ra
      =
      -2\imu \alpha \| \Im M\|^2_{\mathrm{HS}}
      .
  \end{align}
  Since $\la \Im M \ra \geq \theta$, we conclude that
  \begin{equation*}
    \alpha
    =
    \frac{\imu}{2}\frac{\la \Im M \ra}{\|\Im M \|^2_{\mathrm{HS}}}
    ,
  \end{equation*}
  and thus $\Re \alpha = 0$ and $\Im \alpha >0$.
  In particular, \eqref{eq:145} becomes
  \begin{equation}
    \label{eq:154}
    \lambda_{z,\zeta}
    =
    \frac{\imu}{2}\frac{\la \Im M \ra}{\|\Im M \|^2_{\mathrm{HS}}} \,(w - \xi)  + O(|w|^2+|\xi|^2)
    .
  \end{equation}
  The assumption $\tau < \gamma/2$ ensures that $(|w|^2 + |\xi|^2) = o(|w-\xi|)$ on $\Omega \times \Omega$, and thus \eqref{eq:154} gives the leading term in the expansion of $\lambda_{z,\zeta}$.

  Finally, using analytic functional calculus we can decompose $\Bc_{z,\zeta}^{-1}$ on $\Omega^{-}\times \Omega^{+}$ as
  \begin{equation}
    \label{eq:155}
    \Bc_{z,\zeta}^{-1}
    =
    \frac{1}{\lambda_{z,\zeta}}\Pc_{z,\zeta} + \Jc_{z,\zeta}^{(1)}
  \end{equation}
  with
  \begin{equation*}
    \Pc_{z,\zeta}
    :=
    -\frac{1}{2\pi \imu} \int_{|v|=\varepsilon} \Big(\Bc_{z,\zeta} - v\cdot \Id\Big)^{-1} dv
    ,\quad
    \Jc_{z,\zeta}^{(1)}
    :=
    -\frac{1}{2\pi \imu} \int_{\Sigma} \frac{1}{v}\Big(\Bc_{z,\zeta} - v \cdot \Id\Big)^{-1} dv
    ,
  \end{equation*}
  where $\Sigma$ is a contour encircling the eigenvalues of $\Bc_{z,\zeta}$ located away from the $\varepsilon$-disk around zero and not crossing the circle $\{|v|=\varepsilon\}$.
  By analytic perturbation theory, $\Jc_{z,\zeta}^{(1)}$ is bounded for all $(z,\zeta) \in \Omega^{-}\times \Omega^{+}$.
  Moreover, by using \eqref{eq:154} and perturbation theory for the eigenvectors (see, e.g., \cite[Lemma~C.1]{AltErdoKrug20}) we can rewrite the first term in \eqref{eq:155} as
  \begin{align*}
    \frac{1}{\lambda_{z,\zeta}}\Pc_{z,\zeta}
    =
    -\frac{2\imu }{\la \Im M_{0} \ra} \frac{ \la \Im M_{0},\, \cdot \, \ra}{z - \zeta}  \Im M_{0} + \Jc_{z,\zeta}^{(2)}
  \end{align*}
  with $\Jc_{z,\zeta}^{(2)}$ uniformly bounded on $\Omega^{-}\times \Omega^{+}$.
  Setting $\Jc:=\Jc^{(1)} + \Jc^{(2)}$ establishes \eqref{eq:123} and \eqref{eq:124} and finishes the proof of the lemma for $(z,\zeta)\in \Omega^{-}\times \Omega^{+}$.

  In the case when $(z,\zeta)\in \Omega^{+}\times \Omega^{-}$ the equation \eqref{eq:145} becomes
  \begin{equation*}
    \lambda_{z,\zeta}
    =
    - \alpha\,w  - \overline{\alpha} \, \xi  + O(|w|^2+|\xi|^2)
    ,
  \end{equation*}
  while the remaining part of the proof can be repeated line by line.
  This completes the proof.
\end{proof}

It follows from the above lemma that for any $(z,\zeta) \in \Omega\times \Omega$ the equation \eqref{eq:119} has a unique solution given by
\begin{equation*}
  M^{B}(z,\zeta)
  =
  \Bc_{z,\zeta}^{-1}\big[B\big]
  .
\end{equation*}
Combining the results of Lemmas~\ref{lem:3} and \ref{lem:4}  we find the following estimates for $G^{B}(z,\zeta)$ that depend on whether $z$ and $\zeta$ belong to the same complex half-plane.
\begin{cor}
  \label{cor:1}
  Let $\gamma \in (0,1)$, $\tau \in (0, \min\{\gamma/2, 1-\gamma\})$, and let $B \in \CC^{n\times n}$.
  Denote $\etaH := \min\{|\Im z|, |\Im \zeta|\}$.
\item[(i)] Uniformly for  $(z,\zeta) \in \big(\Omega^{+}\times \Omega^{+} \big)\cup\big(\Omega^{-}\times \Omega^{-} \big)$ 
  \begin{equation}
    \label{eq:160}
    \EE\Big[\big\|G^{B}(z,\zeta)\big\|\Big]
    =
    O_{\prec}\bigg( \big\| B \big\| \Big(1 + \frac{1}{N^{1/2} \etaH^{\,3/2}} \Big)\bigg)
    .
  \end{equation}
\item[(ii)] Uniformly for $(z,\zeta) \in \big(\Omega^{-}\times \Omega^{+} \big)\cup\big(\Omega^{+}\times \Omega^{-} \big)$
  \begin{equation}
    \label{eq:161}
    G^{B}(z,\zeta)
    =
    \vartheta \,\frac{2\imu \la \Im M(E_{0}),\, B \, \ra }{\la \Im M(E_{0}) \ra} \frac{1}{z - \zeta}  \Im M(E_{0})  + \errR^{B}(z,\zeta)
  \end{equation}
  with
  \begin{equation}
    \label{eq:162}
    \EE\Big[ \big\| \errR^{B}(z,\zeta) \big\| \Big]
    =
    O_{\prec}\bigg( \big\| B \big\| \Big(1 + \frac{1}{N^{1/2} \etaH^{\,5/2}} \Big)\bigg)
    .
  \end{equation}
\end{cor}
\begin{proof}
  After  multiplying both sides of equation \eqref{eq:90} by $M^{-1}(z)$ from the left and $M^{-1}(\zeta)$ from the right, and reordering the terms, we get that
  \begin{equation*}
    \Bc_{z,\zeta}\big[ G^{B}(z,\zeta)\big]
    =
    B + \frac{1}{M(z)}\errE^{B}\frac{1}{M(\zeta)}
    .
  \end{equation*}
  Lemma~\ref{lem:4} implies that the (linear) operator $\Bc_{z,\zeta}$ is invertible for all $(z,\zeta) \in \Omega\times \Omega$ with a uniformly bounded inverse, therefore,
  \begin{equation}
    \label{eq:164}
    G^{B}(z,\zeta)
    =
    \Bc_{z,\zeta}^{-1}\big[B \big] + \Bc_{z,\zeta}^{-1}\bigg[\frac{1}{M(z)}\errE^{B}\frac{1}{M(\zeta)}\bigg]
  \end{equation}
  for all $(z,\zeta)\in \Omega\times \Omega$.
  Finally \eqref{eq:122} and \eqref{eq:124} give the decomposition and the bounds in \eqref{eq:160} and \eqref{eq:161}.
\end{proof}

The bound in \eqref{eq:162} is obtained by directly applying Lemmas~\ref{lem:3} and \ref{lem:4} and serves as a preliminary estimate, which we will considerably improve in the next section.

\subsection{Proof of Theorem~\ref{thm:main} for $\beta = 2$}

\label{sec:proof-5}

In this section we complete the proof of Theorem~\ref{thm:main} for $\beta = 2$.
For this we first need to improve the error term bound in \eqref{eq:162}.
Indeed, after substituting \eqref{eq:160} and \eqref{eq:161} into \eqref{eq:52} and taking a derivative with respect to $\zeta$, the error term \eqref{eq:162} gives a function of order $O_{\prec}(1+N^{-1/2}\etaH^{\,-7/2})$ on $\Omega \times \Omega$.
As we will see below in the proof of Theorem~\ref{thm:main}, multiplying this function by $\frac{\partial \ft(z)}{\partial \overline{z}}  \frac{\partial \ft(\zeta)}{\partial \overline{\zeta}}$ and integrating over $\Omega\times \Omega$ improves the estimate by $\etaH^{\,2}$.
We end up with an overall bound $O_{\prec}(\etaH^{\,2}+N^{-1/2}\etaH^{\,-3/2})$, which is small only as long as $\etaH \gg N^{-1/3}$.
Therefore, in the following lemma we obtain an improved estimate of $\errR^{B}(z,\zeta)$ from \eqref{eq:161} it the case when $z$ and $\zeta$ belong to different complex half-planes.
The main idea in its proof consists of splitting the matrix $B$ into a projection onto a co-dimension 1 subspace on which the operator $\Bc_{z,\zeta}$ has a bounded inverse and its complement.
Consequently, both terms are treated separately. This approach has been first introduced in the context of Wigner matrices in \cite{CipoErdoSchr21}, and was later generalized to their multi-resolvents in \cite{CipoErdoSchr22}.
\begin{lem}[Improved estimate of $G^{B}(z,\zeta)$]
  \label{lem:6}
  Let $\gamma \in (0,1)$,  $\tau \in (0, \min \{\gamma/2, 1-\gamma\})$, and let $B \in \CC^{n\times n}$.
  Then uniformly for $(z,\zeta) \in \big(\Omega^{-}\times \Omega^{+} \big)\cup\big(\Omega^{+}\times \Omega^{-} \big)$
  \begin{equation}
    \label{eq:165}
    G^{B}(z,\zeta)
    =
    \vartheta \,\frac{2\imu }{z - \zeta } \frac{\la \Im M(E_{0}),\, B \, \ra}{\la \Im M(E_{0}) \ra}  \Im M(E_{0})  + \errR^{B}(z,\zeta)
  \end{equation}
  with
  \begin{equation}
    \label{eq:166}
    \EE\Big[ \big\| \errR^{B}(z,\zeta) \big\| \Big]
    =
    O_{\prec}\bigg(\big\| B \big\| \Big( 1 + \frac{1}{N \etaH^{\,2}} + \frac{1}{N^{1/2} \etaH^{\,3/2}} \Big) \bigg)
  \end{equation}
  and $\vartheta = 1$ for $(z,\zeta) \in \Omega^{+}\times \Omega^{-}$ and $\vartheta = -1$ for $(z,\zeta) \in \Omega^{-}\times \Omega^{+}$.
\end{lem}
\begin{proof}
  We split the proof into two steps by first considering the case $\la \Im M(E_0), B\ra = 0$, and then establishing \eqref{eq:165} for general $B \in \CC^{n\times n}$.
  As before, we denote $M_0 := \lim_{y\downarrow 0} M(E_0 + \imu y)$ for brevity.

  \textbf{Step 1.} Suppose that $\la \Im M_0, B\ra = 0$.
  Denote by $\Qc :\CC^{n\times n} \to \CC^{n\times n}$ the projection onto $\Im M_0$ given by
  \begin{equation}
    \label{eq:167}
    \Qc
    :=
    \frac{\la \Im M_0, \, \cdot \, \ra}{\| \Im M_0 \|^2_{\mathrm{HS}}} \Im M_0
    ,
  \end{equation}
  and notice that from \eqref{eq:124} we have that $(\Id - \Qc) \Bc_{z,\zeta}^{-1} = \Jc_{z,\zeta}$.
  Therefore, \eqref{eq:164}, \eqref{eq:124} and \eqref{eq:91} imply that
  \begin{equation}
    \label{eq:168}
    G^{B}(z,\zeta)
    =
    \Jc_{z,\zeta} \bigg[B + \frac{1}{M(z)}\errE^{B} \frac{1}{M(\zeta)}\bigg] + \Qc\Big[ G^{B}(z,\zeta)\Big]
    ,
  \end{equation}
  where the first term, after taking norm and expectation, is of order $O_{\prec}(\|B\|(1 + N^{-1/2} \etaH^{-3/2}) )$ by\eqref{eq:91}.
  In order to obtain an improved bound of the second term, we notice that
  \begin{equation*}
    \Qc\Big[ G^{B}(z,\zeta)\Big]
    =
    \frac{1}{\| \Im M_0 \|_{\mathrm{HS}}^{2}} \Big\la \Im M_0, G^{B}(z,\zeta) \Big\ra \Im M_0
    .
  \end{equation*}
  Using the cyclicity of the trace together with the definitions \eqref{eq:89} and \eqref{eq:51} we get
  \begin{equation*}
    \Big\la \Im M_0, G^{B}(z,\zeta) \Big\ra
    =
    \Big\la \Im M_0, \frac{1}{N} \sum_{k,l=1}^{N}G_{lk}(z)B G_{kl}(\zeta)\Big\ra
    =
    \Big\la  G^{\Im M_0}_{\overline{z}, \overline{\zeta}}, B \Big\ra
    .
  \end{equation*}
  Applying \eqref{eq:164} to $G^{\Im M_0}_{\overline{z}, \overline{\zeta}}$ gives
  \begin{align*}
    \Big\la  G^{\Im M_0}_{\overline{z}, \overline{\zeta}}, B \Big\ra
    &=
      \Big\la  \Bc_{\overline{z}, \overline{\zeta}}^{-1}\Big[\Im M_0 + \frac{1}{M(\overline{z})}\errE^{\Im M_0}\frac{1}{M(\overline{\zeta})}\Big], B \Big\ra
    \\
    &=
      \Big\la  \Im M_0 + \frac{1}{M(\overline{z})}\errE^{\Im M_0}\frac{1}{M(\overline{\zeta})}, \Big(\Bc_{\overline{z}, \overline{\zeta}}^{-1}\Big)^*\big[B\big] \Big\ra
      ,
  \end{align*}
  where $\EE[\|\errE^{\Im M_0}\|] = O_{\prec}(N^{-1/2}\etaH^{\,-3/2})$ due to \eqref{eq:91}.
  Since $\big\la \Im M_0, B \big \ra = 0$, the decomposition \eqref{eq:124} implies that $\Big(\Bc_{\overline{z}, \overline{\zeta}}^{-1}\Big)^*\big[B\big] = \Jc_{\overline{z},\overline{\zeta}}^{*}[B] = O(\|B\|)$, and we conclude that
  \begin{equation*}
    \Qc\Big[ G^{B}(z,\zeta)\Big]
    =
    O_{\prec}\Big(\|B\|\Big(1 + \frac{1}{N^{1/2}\etaH^{\, 3/2}}\Big)\Big)
    .
  \end{equation*}
  Combining this with the estimate of the first term in \eqref{eq:168} we get that 
  \begin{equation}
    \label{eq:173}
    \EE\Big[\Big\|G^{B}(z,\zeta)\Big\|\Big]
    =
    O_{\prec}\bigg( \big\| B \big\| \Big(1 + \frac{1}{N^{1/2} \etaH^{\,3/2}} \Big)\bigg)
  \end{equation}
  uniformly for $(z,\zeta) \in \big(\Omega^{-}\times \Omega^{+} \big)\cup\big(\Omega^{+}\times \Omega^{-} \big)$.
  
  \textbf{Step 2.} Consider now general $B \in \CC^{n\times n}$. 
  Denote
  \begin{equation*}
    B^{\circ}
    :=
    B - \frac{\la \Im M_0, B \ra}{\la \Im M_0 \ra} I_n
    .
  \end{equation*}
  This definition is reminiscent of the definition of $B^{\circ}$ from \cite{CipoErdoHenhSchr} and \cite{CipoErdoHenhKolu}, where it was used to control the 2-resolvent expression for hermitization of matrices with i.i.d. entries and  generally deformed Wigner matrices correspondingly. 
  Then $\la \Im M_0, B^{\circ} \ra = 0$ and
  \begin{equation}
    \label{eq:175}
    G^{B}(z,\zeta)
    =
    \frac{\la \Im M_0, B \ra}{\la \Im M_0 \ra } G^{I_n}(z,\zeta) + G^{B^{\circ}}(z,\zeta)
    .
  \end{equation}
  Using the resolvent identity we have
  \begin{equation*}
    G^{I_n}(z,\zeta)
    =
    \big(\Id_n \otimes \frac{1}{N}\Tr_N\big) \Big(\Gb(z) \Gb(\zeta)\Big)
    =
    \frac{1}{z - \zeta} \big(\Id_n \otimes \frac{1}{N}\Tr_N\big) \Big(\Gb(z) - \Gb(\zeta)\Big)
    .
  \end{equation*}
  After applying the local law \eqref{eq:33} we get
  \begin{equation*}
    G^{I_n}(z,\zeta)
    =
    \frac{1}{z - \zeta} \big(M(z) - M(\zeta)\big) + O_{\prec}\Big(\frac{1}{N\etaH^{\, 2}} \Big)
    .
  \end{equation*}
  Finally, the analyticity of $M$ implies that
  \begin{equation*}
    G^{I_n}(z,\zeta)
    =
    \vartheta\frac{2 \imu }{z - \zeta} \Im M_0 + O_{\prec}\Big(1 + \frac{1}{N\etaH^{\, 2}} \Big)
    .
  \end{equation*}
  Together with \eqref{eq:175} and \eqref{eq:173} this gives
  \begin{equation*}
    G^{B}(z,\zeta)
    =
    \vartheta\frac{2 \imu}{z - \zeta} \frac{\la \Im M_0, B \ra}{\la \Im M_0 \ra } \Im M_0 + O_{\prec}\Big(\|B\|\Big(1 + \frac{1}{N\etaH^{\, 2}}  + \frac{1}{N^{1/2}\etaH^{\, 3/2}}\Big) \Big)
    .
  \end{equation*}
  This finishes the proof for general $B\in \CC^{n\times n}$.
\end{proof}
Lemma~\ref{lem:5} improves the estimate of the error them in $G^{B}(z,\zeta)$ for $(z,\zeta) \in \big(\Omega^{-}\times \Omega^{+} \big)\cup\big(\Omega^{+}\times \Omega^{-} \big)$ by roughly $\etaH$, which is crucial for establishing Theorem~\ref{thm:main} for \emph{all} mesoscopic scales.
We now proceed to the proof of the main result.

\begin{proof}[Proof of Theorem~\ref{thm:main} for $\beta = 2$]
  We start by introducing additional notation.
  Let $B_{ij}: \CC \to \CC^{n\times n}$, $1\leq i,j \leq n$, be the collection of deterministic $n\times n$ matrix-valued functions
  \begin{equation*}
    B_{ij}(z)
    =
    \frac{1}{M(z)} \Bc_{z}^{-1}[I_n ] \, \Gamma [E_{ij} ]
    .
  \end{equation*}
  Consider the integral
  \begin{equation}
    \label{eq:181}
    \Vc
    :=
    \frac{1}{\pi^2} \int\displaylimits_{\Omega \times \Omega}  \frac{\partial \ft(z)}{\partial \overline{z}}  \frac{\partial \ft(\zeta)}{\partial \overline{\zeta}} \frac{\partial}{\partial \zeta} \EE\Big[  \sum_{i,j=1}^{n}  \Tr \Big(E_{ji} \, G^{B_{ij}(z)}(z,\zeta) \Big)  \efr(t) \Big] \dz \dzeta 
    .
  \end{equation}
  This is the integral that appears in \eqref{eq:52} with $S_{ij}(z,\zeta) = G^{B_{ij}(z)}(z,\zeta)$.
  Using the analyticity of $M$ at $E_0$ and the boundedness of $\Bc_{E_0}^{-1}$, we  write
  \begin{equation}
    \label{eq:182}
    B_{ij}(z)
    =
    B_{ij}(E_0) + \widehat{B}_{ij}(z)  
  \end{equation}
  with
  \begin{equation*}
    B_{ij}(E_0)
    :=
    \frac{1}{M(E_0)} \Bc_{E_0}^{-1}[I_n ] \, \Gamma [E_{ij} ]
  \end{equation*}
  and $\|\widehat{B}_{ij}(z)\| = O(|\Im z|)$.
  In particular, $\|B_{ij}(z)\|\lesssim 1$ uniformly on $\Omega$.

We split the domain of integration in \eqref{eq:181} into four regions determined by the signs of $\Im z$ and $\Im \zeta$, namely,
  \begin{equation}
    \label{eq:184}
    \Omega\times \Omega
    =
    \big(\Omega^{+}\times \Omega^{+}\big) \cup \big(\Omega^{+}\times \Omega^{-}\big) \cup \big(\Omega^{-}\times \Omega^{+}\big) \cup \big(\Omega^{-}\times \Omega^{-}\big)
    ,
  \end{equation}
  and denote the integrals over the corresponding four regions by $\Vc^{\,(+,+)}$, $\Vc^{\,(+,-)}$, $\Vc^{\,(-,+)}$ and $\Vc^{\,(-,-)}$, so that
  \begin{equation}
    \label{eq:185}
    \Vc
    =
    \Vc^{\,(+,+)} + \Vc^{\,(+,-)} + \Vc^{\,(-,+)} + \Vc^{\,(-,-)}
    .
  \end{equation}
  Now we treat each term separately.
  Consider first $\Vc^{\,(+,+)}$.
  If we denote
  \begin{equation*}
    h_1(z,\zeta)
    :=
    \EE\Big[  \sum_{i,j=1}^{n}  \Tr \Big(E_{ji} \, G^{B_{ij}(z)}(z,\zeta) \Big)  \efr(t) \Big]
    ,
  \end{equation*}
  then 
  \begin{equation}
    \label{eq:187}
    \Vc^{(+,+)}
    =
    \frac{1}{\pi^2}\int\displaylimits_{\Omega^{+} \times \Omega^{+}}  \frac{\partial \ft(z)}{\partial \overline{z}}  \frac{\partial \ft(\zeta)}{\partial \overline{\zeta}} \frac{\partial}{\partial \zeta} \,h_{1}(z,\zeta) \, \dz \dzeta \,.
  \end{equation}
  The estimate \eqref{eq:160} and the boundedness of $B_{ij}(z)$ implies that $\EE\big[\big\|G^{B_{ij}(z)}(z,\zeta)\big\|\big]\prec 1 + N^{-1/2}\etaH^{\,-3/2}$ uniformly on $\Omega^{+} \times \Omega^{+}$.
  Since $M^{-1}$ is uniformly bounded on $\Omega$ and $|\efr(t)|\leq 1$, we have that $h_1$ is an analytic function on $\Omega^{+} \times \Omega^{+}$  satisfying the bound
  \begin{equation}
    \label{eq:188}
    |h_1(z,\zeta)|
    \prec
    1 + \frac{1}{N^{1/2} \etaH^{\,3/2}}
    .
  \end{equation}
  The Cauchy integral formula applied to $\frac{\partial}{\partial \zeta} h_{1}(z,\zeta)$ yields the bound
  \begin{equation}
    \label{eq:189}
    \Big|\frac{\partial }{\partial \zeta} h_1(z,\zeta)\Big|
    \prec
    \frac{1}{\etaH} + \frac{1}{N^{1/2} \etaH^{\,5/2}}
  \end{equation}
  uniformly on $\Omega^{+} \times \Omega^{+}$.
  By using Stokes' theorem \eqref{eq:85} (twice) and the definition of $\ft$ in \eqref{eq:26}, we  rewrite the integral in \eqref{eq:187} as 
  \begin{equation*}
    \frac{1}{\pi^2}\int\displaylimits_{\Omega^{+} \times \Omega^{+}}  \frac{\partial \ft(z)}{\partial \overline{z}}  \frac{\partial \ft(\zeta)}{\partial \overline{\zeta}} \frac{\partial}{\partial \zeta} \,h_{1}(z,\zeta) \, \dz \dzeta 
    =
    \frac{1}{4 \pi^2} \int\displaylimits_{\partial\Omega^{+} \times \partial\Omega^{+}}   \ft(z)  \ft(\zeta) \frac{\partial}{\partial \zeta} \,h_{1}(z,\zeta) \, dz d\zeta 
    .
  \end{equation*}
  Since $\ft$ vanishes everywhere on $\partial \Omega^{+}$ except the part that intersects with the line $\Im z = N^{-\tau} \eta_0$, and $\frac{\partial}{\partial \zeta} h_{1}(z,\zeta)$ satisfies the bounds \eqref{eq:189}, we  estimate this integral as
  \begin{multline}
    \label{eq:191}
    \bigg|\int\displaylimits_{\partial\Omega^{+} \times \partial\Omega^{+}}   \ft(z)  \ft(\zeta) \frac{\partial}{\partial \zeta} \,h_{1}(z,\zeta) \, dz d\zeta \bigg|
    \\ \prec
    \int\displaylimits_{E_{0}-2\delta}^{E_{0} + 2\delta} \int\displaylimits_{E_{0}-2\delta}^{E_{0} + 2\delta}   \big|\ft(x_1+\imu N^{-\tau} \eta_{0})\ft(x_2+\imu N^{-\tau} \eta_{0})\big| \Big( \frac{N^{\tau}}{\eta_0} + \frac{N^{5\tau/2}}{N^{1/2}\eta_{0}^{5/2}} \Big) \,dx_1 dx_2
    .
  \end{multline}
  From the definition of $\ft$ in \eqref{eq:26} close the real line we know that $\ft(x+\imu N^{-\tau} \eta_0) = f(x) + \imu N^{-\tau }\eta_0 f'(x)$, and thus \eqref{eq:191} can be further estimated as
  \begin{multline}
    \label{eq:192}
    \bigg|\int\displaylimits_{\partial\Omega^{+} \times \partial\Omega^{+}}   \ft(z)  \ft(\zeta) \frac{\partial}{\partial \zeta} \,h_{1}(z,\zeta) \, dz d\zeta \bigg|
    \\
    \prec
    \big(\|f\|_1 + N^{-\tau}\eta_{0}\|f'\|_{1} \big)^{2} \Big( \frac{N^{\tau}}{\eta_0} + \frac{N^{5\tau/2}}{N^{1/2}\eta_{0}^{5/2}} \Big)
    \prec
    N^{-\tau}\eta_0 + \frac{N^{5\tau/2}}{(N\eta_{0})^{1/2}}
    ,
  \end{multline}
  where in the last step we used the norm bounds for $f$ and $f'$ from \eqref{eq:23}.
  In this case it is sufficient to have $\gamma \in (0,1)$ and  $\tau \in (0, \min \{(1-\gamma)/7, \gamma/2 \})$.
  Indeed, with this restriction on $\gamma$ and $\tau$ the last expression in \eqref{eq:192} gives $N^{-\gamma - \tau} + N^{(-1+\gamma + 5\tau)/2} \lesssim N^{-\tau}$, that ensures the convergence to zero as $N\to \infty$.
  The same argument can be applied to estimate the integral over the set $\Omega^{-} \times \Omega^{-}$, giving
  \begin{equation}
    \label{eq:193}
    \big| \Vc^{(+,+)} \big|
    +
    \big| \Vc^{(-,-)} \big|
    \prec
    N^{-\tau}
    .
  \end{equation}

  Consider now $\Vc^{(+,-)}$, the second term in \eqref{eq:185}
  \begin{equation}
    \label{eq:194}
    \Vc^{(+,-)}
    =
    \frac{1}{\pi^2} \int\displaylimits_{\Omega^{+} \times \Omega^{-}}  \frac{\partial \ft(z)}{\partial \overline{z}}  \frac{\partial \ft(\zeta)}{\partial \overline{\zeta}} \frac{\partial}{\partial \zeta} \EE\Big[  \sum_{i,j=1}^{n}  \Tr \Big(E_{ji} \, G^{B_{ij}(z)}(z,\zeta) \Big)  \efr(t) \Big] \dz \dzeta
    .
  \end{equation}
  Using the improved estimate \eqref{eq:165} for $G^{B(z)}(z,\zeta)$ from Lemma~\ref{lem:6} and the approximation of $B_{ij}(z)$ by $B_{ij}(E_0)$ on $\Omega \times \Omega$ from \eqref{eq:182}, we  rewrite the expectation in \eqref{eq:194} as
  \begin{multline}
    \label{eq:195}
    \frac{1}{\pi^2} \EE\Big[  \sum_{i,j=1}^{n}  \Tr \Big(E_{ji} \, G^{B_{ij}(z)}(z,\zeta) \Big)  \efr(t) \Big]
    \\
    =
     \frac{\EE\big[  \efr(t) \big]}{\pi^2 (z - \zeta )}  \sum_{i,j=1}^{n}  \Tr \Big(E_{ji} \, \frac{2\imu }{\la \Im M(E_{0}) \ra} \Im M(E_{0}) \la \Im M(E_{0}),\, B_{ij}(E_0) \, \ra  \Big)  + h_2(z,\zeta) 
    ,
  \end{multline}
  where $h_{2}(z,\zeta)$ is analytic on $\Omega^{+}\times \Omega^{-}$ and satisfies $|h_{2}(z,\zeta)|\prec (N^{\tau} + N^{-1/2}\etaH^{\,-3/2})$ uniformly on $\Omega^{+}\times \Omega^{-}$.
  The $N^{\tau}$ term in the last estimate comes from the bound
  \begin{equation}
    \label{eq:196}
    \Big| \frac{B_{ij}(z)- B_{ij}(E_0)}{z-\zeta} \Big|
    \prec
    N^{\tau},
  \end{equation}
  holding on $\Omega^{+}\times \Omega^{-}$.
  Repeating the computations in \eqref{eq:188}-\eqref{eq:192} we arrive at
  \begin{equation}
    \label{eq:197}
    \bigg|   \int\displaylimits_{\Omega^{+} \times \Omega^{-}}  \frac{\partial \ft(z)}{\partial \overline{z}}  \frac{\partial \ft(\zeta)}{\partial \overline{\zeta}} \frac{\partial}{\partial \zeta} \,h_{2}(z,\zeta) \, \dz \dzeta
    \bigg|
    \prec
    \eta_0 +  \frac{N^{5\tau/2}}{(N\eta_{0})^{1/2}}
    \lesssim
    N^{-\tau}
  \end{equation}
  for $\gamma \in (0,1)$ and  $\tau \in (0, \min \{(1-\gamma)/7, \gamma/2 \})$.

  We now simplify the expression in the first term in \eqref{eq:195}.
  Denote for brevity $M_0 := \lim_{y\downarrow 0} M(E_0 + \imu y)$ and $M_0' := \lim_{y\downarrow 0} M'(E_0 + \imu y)$.
  Then the $(z,\zeta)$-independent constant appearing in the first term of \eqref{eq:195} is written as
  \begin{align} \label{eq:phi+-}
    \phi^{\,(+,-)}
    &:=
      \sum_{i,j=1}^{n}  \Tr \Big(E_{ji} \, \frac{2\imu }{\la \Im M_{0} \ra}  \Im M_{0} \Big\la \Im M_{0},\, \frac{1}{M_{0}} \Bc_{E_0}^{-1}[I_n ] \, \Gamma [E_{ij} ] \, \Big\ra \Big)
    \\ \nonumber
    &=
      \frac{2\imu n }{\la \Im M_{0} \ra} \sum_{i,j=1}^{n} \Big\la \Im M_{0}\, \frac{1}{M_{0}} \Bc_{E_0}^{-1}[I_n ] \,  \Gamma [E_{ij} ] \, \Big\ra  \Big\la E_{ji} \,   \Im M_{0}  \Big\ra
    \\ \nonumber
    &
      =
      \frac{2\imu n }{\la \Im M_{0} \ra} \sum_{i,j=1}^{n} \Big\la \Gamma \Big[\Im M_{0}\, \frac{1}{M_{0}} \Bc_{E_0}^{-1}[I_n ] \,  \Big] E_{ij}  \, \Big\ra  \Big\la E_{ji} \,   \Im M_{0}  \Big\ra
    \\ \label{eq:198}
    &
      =
      \frac{2\imu  }{\la \Im M_{0} \ra}  \Big\la \Gamma \Big[\Im M_{0}\, \frac{1}{M_{0}} \Bc_{E_0}^{-1}[I_n ] \Big]  \,   \Im M_{0}  \Big\ra
      .
  \end{align}
  By differentiating the Dyson equation \eqref{eq:7}  at $z = E_0$ and taking the conjugate-transpose on both sides, we get that $\Bc_{E_0}\big[ M'_{0}\big] = I_n$, so that
  \begin{align}
    \label{eq:199}
    \Bc_{E_{0}}^{-1}[I_n ]
    =
    M'_0
    .
  \end{align}
  Plugging this into \eqref{eq:198} yields
  \begin{align}
    \nonumber
    &\Big\la \Gamma \Big[\Im M_{0}\, \frac{1}{M_{0}}M'_0 \Big]  \,   \Im M_{0}  \Big\ra
     =
      \Big\la \Im M_{0}\, \frac{1}{M_{0}} M'_0 \,   \Gamma \big[\Im M_{0}\big]  \Big\ra
    \\ \label{eq:200}
    &  \qquad\qquad\qquad\qquad
      =
      \Big\la \Im M_{0}\, \frac{1}{M_{0}}M'_0  \,   \frac{1}{M_0} \Im M_0 \frac{1}{(M_0)^*}  \Big\ra
      =
      \frac{\imu}{2} \big\la \Im M_0 \big\ra
      ,
  \end{align}
  where we used \eqref{eq:149} in the second  and  \eqref{eq:152} in the last step.  For $\phi^{\,(+,-)}$ from \eqref{eq:phi+-} we obtain  $ \phi^{\,(+,-)} = -1$.
  Now, using the analyticity of the function $(z-\zeta)^{-2}$ on $\Omega^{+}\times \Omega^{-}$ and Stokes' theorem we get
  \begin{align} \label{eq:201}
    &\int\displaylimits_{\Omega^{+} \times \Omega^{-}}\frac{\partial \ft(z)}{\partial \overline{z}}  \frac{\partial \ft(\zeta)}{\partial \overline{\zeta}} \frac{\partial}{\partial \zeta} \Big( \frac{\phi^{\,(+,-)}}{z-\zeta} \Big) \frac{\dz \dzeta}{\pi^2}
      =
      \int\displaylimits_{\Omega^{+} \times \Omega^{-}}\frac{\partial \ft(z)}{\partial \overline{z}}  \frac{\partial \ft(\zeta)}{\partial \overline{\zeta}} \frac{(-1)}{(z-\zeta)^2} \frac{\dz \dzeta}{\pi^2}
    \\ \nonumber
    &
      =
      \int\displaylimits_{\Omega^{+} \times \Omega^{-}}\frac{\partial }{\partial \overline{z}}  \frac{\partial }{\partial \overline{\zeta}} \frac{(\ft(z)- \ft(\zeta))^2}{(z-\zeta)^2} \frac{ \dz \dzeta}{2 \pi^2}
      =
      -\frac{1}{8 \pi^2}\int\displaylimits_{\partial \Omega^{+} \times \partial \Omega^{-}} \frac{(\ft(z)- \ft(\zeta))^2}{(z-\zeta)^2}  dz d\zeta
      .
  \end{align}
  The function $\ft$ vanishes everywhere on $\partial \Omega^{+} \times \partial \Omega^{-}$ except the part that intersects with the lines $\Im z = N^{-\tau} \eta_0$ and $\Im \zeta = N^{-\tau} \eta_0$, thus, the last integral in \eqref{eq:201} can be written as
  \begin{multline}
    -\frac{1}{8 \pi^2}\int_{\partial \Omega^{+} \times \partial \Omega^{-}} \frac{(\ft(z)- \ft(\zeta))^2}{(z-\zeta)^2}  dz d\zeta
    \\ \label{eq:203}
    =
    \frac{1}{8 \pi^2}\int\displaylimits_{E_0-2\delta}^{E_0+2\delta} \int\displaylimits_{E_0-2\delta}^{E_0+2\delta} \frac{(f(x) + \imu N^{-\tau} \eta_0 f'(x) - f(y) + \imu N^{-\tau} \eta_0 f'(y))^2}{(x-y+2\imu N^{-\tau}\eta_0)^2}  dx dy
    ,
  \end{multline}
  where the change of the sign is due to the change in the orientation of the contour $\partial \Omega^{-}$.
  Recall that $f(x) = g((x-E_0)/\eta_0)$.
  By changing the variables $s=(x-E_0)/\eta_0$, $t=(y-E_0)/\eta_0$, the last integral becomes
  \begin{align*}
    \frac{1}{8 \pi^2}\int_{\RR} \int_{\RR} \frac{(g(s) - g(t) + \imu N^{-\tau}( g'(s)   + g'(t))^2}{(s-t+2\imu N^{-\tau})^2}  ds dt
    .
  \end{align*}
  Combining this with \eqref{eq:195}, \eqref{eq:197}, \eqref{eq:phi+-} and \eqref{eq:201} we end up with the following expression for $\Vc^{(+,-)}$ from \eqref{eq:194}:
  \begin{equation}
    \label{eq:205}
    \Vc^{(+,-)}
    =
    \frac{\EE\big[ \efr(t) \big]}{8 \pi^2}\int_{\RR} \int_{\RR} \frac{(g(s) - g(t) + \imu N^{-\tau}( g'(s)   + g'(t))^2}{(s-t+2\imu N^{-\tau})^2}  ds dt + O_{\prec}(N^{-\tau})
    .
  \end{equation}

  The last term left to evaluate is
  \begin{equation*}
    \Vc^{(-,+)}
    =
    \frac{1}{\pi^2} \int\displaylimits_{\Omega^{-} \times \Omega^{+}}  \frac{\partial \ft(z)}{\partial \overline{z}}  \frac{\partial \ft(\zeta)}{\partial \overline{\zeta}} \frac{\partial}{\partial \zeta} \EE\Big[  \sum_{i,j=1}^{n}  \Tr \Big(E_{ji} \, G^{B_{ij}(z)}(z,\zeta) \Big)  \efr(t) \Big] \dz \dzeta
    .
  \end{equation*}
  Proceeding as in the analysis of $\Vc^{(+,-)}$ in \eqref{eq:195}-\eqref{eq:197} we have that
  \begin{align*}
    \Vc^{(-,+)}
    =
    \frac{E\big[\efr(t)\big]}{\pi^2}\int_{\Omega^{-} \times \Omega^{+}}\frac{\partial \ft(z)}{\partial \overline{z}}  \frac{\partial \ft(\zeta)}{\partial \overline{\zeta}} \frac{\partial}{\partial \zeta} \Big( \frac{\phi^{\,(-,+)}}{z-\zeta} \Big) \dz \dzeta
    +
    O_{\prec}(N^{-\tau})
    ,
  \end{align*}
  where
  \begin{equation}
    \label{eq:207}
    \phi^{\,(-,+)}
    =
    - \sum_{i,j=1}^{n}  \Tr \Big(E_{ji} \, \frac{2\imu }{\la \Im M_{0} \ra}  \Im M_{0} \Big\la \Im M_{0},\, \frac{1}{M_0^{*}}\big(\Bc_{E_{0}}^{-1}[I_n ]\big)^* \Gamma [E_{ij} ] \, \Big\ra \Big)
    .
  \end{equation}
  The expression in \eqref{eq:207} reflects the fact that $\vartheta =-1$ in \eqref{eq:124} and  $\lim_{z\to E_0}M(z) = (M(E_0))^{*}$ due to the analytic extension of $M(z)$ for  $z\in \CC_{-}$.
  Then from \eqref{eq:199} and \eqref{eq:200} we obtain that
  \begin{equation*}
    \Big\la \Gamma \Big[\Im M_{0}\, \frac{1}{(M_{0})^{*}}(M'_0)^{*} \Big]  \,   \Im M_{0}  \Big\ra
    =
    -\frac{\imu}{2} \big\la \Im M_0 \big\ra
    .
  \end{equation*}
  This cancels the minus sign coming from $\vartheta$, and therefore $\phi^{\,(-,+)} = -1$ and
  \begin{equation}
    \label{eq:209}
    \Vc^{(-,+)}
    =
    \frac{\EE\big[ \efr(t)\big]}{8 \pi^2}\int\displaylimits_{\RR} \int\displaylimits_{\RR} \frac{(g(s) - g(t) + \imu N^{-\tau}( g'(s)   + g'(t))^2}{(s-t-2\imu N^{-\tau})^2}  ds dt
    +
    O_{\prec}(N^{-\tau})
    .
  \end{equation}
  Combining Lemma~\ref{lem:2}, \eqref{eq:185}, \eqref{eq:192}, \eqref{eq:193}, \eqref{eq:205} and \eqref{eq:209} gives
  \begin{equation}
    \label{eq:210}
    \frac{d}{dt} \EE[\efr(t)]
    =
    - t \, V_N[g]\, \EE[\efr(t)] + O_{\prec}\big(N^{-\tau}(1 + |t|)\big)
    ,
  \end{equation}
  where
  \begin{equation*}
    V_N[g]
    :=
    \frac{1}{4 \pi^2}\int_{\RR} \int_{\RR} \frac{\big(g(s) - g(t) + \imu N^{-\tau}( g'(s)   + g'(t)\big)^2}{\big(s-t+2\imu N^{-\tau}\big)^2}  ds dt
    .
  \end{equation*}
  Notice that 
  \begin{equation*}
    \bigg| \frac{\big(g(s) - g(t) + \imu N^{-\tau}( g'(s)   + g'(t)\big)^2}{\big(s-t+2\imu N^{-\tau}\big)^2} \bigg|
    =
    \frac{|g(s) - g(t)|^2 }{|s-t|^2 + 4 N^{-2\tau}} +  N^{-2\tau}\frac{| g'(s)   + g'(t)|^2}{|s-t|^2 + 4 N^{-2\tau}}
    ,
  \end{equation*}
  so that
  \begin{equation*}
    \int_{\RR}\int_{\RR}  N^{-2\tau} \frac{ | g'(s)   + g'(t)|^2}{|s-t|^2 + 4 N^{-2\tau}} ds dt
    \lesssim
    \|g'\|_2^2
  \end{equation*}
  and
  \begin{equation*}
    \int_{\RR}\int_{\RR}  \frac{|g(s) - g(t)|^2 }{|s-t|^2 + 4 N^{-2\tau}}  ds dt
    \lesssim
    \int_{\RR}\int_{\RR}  \frac{(g(s) - g(t))^2 }{(s-t)^2}  ds dt
    <
    \infty
    ,
  \end{equation*}
  and therefore the sequence $(V_N[g])_{N \in \NN}$ is bounded.
  Now we complete the proof using a standard argument.
  If we denote $\varphi_N(t):= \EE[\efr(t)]$, then by \eqref{eq:210} and the boundedness of $V_N[g]$, the sequences of functions $(\varphi_N)$ and $(\varphi_N')$ are uniformly bounded on $[-a,a]$ for any $a>0$.
  By Arzelà-Ascoli theorem, any subsequence of $(\varphi_N)$ has a subsequence that converges uniformly on $[-a,a]$ to a function $\varphi$.
  From the dominated convergence theorem,  $V_N[g]$ converges to $ V[g]$ as $N \to \infty$, and thus by integrating \eqref{eq:210}, taking the limit over the convergent subsequence and switching the limit and integration we find that $\varphi(t) = e^{-V[g] t^2/2}$.
  Since $\varphi(t)$ is the unique limit for all the converging subsequences, the entire sequence $(\varphi_N)$ converges to $\varphi$, and we have that for any $t \in \RR$
  \begin{equation*}
    \lim_{N\to \infty}\EE[\efr(t)]
    =
    e^{-\frac{t^2}{2}V[g]}
    .
  \end{equation*}
  Finally, by Lemma~\ref{lem:1}, we have the convergence of the characteristic function
  \begin{equation*}
    \lim_{N\to \infty}\EE[e(t)]
    =
    e^{-\frac{t^2}{2}V[g]}
    ,
  \end{equation*}
  and we finish the proof of Theorem~\ref{thm:main} for $\beta = 2$ by invoking  Lévy's continuity theorem.  
\end{proof}

\section{Real symmetric case}
\label{sec:proof-real}

In this section we collect the changes to the proof presented in Sections~\ref{sec:proof} that establish Theorem~\ref{thm:main} for $\beta = 1$.
In this case the model \eqref{eq:1} is constructed  with real i.i.d. blocks $\Xb_{\alpha}$, $1\leq \alpha \leq d$.

\subsection{Results that do not require changes}
\label{sec:proof-r1}

First we notice that all the results about the spectral properties of $\Hb$ stated in Section~\ref{sec:proof-1} hold independently of the symmetry class.
In particular, if the self-energy operator
\begin{equation*}
  \Gamma[R] := \sum_{\alpha = 1}^{d} L_{\alpha} R \,  L_{\alpha}^{\, t} + L_{\alpha}^{\,t} R \, L_{\alpha}  
\end{equation*}
satisfies the property \textbf{(A)} \eqref{eq:10}, then the local laws \eqref{eq:32}-\eqref{eq:33} hold for $\Hb^{(1)}$. 
Lemma~\ref{lem:1} remains valid for $\beta = 1$ without any changes in the proof.

The proof of Lemma~\ref{lem:3} relies on the general properties of the resolvent \eqref{eq:96}-\eqref{eq:97} and the local laws \eqref{eq:32}-\eqref{eq:33}, and therefore holds independently of the symmetry class and can be used for $\beta = 1$ as stated.
The differences in the proof are purely notational.
Since $L_{\alpha}$ and $\Xb_{\alpha}$ are real symmetric, the expression for $W_{ij}$ \eqref{eq:93} is replaced by
\begin{equation*}
  W_{ij}
  =
  \sum_{\alpha=1}^{d}\bigg(L_{\alpha}  \, x^{(\alpha)}_{ij} + L_{\alpha}^{\,t}  \, x^{(\alpha)}_{ji}\bigg)
  ,
\end{equation*}
and the identity \eqref{eq:105} is replaced by its real counterpart
\begin{equation*}
  \EE[x_{kp}^{(\alpha_1)}x_{ql}^{(\alpha_2)}]
  =
  \delta_{\alpha_1 \alpha_2}\delta_{kq} \delta_{pl} \frac{1}{N}
  .
\end{equation*}
The above identities are applied in \eqref{eq:102}-\eqref{eq:104} and do not affect the proof.
The sum in \eqref{eq:102} is taken over $p,q\neq l$, and for any $R\in \CC^{n\times n} $ we have
\begin{equation*}
  \EE_{l}[ W_{lp} R W_{ql}]
  =
  \frac{1}{N} \Gamma[R] \, \delta_{pq}
  .
\end{equation*}
Therefore, we end up with the same expression as in the complex case in \eqref{eq:103}.
The remainder of the proof does not need any changes.

Lemma~\ref{lem:4} establishes the properties of the operator $\Bc_{z,\zeta}$ that are independent of the symmetry class of the Kronecker model.

\subsection{Computing $\EE[\efr(t)]$ for $\beta = 1$}
\label{sec:proof-r2}

The derivation and analysis of the approximate equation for $\frac{d}{dt} \EE[\efr(t)]$ in the case of the Kronecker model with real i.i.d. blocks requires several important modifications compared to the complex case.
These modifications are needed to take into account the differences in the correlation structures of real and complex Kronecker models.
In order to compare the two cases, consider the operator $\Scc^{(\beta)}: \CC^{nN \times nN} \to \CC^{nN \times nN}$ given by
\begin{equation*}
  \Scc^{(\beta)} \big[ \Rb \big]
  :=
  \EE \Big[\Wb^{(\beta)} \Rb  \,  \Wb^{(\beta)} \Big]
\end{equation*}
for $\Rb \in \CC^{nN \times nN}$ and $\beta \in \{1,2\}$.
This operator appears naturally in the proof of Lemma~\ref{lem:2}, where it takes the form \eqref{eq:44}
\begin{equation*}
  \Scc^{(2)}[\Rb]
  =
  \Scc [\Rb]
  =
  \Gamma \Big[ \frac{1}{N}\sum_{j=1}^{N} R_{j j} \Big] \otimes  \Ib_N
  .
\end{equation*}
On the other hand, in the case of real i.i.d. blocks we have
\begin{align}
  \label{eq:223}
  \Scc^{(1)}[\Rb]
  =
  \Gamma \Big[ \frac{1}{N}\sum_{j=1}^{N} R_{j j} \Big] \otimes  \Ib_N +  \frac{1}{N}  \sum_{i,j=1}^{N} \Gammat \big[ R_{j i} \big] \otimes  \Eb_{ij}
  ,
\end{align}
where the operator $\Gammat: \CC^{n\times n} \to \CC^{n\times n}$ is given by
\begin{equation*}
  \Gammat \big[R\big]
  :=
  \sum_{\alpha = 1}^{d} \Big( L_{\alpha} R L_{\alpha} + L_{\alpha}^{\,t} R L_{\alpha}^{\,t}\Big)
  .
\end{equation*}
Notice that in the case when $L_{\alpha} = L_{\alpha}^{\, t}$ for all $\alpha \in \{1,\ldots, d\}$ (i.e., when $\Hb^{(1)}$ is constructed using real Wigner blocks) the operators $\Gamma $ and $\Gammat$ coincide.
In the general case the operators $\Gamma$ and $\Gammat$ are  different.

Moreover, applying operator $\Scc^{(1)}$ in the derivation of the equation for $\frac{d}{dt} \EE[\efr(t)]$ gives rise to two types of multiresolvent averages
\begin{equation}
  \label{eq:225}
  G^{B}(z,\zeta)
  :=
  \frac{1}{N}\sum_{k,l = 1}^{N} G_{l k}(z) B G_{k l}(\zeta)
  ,\quad
  \Gt^{B}(z,\zeta)
  :=
  \frac{1}{N}\sum_{k,l = 1}^{N} G_{l k}(z) B G_{l k}(\zeta)
  .
\end{equation}
The details of how these terms emerge when differentiation $\EE[\efr(t)]$ with respect to $t$, are presented in Lemma~\ref{lem:7} below.
The first quantity $G^B$ from \eqref{eq:225} was studied in Section~\ref{sec:proof-3} for $\beta = 2$, and the results of Lemma~\ref{lem:3}, Corollary~\ref{cor:1} and Lemma~\ref{lem:6} hold for $\beta = 1$ without any changes.
The second quantity $ \Gt^{B}$ on the other hand requires a new proof which we present below.
Notice also, that if $n = 1$ (i.e., $\Hb^{(1)}$ is a real Wigner matrix), then $G_{kl} = G_{lk}$ and thus the two quantities in \eqref{eq:225} are the same.
Therefore, in the case of one real  Wigner matrix the variance of the fluctuations of the linear spectral statistics follows almost immediately from the complex case using the fact that $\Gammat = \Gamma$ in \eqref{eq:223} and $ G^{B}(z,\zeta) = \Gt^{B}(z,\zeta)$.

We start by proving the approximation of  $\Gt^{B}(z,\zeta)$, which is analogous to part (i) of Corollary~\ref{cor:1} and Lemma~\ref{lem:6} for the $\beta = 2$ case.
\begin{lem}
  \label{lem:5}
  Let $\gamma \in (0,1)$,  $\tau \in (0,\min\{\gamma/2,1-\gamma \})$, and let $B \in \CC^{n\times n}$.
  Denote $\etaH := \min\{|\Im z|, |\Im \zeta|\}$.
\item[(i)] Uniformly on $\big(\Omega^{+}\times \Omega^{+} \big)\cup\big(\Omega^{-}\times \Omega^{-} \big)$ 
  \begin{equation}
    \label{eq:226}
    \EE\Big[\big\|\Gt^{B}(z,\zeta)\big\|\Big]
    =
    O_{\prec}\bigg( \big\| B \big\| \Big(1 + \frac{1}{N^{1/2} \etaH^{\,3/2}} \Big)\bigg)
    .
  \end{equation}
\item[(ii)] Uniformly on $\big(\Omega^{-}\times \Omega^{+} \big)\cup\big(\Omega^{+}\times \Omega^{-} \big)$ 
  \begin{equation}
    \label{eq:227}
    \Gt^{B}(z,\zeta)
    =
    \vartheta \,\frac{2\imu }{\la \Im M(E_{0}) \ra} \frac{1}{z - \zeta} \Im M(E_{0}) B^{\, t} \Im M(E_{0})  
    +
    \errS^{B}(z,\zeta)
  \end{equation}
  with
  \begin{equation}
    \label{eq:228}
    \EE\Big[ \big\| \errS^{B}(z,\zeta) \big\| \Big]
    =
   O_{\prec}\bigg( \big\| B \big\| \Big(1 + \frac{1}{N \etaH^{\,2}} +  \frac{1}{N^{1/2} \etaH^{\,3/2}} \Big)  \bigg)
  \end{equation}
  and $\vartheta = 1$ for $(z,\zeta) \in \Omega^{+}\times \Omega^{-}$ and $\vartheta = -1$ for $(z,\zeta) \in \Omega^{-}\times \Omega^{+}$.
\end{lem}
\begin{proof}
  For convenience we  view  $n\times n$ matrices as elements of the vector space $\CC^{n}\otimes \CC^{n}$ with the two interpretations related by the isomorphism $\varphi : \CC^{n \times n} \to \CC^{n}\otimes \CC^{n}$ acting on the basis vectors $E_{ij} \in \CC^{n\times n}$ as $\varphi(E_{ij})  = e_{i}\otimes e_{j}$ for $i,j\in \{1,\ldots, n\}$.
  Then for any $A,B \in \CC^{n\times n}$ the linear operator $\CC^{n\times n} \ni R \mapsto A \, R \, B \in \CC^{n\times n}$ corresponds to the linear operator $A \otimes B^{\,t}: \CC^{n}\otimes \CC^{n} \to \CC^{n}\otimes \CC^{n}$ in the sense that
  \begin{equation}
    \label{eq:229}
    \varphi \big(A R B \big)
    =
    A \otimes B^{\,t}\, \varphi (R)
    ,
  \end{equation}
  where as before $\otimes$ denotes the tensor (Kronecker) product and $A \otimes B^{\,t}$ is an $n^2 \times n^2$ matrix.
  Using the above isomorphism we can write $\Gt^{B}(z,\zeta)$ as
  \begin{equation}
    \label{eq:230}
    \varphi\Big(\Gt^{B}(z,\zeta) \Big)
    =
    \Big[\frac{1}{N}\sum_{k,l=1}^{N} G_{lk}(z) \otimes (G_{lk}(\zeta))^{t}\Big] \varphi(B)
    .
  \end{equation}
  Notice that for all $k,l \in \{1,\ldots, N\}$ the symmetry of $\Wb$ implies that $W_{l k}^{\, t} = W_{kl}$, $\big(G_{ l k}(z) \big)^{t} = G_{ k l }(z)$,
  and thus $G_{lk}(z) \otimes (G_{lk}(\zeta))^{t} = G_{lk}(z) \otimes G_{kl}(\zeta)$.
  Moreover, the solution of the Dyson equation \eqref{eq:7} in the case of real symmetric matrices $L_{\alpha}$ is also real symmetric, $M(z) = (M(z))^{t}$.
  
  In the first part we derive an approximate equation for $\Gt^{B}(z,\zeta)$.
  The proof is similar to the proof of Lemma~\ref{lem:3} and relies on the general resolvent identities \eqref{eq:99}-\eqref{eq:100} and the local laws \eqref{eq:32}-\eqref{eq:33}.
  Fix $k\in \{1,\ldots, N\}$ and recall that $\etaH := \min\{|\Im z|, |\Im \zeta|\}$.
  From \eqref{eq:99}, \eqref{eq:100} and \eqref{eq:101} we have that for any $l\neq k$ uniformly on $(z,\zeta) \in \Omega \times \Omega$ 
  \begin{align*}
    G_{lk}(z) \otimes G_{kl}(\zeta)
    &=
      M(z) \sum_{p \neq l} W_{ l p } G^{(l)}_{ p k}(z)   \otimes \sum_{q \neq l}  G^{(l)}_{kq}(\zeta)  W_{ql } M(\zeta)    + O_{\prec}\Big(\frac{1 }{(N\etaH)^{\,3/2}}\Big)
      .
  \end{align*}
  By taking the partial expectation $\EE_{l}$ we get
  \begin{align*}
    &G_{lk}(z) \otimes G_{kl}(\zeta)
     =
      M(z) (1-\EE_{l})\Big[ G_{lk}(z) \otimes G_{kl}(\zeta) \Big] M(\zeta)
      + O_{\prec}\Big(\frac{1 }{(N\etaH)^{\,3/2}}\Big) 
    \\
    & \qquad \qquad
      + M(z)\frac{1}{N} \sum_{p \neq l} \sum_{\alpha = 1}^{d}  \Big[ L_{\alpha} G^{(l)}_{ p k}(z)   \otimes  G^{(l)}_{kp}(\zeta) L_{\alpha}^{\,t} +  L_{\alpha}^{\,t} G^{(l)}_{ p k}(z)   \otimes  G^{(l)}_{ kp}(\zeta) L_{\alpha} \Big] M(\zeta)
      .
  \end{align*}
  Applying \eqref{eq:107}, \eqref{eq:112} and summing the above equality over $l\in \{1,\ldots, N\}$ yields 
  \begin{align}
    \label{eq:234}
    &\sum_{l = 1}^{N} G_{lk}(z) \otimes G_{kl}(\zeta)
    =
      M(z) \otimes M(\zeta)
    \\ \nonumber
    & \qquad
    + M(z)\sum_{p =1}^{N} \sum_{\alpha = 1}^{d}  \Big[ L_{\alpha} G_{ p k}(z)   \otimes  G_{kp}(\zeta) L_{\alpha}^{\,t} +  L_{\alpha}^{\,t} G_{ p k}(z)   \otimes  G_{ kp}(\zeta) L_{\alpha} \Big] M(\zeta)
    \\ \nonumber
    & \qquad
      +  M(z) \sum_{l: l\neq k} (1-\EE_{l})\Big[ G_{lk}(z) \otimes G_{kl}(\zeta) \Big]M(\zeta)
    + O_{\prec}\Big(\frac{1 }{N^{1/2}\etaH^{\,3/2}}\Big)
    .
  \end{align}
  We multiply the left tensor factor in \eqref{eq:234} from the left by $M^{-1}(z)$ and the right tensor factor from the right by  $M^{-1}(\zeta)$.
  By averaging over $k\in \{1,\ldots, N\}$ and repeating the proof in Lemma~\ref{lem:3} we show that the error term 
  \begin{multline}
    \label{eq:235}
    \widetilde{\errE} (z,\zeta)
    :=
      \frac{1}{N}\sum_{k,l = 1}^{N} \frac{1}{M(z)}G_{lk}(z) \otimes G_{kl}(\zeta) \frac{1}{M(\zeta)}
    \\ 
      -
      I_{n} \otimes I_{n}
      - \frac{1}{N}\sum_{k,p =1}^{N} \sum_{\alpha = 1}^{d}  \Big[ L_{\alpha} G_{ p k}(z)   \otimes  G_{kp}(\zeta) L_{\alpha}^{\,t} +  L_{\alpha}^{\,t} G_{ p k}(z)   \otimes  G_{ kp}(\zeta) L_{\alpha} \Big]
  \end{multline}
  is analytic on $\Omega \times \Omega$ and satisfies the bound
  \begin{equation}
    \label{eq:236}
    \EE[\| \widetilde{\errE}(z,\zeta) \|]
    \prec
    \frac{1}{N^{1/2} \etaH^{\,3/2}}
    .
  \end{equation}

  Now  we  rewrite equation \eqref{eq:235} as
  \begin{equation}
    \label{eq:237}
    \Bch_{z,\zeta}\Big[\frac{1}{N}\sum_{k,l = 1}^{N} G_{l k}(z) \otimes G_{kl} (\zeta) \Big]
    =
    I_n \otimes I_n + \widetilde{\errE} (z,\zeta)
    ,
  \end{equation}
  where the linear operator $\Bch_{z,\zeta} : \CC^{n\times n}\otimes \CC^{n\times n} \to \CC^{n\times n} \otimes \CC^{n\times n}$ is given by
  \begin{equation*}
    \Bch_{z,\zeta}\big[A \otimes B \big]
    =
    \frac{1}{M(z)}A \otimes B \frac{1}{M(\zeta)} - \sum_{\alpha = 1}^{d} \Big[ L_{\alpha} A \otimes B L_{\alpha}^{\,t} + L_{\alpha}^{\,t} A \otimes B L_{\alpha} \Big]
  \end{equation*}
  for all $A\otimes B \in \CC^{n\times n} \otimes \CC^{n\times n}$.
 To show the invertibility of $\Bch_{z,\zeta}$ we define the linear operator $\Phi: \CC^{n\times n}\otimes \CC^{n\times n} \to \CC^{n\times n}\otimes \CC^{n\times n}$ acting on the basis vectors $E_{ij} \otimes E_{kl}$ as
  \begin{equation*}
    \Phi\big[E_{ij} \otimes E_{kl}\big]
    =
    E_{il} \otimes E_{kj}
    .
  \end{equation*}
  The operator $\Phi$ is an involution, $\Phi^2 = \Id_n \otimes \Id_n$.
  For any vectors $a,b,c,d \in \CC^{n}$ it satisfies $\Phi\big[ a b^{\,t} \otimes c d^{\,t} \big] = a d^{\,t} \otimes  c b^{\,t}$, which in composition with $\Bch_{z,\zeta}$ gives
  \begin{align*}
    \Phi \Bch_{z,\zeta}\big[ E_{ij} \otimes E_{kl} \big]
    &=
      \Phi \bigg[ \frac{1}{M(z)} E_{ij} \otimes E_{kl} \frac{1}{M(\zeta)} - \sum_{\alpha = 1}^{d} \Big[ L_{\alpha} E_{ij} \otimes E_{kl} L_{\alpha}^{\,t} + L_{\alpha}^{\,t} E_{ij} \otimes E_{kl} L_{\alpha} \Big] \bigg]
    \\
    &=
      \frac{1}{M(z)} E_{il} \frac{1}{M(\zeta)} \otimes E_{kj}  - \sum_{\alpha = 1}^{d} \Big[ L_{\alpha} E_{il} L_{\alpha}^{\,t}  \otimes E_{kj} + L_{\alpha}^{\,t} E_{il} L_{\alpha} \otimes E_{kj}  \Big]
    \\
    &
      =
      \big( \Bc_{z,\zeta} \otimes \Id \big) \Phi\Big[ E_{ij} \otimes E_{kl}\Big]
      .
  \end{align*}
  We see that $\Bch_{z,\zeta} = \Phi \big( \Bc_{z,\zeta} \otimes \Id \big) \Phi$, which means that the invertibility of $\Bch_{z,\zeta}$ is equivalent to the invertibility of $\Bc_{z,\zeta}$.
  By Lemma~\ref{lem:4}, the operator $\Bch_{z,\zeta}$ is invertible for $(z,\zeta) \in \Omega \times \Omega$ and
  \begin{equation}
    \label{eq:243}
    \Bch_{z,\zeta}^{-1}
    =
    \Phi \, \big( \Bc_{z,\zeta}^{-1} \otimes \Id \big) \, \Phi
    .
  \end{equation}
  Applying now $\Bch_{z,\zeta}^{-1}$ on both sides of \eqref{eq:237} yields 
  \begin{equation}
    \label{eq:244}
    \frac{1}{N}\sum_{k,l = 1}^{N} G_{l k}(z) \otimes G_{kl}(\zeta)
    =
    \Phi \big( \Bc_{z,\zeta}^{-1} \otimes \Id \big) \Phi \Big[I_n \otimes I_n\Big] +       \Phi \big( \Bc_{z,\zeta}^{-1} \otimes \Id \big) \Phi \Big[\widetilde{\errE} (z,\zeta)\Big]
    .
  \end{equation}
  After applying $\varphi^{-1}$ in \eqref{eq:230}, the estimate \eqref{eq:226} follows from the error bound \eqref{eq:236} and the boundedness of $\Bc_{z,\zeta}$ given in \eqref{eq:122}.

  The direct application of \eqref{eq:124} to \eqref{eq:244} results in an error term that is assymptotically too large on certain mesoscopic scales, and therefore is insufficient to prove the estimate \eqref{eq:227}-\eqref{eq:228}.
  In order to resolve this problem, consider the operator $\widehat{\Qc}: \CC^{n\times n} \otimes \CC^{n\times n} \to \CC^{n\times n} \otimes \CC^{n\times n}$ given by
\begin{equation*}
  \widehat{\Qc}
  :=
  \Phi \,\big(\Qc \otimes \Id_n \big) \, \Phi
  ,
\end{equation*}
where $\Qc$ was defined in \eqref{eq:167}.
Since $\Phi$ is an involution, we have
\begin{equation*}
  \Id_n \otimes \Id_n - \widehat{\Qc}
  =
  \Phi \, \big((\Id_n - \Qc) \otimes \Id_n \big) \, \Phi
  ,
\end{equation*}
which together with \eqref{eq:243} and \eqref{eq:124} implies that
\begin{equation}
  \label{eq:247}
  \Big(\Id_n \otimes \Id_n - \widehat{\Qc}\Big) \Bch_{z,\zeta}^{-1}
  =
  \Phi \, \big( \Jc_{z,\zeta}\otimes \Id_n \big) \, \Phi
  .
\end{equation}
Now, by using \eqref{eq:247} and \eqref{eq:237}, we decompose the left-hand side of \eqref{eq:244} as
\begin{align}
  \label{eq:248}
  \frac{1}{N}\sum_{k,l = 1}^{N} G_{l k}(z) \otimes G_{kl}(\zeta)
    & =
      \widehat{\Qc}\,\Big[ \frac{1}{N}\sum_{k,l = 1}^{N} G_{l k}(z) \otimes G_{kl}(\zeta) \Big]
  \\
    & \qquad
      + \Phi \big( \Jc_{z,\zeta}\otimes \Id_n \big) \Phi \,\Big[ I_n \otimes I_n + \widetilde{\errE}(z,\zeta) \Big]
    .
\end{align}
The last term in \eqref{eq:248}, after evaluating it at $\varphi(B)$, applying $\varphi^{-1}$ and taking the expectation, is of order $O_{\prec}(\| B \| (1 + N^{-1/2} \etaH^{\,-3/2}))$ and analytic on $(z,\zeta) \in (\Omega^{+}\times \Omega^{-})\cup (\Omega^{-}\times \Omega^{+})$.

It remains to estimate
\begin{equation}
  \label{eq:249}
  \varphi^{-1}\bigg(\widehat{\Qc}\,\Big[ \frac{1}{N}\sum_{k,l = 1}^{N} G_{l k}(z) \otimes G_{kl}(\zeta) \Big] \varphi\big(B \big) \bigg)
  .
\end{equation}
Denote $M_0 := \lim_{y \downarrow 0} M(E_0 + \imu y)$ for brevity, and denote by $\{\ell_{k}, 1 \leq k \leq n \}$ the collection of (non-normalized) eigenvectors of $\Im M_0$, so that $\Im M_0 = \sum_{k = 1}^{n} \ell_k \ell_k^{t}$.
  Here we used that $   \Im M_0 $ is real symmetric in the case $\beta=1$.
  Indeed, $\Gamma[R]^t = \Gamma[R^t]$ for any $R \in \CC^{n \times n}$, therefore,  $M(z)^t$ also  satisfies the  Dyson equation, and by the uniqueness of the solution to the Dyson equation $M(z) = M(z)^t$.
  For any $S = (s_{ij})\in \CC^{n\times n}$ and $T = (t_{\alpha \beta})\in \CC^{n\times n}$ we have
\begin{align*}
  \widehat{\Qc}\,\Big[ S \otimes T \Big]
  & =
    \sum_{i,j,\alpha,\beta = 1}^{n} s_{ij} t_{\alpha \beta} \, \Phi \, \big(\Qc \otimes \Id_n \big) \, \Phi\,\Big[ E_{ij} \otimes E_{\alpha \beta} \Big]
  \\
  & =
    \sum_{i,j,\alpha,\beta = 1}^{n} s_{ij} t_{\alpha \beta} \, \Phi \, \big(\Qc \otimes \Id_n \big)\,\Big[ E_{i \beta} \otimes E_{\alpha j} \Big]
  \\
  & =
    \frac{1}{\| \Im M_0 \|_{\textrm{HS}}^2}\sum_{i,j,\alpha,\beta, m = 1}^{n} s_{ij} t_{\alpha \beta} \, \la \Im M_0, E_{i \beta} \ra \, \Phi \,\Big(\ell_m \ell_m^{t} \otimes E_{\alpha j} \Big)
  \\
  & =
   \frac{1}{n \| \Im M_0 \|_{\textrm{HS}}^2} \sum_{i,j,\alpha,\beta,m} s_{ij} t_{\alpha \beta} \,  \Big(\ell_m e_j^{t}\otimes e_{\alpha} e_{\beta}^t\Im M_0 \, e_{i}  \, \ell_m^{t}  \Big)
  \\
  & =
   \frac{1}{n \| \Im M_0 \|_{\textrm{HS}}^2} \sum_{i,j,m} s_{ij}  \,  \Big(\ell_m e_j^{t}\otimes T \Im M_0 \, e_{i}  \, \ell_m^{t}  \Big)
    .
\end{align*}
After evaluating the above operator at $\varphi(B)$ and taking $\varphi^{-1}$, we get
\begin{align*}
  \varphi^{-1}\bigg(\widehat{\Qc}\,\Big[ S \otimes T \Big] \varphi\big(B \big) \bigg)
  & =
    \frac{1}{n \| \Im M_0 \|_{\textrm{HS}}^2} \sum_{i,j,m = 1}^{n}  s_{ij}  \, \varphi^{-1}\bigg(  \,  \Big(\ell_m e_j^{t}\otimes T \Im M_0 e_{i}  \, \ell_m^{t}  \Big) \varphi\big(B \big) \bigg)
  \\
  & =
    \frac{1}{n \| \Im M_0 \|_{\textrm{HS}}^2} \sum_{i,j,m = 1}^{n}  s_{ij}  \, \ell_m e_j^{t} B   \ell_m e_{i}^{t} \Im M_0 T^t
  \\
  & =
   \frac{1}{n \| \Im M_0 \|_{\textrm{HS}}^2} \sum_{i,j,m = 1}^{n}  s_{ij}  \, \ell_m  \ell_m^{t} B^{t} e_j   e_{i}^{t} \Im M_0 T^t
  \\
  &=
    \frac{\Im M_0 \, B^t \, S^t \, \Im M_0 \,T^t}{ n \| \Im M_0 \|_{\textrm{HS}}^2} 
    ,
\end{align*}
where in the second step we used the definition of $\varphi$ in \eqref{eq:229}.

Now we apply the above formula with $S = G_{lk}(z)$ and $T = G_{kl}(\zeta)$, and use that $(G_{kl})^{t} = G_{lk}$ together with the linearity of $\varphi$ and $\widehat{\Qc}$, to rewrite \eqref{eq:249} as
\begin{align*}
  \varphi^{-1}\bigg(\widehat{\Qc}\,\Big[ \frac{1}{N}\sum_{k,l = 1}^{N} G_{l k}(z) \otimes G_{kl}(\zeta) \Big] \varphi\big(B \big) \bigg)
  & =
    \frac{\Im M_0}{n \| \Im M_0 \|_{\textrm{HS}}^2}  \,  B^{t} \frac{1}{N}\sum_{k,l = 1}^{N} G_{kl}(z) \Im M_0  G_{lk}(\zeta)
  \\
  & =
    \frac{\Im M_0}{n \| \Im M_0 \|_{\textrm{HS}}^2}  \,  B^{t} \,  G^{\Im M_0} (z,\zeta) 
  .
\end{align*}
Lemma~\ref{lem:6} implies that
\begin{equation*}
  \frac{1}{n \| \Im M_0 \|_{\textrm{HS}}^2} \Im M_0 \,  B^{t} \,  G^{\Im M_0} (z,\zeta) 
  =
  \vartheta \,\frac{2\imu }{z - \zeta } \frac{\Im M_0 B^{\,t} \Im M_{0} }{n \la \Im M_{0} \ra}    + \errR^{B}(z,\zeta)
  ,
\end{equation*}
  where $\EE[ \| \errR^{B}(z,\zeta) \| ] = O_{\prec}\Big( 1 + N^{-1} \etaH^{\, -2} + \| B \| \big(1 + N^{-1/2} \etaH^{\,-3/2} \big)  \Big)$, and $\vartheta = 1$ for $(z,\zeta) \in \Omega^{+}\times \Omega^{-}$ and $\vartheta = -1$ for $(z,\zeta) \in \Omega^{-}\times \Omega^{+}$.
  This, together with the estimate of the last term in \eqref{eq:248}, completes the proof of the lemma.
\end{proof}
Now we state and prove the analog of Lemma~\ref{lem:2} for $\beta = 1$.
\begin{lem}
  \label{lem:7}
  Let $\gamma \in (0,1)$ and $\tau \in \big(0, \min\{\gamma, (1-\gamma)\}/7 \big)$.
  Then 
  \begin{equation}
    \label{eq:253}
    \frac{d}{dt} \EE[\efr(t)]
    =
    - \frac{t}{\pi^2} \int\displaylimits_{\Omega \times \Omega}  \frac{\partial \ft(z)}{\partial \overline{z}}  \frac{\partial \ft(\zeta)}{\partial \overline{\zeta}} \frac{\partial}{\partial \zeta} \EE\Big[  \sum_{i,j=1}^{n}  \Tr \Big(E_{ji} \,(S_{ij} + \ST_{ij}) \Big)  \efr(t) \Big] \dzeta  \dz + \errT
  \end{equation}
  where $|\errT |\prec N^{-\tau}$, and
  \begin{align}
    \label{eq:254}
    S_{ij}(z,\zeta)
    :=
    \frac{1}{N} \sum_{k,l = 1}^{N}  G_{l k}(z)  \frac{1}{M(z)} \Bc_{z}^{-1}[I_n ] \, \Gamma [E_{ij} ]   G_{k l}(\zeta)
    ,
    \\
    \label{eq:255}
    \ST_{ij}(z,\zeta)
    :=
    \frac{1}{N} \sum_{k,l = 1}^{N}  G_{l k}(z)  \frac{1}{M(z)} \Bc_{z}^{-1}[I_n ] \, \Gammat [E_{ij} ]   G_{lk}(\zeta)
    .
  \end{align}
\end{lem} 
\begin{proof}
  Recall that the operator $\Scc^{(1)}$ was defined in \eqref{eq:223}.
  We keep the notation $\Scc[\Rb] = \Gamma \Big[ \frac{1}{N}\sum_{j=1}^{N} R_{j j} \Big] \otimes  \Ib_N$ as in Section~\ref{sec:proof}, and denote $\Scct[\Rb]:= \frac{1}{N}\sum_{i,j = 1}^{N} \Gammat[R_{ji}]\otimes \Eb_{ij}$, so that
  \begin{equation}
    \label{eq:256}
    \Scc^{(1)}
    =
    \Scc + \Scct
    .
  \end{equation}
  The lines \eqref{eq:54}-\eqref{eq:60} are repeated exactly as in the proof of Lemma~\ref{lem:2} with $\Scc$ replaced by $\Scc^{(1)}$.
  Then  we apply the decomposition \eqref{eq:256} on the right-hand side of
  \eqref{eq:62} and treat the part containing $\Scct$ as an additional error term.
  We end up with the equation
  \begin{align*}
    -\EE\Big[\efr(t)\big(1-\EE\big)\big[\Scc^{(1)}[\Gb(z)]\Gb(z)\big] \Big]
    =
      \EE\Big[\efr(t)\Big(\frac{1}{\Mb(z)} + z\Ib_{nN} - \Kb_{0}\Big) \big(1-\EE\big)\Gb(z) \Big]
    \\
     -\EE\Big[\efr(t)\big(1-\EE\big)\big[\Scc[\Gb(z)]\Mb(z)\big] \Big]
      + \errH(z,\zeta) + \errHt(z,\zeta)
      ,
  \end{align*}
  where $\errH$ is defined in \eqref{eq:63} and the second error term above is given by
  \begin{equation*}
    \errHt(z,\zeta)
    :=
    -
    \EE\Big[\efr(t)\big(1-\EE\big)\big[\Scct[\Gb(z)]\Gb(z)\big] \Big]
    .
  \end{equation*}
  After using \eqref{eq:59} (with $\Scc^{(1)}$ instead of $\Scc$) and \eqref{eq:55} we obtain the same equation as in \eqref{eq:65} with an additional error term
  \begin{equation}
    \label{eq:259}
    \Bcc_z \, \EE\Big[\efr(t)\big(1-\EE\big)\big[\Gb\big] \Big]
    =
    -\EE\Big[ \widetilde{\Wb}  \Gb \, \nabla_{\widetilde{\Wb}} \big(\efr(t)\big)\Big] \frac{1}{\Mb }
    - \Big(\errC + \errH + \errHt  \Big) \frac{1}{\Mb }
    .
  \end{equation}
  We see that the operator $\Scct$ does not explicitly appear in leading expressions of \eqref{eq:259} anymore and the terms have exactly the same form as in \eqref{eq:65}.
  Therefore, we repeat \eqref{eq:65}-\eqref{eq:76} in the proof of Lemma~\ref{lem:2} line by line with the error term $\errH$ replaced by $\errH + \errHt$.
  The operator $\Scc^{(1)}$  reappears again when we repeat the computations in \eqref{eq:77}-\eqref{eq:78}.
  After taking the partial expectation with respect to $\Wbt$, the operator $\Scc$ in \eqref{eq:78} is replaced by $\Scc^{(1)} = \Scc + \Scct$.
  More precisely, we have 
  \begin{align*}
    &\EE\Big[ \Tr \Big( \Bcc_{z}^{-1}\Big[ \widetilde{\Wb} \Gb(z)  \frac{1}{\Mb(z)} \Big] \Big) \Tr \Big(\widetilde{\Wb} \Gb(\zeta) \Big)  \efr(t) \Big]
    \\
    &
      =\EE\Big[ \sum_{i,j=1}^{n}\sum_{p,q=1}^{N} \Tr \Big((E_{ji}\otimes \Eb_{q p}) \Gb(z)  \frac{1}{\Mb(z)} \big((\Bcc_{z}^{*})^{-1}[\Ib_{nN} ]\big)^*  \Scc^{(1)}[E_{ij}\otimes \Eb_{p q} ] \,\Gb(\zeta) \Big)  \efr(t) \Big]
      .
  \end{align*}
  By applying the decomposition \eqref{eq:256} and using \eqref{eq:80}) we rewrite the term containing $\Scc$ as
  \begin{align*}
    \label{eq:261}
    &\EE\Big[ \sum_{i,j=1}^{n}\sum_{p,q=1}^{N} \Tr \Big((E_{ji}\otimes \Eb_{q p}) \Gb(z)  \frac{1}{\Mb(z)} \big((\Bcc_{z}^{*})^{-1}[\Ib_{nN} ]\big)^*  \Scc[E_{ij}\otimes \Eb_{p q} ] \,\Gb(\zeta) \Big)  \efr(t) \Big]
    \\
    &\qquad\qquad\qquad\qquad
    =
    \EE\Big[  \sum_{i,j=1}^{n} \sum_{k,l = 1}^{N}  \Tr \Big(E_{ji} G_{l k}(z)  \frac{1}{M(z)} \Bc_{z}^{-1}[I_n ]   \Gamma [E_{ij} ]   G_{k l}(\zeta) \Big)  \efr(t) \frac{1}{N}\Big]
    ,
  \end{align*}
  which gives $S_{ij}(z,\zeta)$ in \eqref{eq:253}. 
  In order to obtain the term in \eqref{eq:253} that contains $\ST_{ij}(z,\zeta)$,  we notice that 
  \begin{equation*}
    \Scct\big[ E_{ij}\otimes \Eb_{pq} \big]
    =
    \frac{1}{N} \Gammat[E_{ij}]\otimes \Eb_{qp}
  \end{equation*}
  for any $E_{ij}\otimes \Eb_{pq}$.
  Therefore,
  \begin{align*}
    &\EE\Big[ \sum_{i,j=1}^{n}\sum_{p,q=1}^{N} \Tr \Big((E_{ji}\otimes \Eb_{q p}) \Gb(z)  \frac{1}{\Mb(z)} \big((\Bcc_{z}^{*})^{-1}[\Ib_{nN} ]\big)^*  \Scct[E_{ij}\otimes \Eb_{p q} ] \,\Gb(\zeta) \Big)  \efr(t) \Big]
    \\
    &
      = \EE\Big[ \frac{\efr(t) }{N} \sum_{i,j=1}^{n}\sum_{p,q=1}^{N} \Tr \Big((E_{ji}\otimes \Eb_{q p}) \Gb(z)  \frac{1}{\Mb(z)} \big((\Bcc_{z}^{*})^{-1}[\Ib_{nN} ]\big)^*  \big(\,\Gammat[E_{ij}]\otimes \Eb_{qp}\big) \,\Gb(\zeta) \Big)  \Big]
    \\
    &
      =
      \EE\Big[ \frac{\efr(t) }{N} \sum_{i,j=1}^{n} \sum_{k,l = 1}^{N}  \Tr \Big(E_{ji} G_{l k}(z)  \frac{1}{M(z)} \Bc_{z}^{-1}[I_n ]   \Gammat [E_{ij} ]   G_{lk}(\zeta) \Big)  \Big]
      ,
  \end{align*}
  which gives rise to the summand containing $\ST_{ij}(z,\zeta)$ from \eqref{eq:255} in \eqref{eq:253}.

  It remains to estimate the error term $\errT = \errK + \errKt$, where $\errKt$ is defined through \eqref{eq:74} and \eqref{eq:67} with $\errH$ replaced by $\errHt$.
  The error term $\errK$ satisfies the same bound as in Lemma~\ref{lem:2}, namely $|\errK| \prec N^{-\tau}$.
  For $\errKt$ we see that
  \begin{align*}
    \label{eq:264}
    &(\Id_n\otimes \Tr_N) \Big[\Scct[\Gb(z)]\Gb(z)\Big]
    =
      \frac{1}{N}\sum_{k,l=1}^{N} \Gammat[G_{lk}(z)] G_{lk}(z)
    \\
    & \qquad \qquad \qquad \qquad
      =
      \sum_{\alpha = 1}^{d} \Big(L_{\alpha}\frac{1}{N}\sum_{k,l=1}^{N} G_{lk}(z) L_{\alpha} G_{lk}(z) + L_{\alpha}^{\, t}\frac{1}{N}\sum_{k,l=1}^{N} G_{lk}(z) L_{\alpha}^{\, t} G_{lk}(z) \Big)
    \\
    & \qquad \qquad \qquad \qquad
      =
      \sum_{\alpha = 1}^{d} \Big(L_{\alpha}\, \Gt^{L_{\alpha}}(z,z) + L_{\alpha}^{\, t}\, \Gt^{L_{\alpha}^{t}}(z,z) \Big)
      .
  \end{align*}
  From the first statement in Lemma~\ref{lem:5} applied to $\Gt^{L_{\alpha}}(z,z)$ and $\Gt^{L^{t}_{\alpha}}(z,z)$ we have
  \begin{equation*}
    \EE\bigg[\Big\|(\Id_n\otimes \Tr_N) \Big[\Scct[\Gb(z)]\Gb(z)\Big] \Big\|\bigg]
    =
    O_{\prec}\Big(1 + \frac{1}{N^{1/2} |\Im z|^{\,3/2}}\Big)
    ,
  \end{equation*}
  and thus (see \eqref{eq:67})
  \begin{equation*}
    \Big\|\Bc_{z}^{-1} (\Id_n\otimes \Tr_N) \Big[  \errHt  \frac{1}{\Mb(z) } \Big]\Big\|
    =
    O_{\prec}\Big(1 + \frac{1}{N^{1/2}|\Im z|^{3/2}}\Big)
  \end{equation*}
  uniformly on $\Omega \times \Omega$.
  Finally, integrating the trace of the above expression as in \eqref{eq:86} yields the estimate
  \begin{equation*}
    |\errKt|
    \prec
    N^{3\tau/2}\big(\|g\|_{1} + \|g'\|_{1}\big)\frac{1}{(N \eta_0)^{1/2}}
    .
  \end{equation*}
  Choosing $\tau$ such that $1-\gamma > 7 \tau$ ensures that both $|\errK|$ and $|\errKt|$ are bounded by $N^{-\tau}$.    
\end{proof}

\subsection{Proof of Theorem~\ref{thm:main} for $\beta = 1$}
\label{sec:proof-r3}

With Lemma~\ref{lem:7} established, we proceed with the proof of Theorem~\ref{thm:main} for $\beta = 1$.
Consider the integral
\begin{equation}
  \label{eq:268}
  \Vc
  =
  \frac{1}{\pi^2}\int_{\Omega \times \Omega}  \frac{\partial \ft(z)}{\partial \overline{z}}  \frac{\partial \ft(\zeta)}{\partial \overline{\zeta}} \frac{\partial}{\partial \zeta} \EE\Big[  \sum_{i,j=1}^{n}  \Tr \Big(E_{ji} \,S_{ij}(z,\zeta) + E_{ji}\ST_{ij}(z,\zeta) \Big)  \efr(t) \Big] \dzeta  \dz
\end{equation}
appearing in \eqref{eq:253} in Lemma~\ref{lem:7} with $S_{ij}(z,\zeta)$ and $\ST_{ij}(z,\zeta)$ defined in \eqref{eq:254} and \eqref{eq:255}.
Using the notation introduced in \eqref{eq:225} we can write $S_{ij}(z,\zeta)$ and $\ST_{ij}(z,\zeta)$ in \eqref{eq:268} as
\begin{equation}
  \label{eq:269}
  S_{ij}(z,\zeta)
  =
  G^{B_{ij}(z)}(z,\zeta)
  ,\quad
  \ST_{ij}(z,\zeta)
  =
  \Gt^{\BT_{ij}(z)}(z,\zeta)
\end{equation}
with
\begin{equation*}
  B_{ij}(z)
  :=
  \frac{1}{M(z)} \Bc_{z}^{-1}[I_n ]  \, \Gamma [E_{ij} ]
  ,\quad
  \BT_{ij}(z)
  :=
  \frac{1}{M(z)} \Bc_{z}^{-1}[I_n ]  \, \Gammat [E_{ij} ]
  .
\end{equation*}
Similarly as in \eqref{eq:182}, we have
\begin{equation}
  \label{eq:271}
      B_{ij}(z)
    =
    B_{ij}(E_0) + O(|\Im z|)  
    ,\qquad
        \BT_{ij}(z)
    =
    \BT_{ij}(E_0) + O(|\Im z|)
    .
  \end{equation}
 It follows from the argument in \eqref{eq:188}-\eqref{eq:192} that if for some $\dagger, \star \in \{-, +\}$ the function $h$ is analytic on $\Omega^{\dagger}\times \Omega^{\star}$ and satisfies on this set the bound $|h(z,\zeta)| \prec |\Im z|^{-\alpha} |\Im \zeta |^{-\beta}$, then 
\begin{equation}
  \label{eq:272}
  \bigg|\frac{1}{\pi^2}\int_{\Omega^{\dagger} \times \Omega^{\star}}  \frac{\partial \ft(z)}{\partial \overline{z}}  \frac{\partial \ft(\zeta)}{\partial \overline{\zeta}} \frac{\partial}{\partial \zeta} h(z,\zeta) \dzeta  \dz \, \bigg|
  \prec
  \frac{N^{(\alpha + \beta + 1)\tau}}{\eta_0^{\alpha + \beta - 1}}
  .
\end{equation}

Now, following \eqref{eq:184}-\eqref{eq:185}, we split the integral over $\Omega \times \Omega$ into four parts
\begin{equation*}
      \Vc
    =
    \Vc^{\,(+,+)} + \Vc^{\,(+,-)} + \Vc^{\,(-,+)} + \Vc^{\,(-,-)}
\end{equation*}
determined by the signs of $\Im z$ and $\Im \zeta$. 
For $\efr = \efr(t)$, denote
\begin{equation*}
  h_1(z,\zeta)
    :=
    \EE\Big[  \sum_{i,j=1}^{n}  \Tr \Big(E_{ji}  G^{B_{ij}(z)}(z,\zeta) \Big)  \efr \Big]
    ,  \,\,\,
  \tilde{h}_1(z,\zeta)
    :=
    \EE\Big[  \sum_{i,j=1}^{n}  \Tr \Big(E_{ji} \, \Gt^{\BT_{ij}(z)}(z,\zeta) \Big)  \efr \Big]
    .
  \end{equation*}
  Using the representation \eqref{eq:269} together with the bounds \eqref{eq:160}, \eqref{eq:226} and \eqref{eq:271} we obtain that
  \begin{equation*}
    |h_1(z,\zeta)| + |h_2(z,\zeta)|
    \prec
    1 + \frac{1}{N^{1/2} \etaH^{\,3/2}}
  \end{equation*}
  uniformly on $(z,\zeta) \in \big(\Omega^{+}\times \Omega^{+}\big)\cup \big(\Omega^{-}\times \Omega^{-}\big)$.
  Then \eqref{eq:272} implies that
  \begin{equation*}
    \big|\Vc^{\,(+,+)} \big| + \big|\Vc^{\,(-,-)} \big|
    \prec
    N^{\tau} \eta_0 + \frac{N^{5/2 \tau}}{(N \eta_0)^{1/2}}
    .
  \end{equation*}
  Taking $\gamma \in (0,1)$ and  $\tau \in (0, \min \{(1-\gamma)/7, \gamma/2 \})$ ensures that $\big|\Vc^{\,(+,+)} \big| + \big|\Vc^{\,(-,-)} \big| \prec N^{-\tau}$.

  We compute $\Vc^{\,(+,-)}$ and $\Vc^{\,(-,+)}$ in three steps.
  First, we split the integrands in \eqref{eq:268} with the identities \eqref{eq:269} into leading and error terms using the approximations \eqref{eq:165}-\eqref{eq:166} and \eqref{eq:227}-\eqref{eq:228}.
  The integrals of the \emph{error} terms that arise from these decompositions are treated analogously to $\Vc^{\,(+,+)}$ and $\Vc^{\,(-,-)}$.
  Indeed, if we replace $h(z,\zeta)$ in \eqref{eq:272} with the error terms from \eqref{eq:165}-\eqref{eq:166} and \eqref{eq:227}-\eqref{eq:228}, and take $\gamma \in (0,1)$ and  $\tau \in (0, \min \{(1-\gamma)/7, \gamma/2 \})$, then the resulting integrals are of order $O_{\prec}(N^{-\tau})$.
  Therefore, we can replace $S_{ij}(z,\zeta)$ and $\ST_{ij}(z,\zeta)$ in \eqref{eq:268} by their deterministic approximations defined in \eqref{eq:165} and \eqref{eq:227}.

  In the second step, exactly as in \eqref{eq:196}-\eqref{eq:197}, we use \eqref{eq:271} to replace $B_{ij}(z)$ and $\BT_{ij}(z)$ in the deterministic approximations of $S_{ij}(z,\zeta)$ and $\ST_{ij}(z,\zeta)$ with $B_{ij}(E_0)$ and $\BT_{ij}(E_0)$ correspondingly.
  After integrating, under the assumption that $\gamma \in (0,1)$ and  $\tau \in (0, \min \{(1-\gamma)/7, \gamma/2 \})$, this replacement gives a term of size $O_{\prec}(N^{-\tau})$.

  After applying the simplifications from the first two steps, we arrive at
  \begin{align*} 
    \Vc^{\,(+,-)}
    &=
      \frac{\EE\big[ \efr(t) \big]}{\pi^2}\int\displaylimits_{\Omega^{+} \times \Omega^{-}}  \frac{\partial \ft(z)}{\partial \overline{z}}  \frac{\partial \ft(\zeta)}{\partial \overline{\zeta}} \frac{\partial}{\partial \zeta} \frac{1}{z - \zeta} \Big( \phi^{\,(+,-)} + \tilde{\phi}^{\,(+,-)} \Big) \, \dzeta  \dz + O_{\prec}(N^{-\tau})
      ,
    \\ 
    \Vc^{\,(-,+)}
    &=
      \frac{\EE\big[ \efr(t) \big]}{\pi^2}\int\displaylimits_{\Omega^{-} \times \Omega^{+}}  \frac{\partial \ft(z)}{\partial \overline{z}}  \frac{\partial \ft(\zeta)}{\partial \overline{\zeta}} \frac{\partial}{\partial \zeta} \frac{1}{z - \zeta} \, \Big( \phi^{\,(-,+)} + \tilde{\phi}^{\,(-,+)} \Big)   \dzeta  \dz + O_{\prec}(N^{-\tau})
      ,
  \end{align*}
  where we introduce the $(z,\zeta)$-independent constants
  \begin{align}
    \label{eq:280}
    \phi^{\,(+,-)}
    &=
      \frac{2\imu }{\la \Im M_{0} \ra}  \sum_{i,j=1}^{n} \Tr \Big(E_{ji} \,\Big\la \Im M_{0},\, \frac{1}{M_0} M_0'  \, \Gamma [E_{ij} ] \, \Big\ra  \Im M_{0}\Big)
      ,
    \\ \label{eq:281}
      \tilde{\phi}^{\,(+,-)}
    &=
      \frac{2\imu }{n \la \Im M_{0} \ra}  \sum_{i,j=1}^{n} \Tr \Big(E_{ji} \Im M_{0} \, \Big(\frac{1}{M_0} M_0'  \, \Gammat [E_{ij} ] \Big)^{\, t} \Im M_{0}\Big)
      ,
    \\ \label{eq:282}
        \phi^{\,(-,+)}
    &=
      \frac{-2\imu }{\la \Im M_{0} \ra}  \sum_{i,j=1}^{n} \Tr \Big(E_{ji} \, \Big\la \Im M_{0},\, \frac{1}{M_0^{*}} (M_0')^{*}  \, \Gamma [E_{ij} ] \, \Big\ra  \Im M_{0}\Big)
      ,
    \\ \label{eq:283}
        \tilde{\phi}^{\,(-,+)}
    &=
      \frac{-2\imu }{n \la \Im M_{0} \ra}  \sum_{i,j=1}^{n} \Tr \Big(E_{ji} \Im M_{0} \Big(\frac{1}{M_0^{*}} (M_0')^{*}  \, \Gammat [E_{ij} ] \Big)^{\, t} \Im M_{0}\Big)
      .
  \end{align}
  For $\dagger, \star \in \{+,-\}, \dagger \neq \star$, the scalar quantities $\phi^{\,(\dagger,\star)}$ arise from the deterministic approximations \eqref{eq:165} of $G^{B_{ij}(E_0)}(z,\zeta)$ for $(z,\zeta) \in \Omega^{\dagger}\times \Omega^{\star}$, and $\tilde{\phi}^{\,(\dagger,\star)}$ from the approximations \eqref{eq:227} of $\Gt^{\BT_{ij}(E_0)}(z,\zeta)$ for $(z,\zeta) \in \Omega^{\dagger}\times \Omega^{\star}$.
  When expressing $\phi^{\,(\dagger,\star)}$ and  $\tilde{\phi}^{\,(\dagger,\star)}$ in \eqref{eq:280}-\eqref{eq:283} we used that $\vartheta = 1$ for $(z,\zeta) \in \Omega^{+}\times \Omega^{-}$ and $\vartheta = -1$ for $(z,\zeta) \in \Omega^{-}\times \Omega^{+}$. 
  Moreover, we applied the identity \eqref{eq:199}, introduced again the shorthand notations $M_0 := \lim_{y \downarrow 0} M(E_0 + \imu y)$ and $M_0' := \lim_{y \downarrow 0} M'(E_0 + \imu y)$, and used $\lim_{\Omega^{-}\ni z \to E_0} M(z) = M_0^{*}$ and $\lim_{\Omega^{-}\ni z \to E_0} M'(z) = (M_0')^{*}$.
It has been shown in \eqref{eq:phi+-}-\eqref{eq:200} that
\begin{equation}
  \label{eq:284}
  \phi^{\,(+,-)}
    =
  \frac{2\imu  }{\la \Im M_{0} \ra}  \Big\la \Gamma \Big[\Im M_{0}\, \frac{1}{M_{0}} M'_{0} \Big]  \,   \Im M_{0}  \Big\ra
  =
  -1
  .
\end{equation}
Since $M(z)$ and $M'(z)$ are symmetric and $\big(\Gammat[R]\big)^{t} = \Gammat[R^{t}]$ for any $R \in \CC^{n\times n}$, we have
\begin{align*}
  & \frac{1}{n} \sum_{i,j=1}^{n}  \Tr \Big(E_{ji} \, \Im M_{0} \Big(\frac{1}{M_{0}} M'_{0} \, \Gammat [E_{ij} ] \Big)^{t} \Im M_{0}   \Big)
  \\ 
  &
    =
    \sum_{i,j=1}^{n} \Big\la E_{ji} \Im M_0  \Gammat [E_{ji} ] \Big(\frac{1}{M_0} M_0' \Big)^t  \Im M_0 \Big\ra 
  \\
  & 
    =
    \sum_{i,j=1}^{n} \sum_{\alpha = 1}^{d} \Big\la E_{ji} \Im M_0  \Big(L_{\alpha}E_{ji} \Big( \Im M_0  \frac{1}{M_0} M_0' L_{\alpha}^t \Big)^t  
      + 
    L_{\alpha}^{t}E_{ji} \Big( \Im M_0  \frac{1}{M_0} M_0' L_{\alpha} \Big)^t\, \Big) \Big\ra 
  \\
  & 
    =
    \sum_{i,j=1}^{n} \sum_{\alpha = 1}^{d} \Big\la E_{ii} \Im M_0  L_{\alpha}E_{jj} \Im M_0  \frac{1}{M_0} M_0' L_{\alpha}^t  \Big\ra
    + \Big\la E_{ii} \Im M_0 L_{\alpha}^{t}E_{jj} \Im M_0  \frac{1}{M_0} M_0' L_{\alpha}  \Big\ra 
  \\
  & 
    =
    \Big\la \Im M_0 \, \Gamma\Big[\Im M_0  \frac{1}{M_0} M_0' \Big]\Big\ra
    .
\end{align*}
Plugging this into \eqref{eq:281} and using \eqref{eq:284} yields
\begin{equation*}
  \phiT^{\,(+,-)}
  =
  \frac{2\imu  }{\la \Im M_{0} \ra}  \Big\la \Gamma \Big[\Im M_{0}\, \frac{1}{M_{0}} M'_{0} \Big]  \,   \Im M_{0}  \Big\ra
  =
  -1
  .
\end{equation*}
By applying the same method we show that $\phi^{\,(-,+)} = \phiT^{\,(-,+)} = -1 $.
We conclude that 
\begin{equation*}
  \Vc
  =
  -\frac{2}{\pi^2}\EE\big[ \efr(t) \big]\int\displaylimits_{(\Omega^{+} \times \Omega^{-})\cup (\Omega^{-}\times \Omega^{+})}  \frac{\partial \ft(z)}{\partial \overline{z}}  \frac{\partial \ft(\zeta)}{\partial \overline{\zeta}} \frac{\partial}{\partial \zeta} \frac{1}{z - \zeta}\, \dzeta  \dz + O\big(N^{-\tau}\big)
  .
\end{equation*}
Finally, using \eqref{eq:203}, \eqref{eq:209} and the same argument as at the end of Section~\ref{sec:proof-5} we obtain that in the real symmetric case
$\lim_{N\to \infty} \EE[e(t)] = e^{-\frac{t^2}{2}V[g]}$,
where
\begin{equation*}
  V[g]
  =
  \frac{1}{2 \pi^2} \int_{\RR} \int_{\RR} \frac{(g(x) - g(y))^{2}}{(x-y)^2} dx dy
  .
\end{equation*}
This finishes the proof of Theorem~\ref{thm:main}.

\appendix

 \section{Kronecker random matrix models with $L$-flat self-energy}
\label{sec:lflat}

In this section we state several important properties of the solution to the matrix Dyson equation \eqref{eq:7} under the assumption that the self-energy operator $\Gamma$ satisfies the $L$-flatness property \textbf{(A)}.
Most of these results are established by following the proofs of the corresponding results in Propositions~2.2, 4.2 and 4.7 of \cite{AjanErdoKrug19}, where the matrix Dyson equation \eqref{eq:7} has been studied under the stronger assumption of $1$-flat self-energy \eqref{eq:11}.
A statement similar to that of Proposition~\ref{pr:1} below for the  general $L$-flat self-energy can be found in an early arXiv version (version v3) of \cite{AjanErdoKrug19}. 
For the reader's convenience, we present the full proof, adjusted to our current setup, and show how the modifications of  the arguments from \cite{AjanErdoKrug19} yield Proposition~\ref{pr:1}.

Consider the Kronecker random matrix model $\Hb^{(\beta)}$, $\beta \in \{1,2\}$, defined in \eqref{eq:1}.
Let $M(z)$ be the solution to the corresponding matrix Dyson equation \eqref{eq:7}.
Recall that by \cite[Theorem 2.1]{HeltRaFaSpei07}, for any $z\in \CC_{+}$, the solution to \eqref{eq:7} satisfying $\Im M(z) > 0$ exists and is unique.
The solution also admits the Stieltjes transform representation \eqref{eq:9} (see, e.g., \cite[Proposition 2.1]{AltErdoKrug20}).
  Denote
\begin{equation*}
  \rho(z)
  :=
  \frac{1}{\pi} \big\la \Im M(z) \big\ra
\end{equation*}
for $z\in \CC_{+}$, and recall that for $T\in\CC^{n\times n}$ we denote by $\Cc_{T}$ an operator on $\CC^{n\times n}$ defined in \eqref{eq:47} by $\Cc_{T}[R] = T R T$ for all $R\in\CC^{n\times n}$.

\begin{pr}[Properties of the solution to the MDE]
  \label{pr:1}
  Suppose that the self-energy operator $\Gamma$ satisfies the $L$-flatness property \textbf{(A)} for some $L\in \NN$.
  Then the following holds.
  \begin{itemize}
  \item[(i)] \textbf{$M(z)$ and $M^{-1}(z)$ are uniformly bounded}: The solution $M(z)$ satisfies the bounds
    \begin{equation}
      \label{eq:291}
      \big\| M(z) \big\|
      \lesssim
      1
      , \qquad
      \big\| M^{-1}(z) \big\|
      \lesssim
      1 + |z|
    \end{equation}
    uniformly for $z\in \CC_{+}$.
  \item[(ii)] \textbf{$\Im M(z)$ and $\rho(z) I_n$ are comparable}: The relation
    \begin{equation*}
      \Im M(z)
      \sim
      \rho(z) I_n
    \end{equation*}
    holds uniformly on $\CC_{+}$.
  \item[(iii)] \textbf{Linear stability}: The bound
    \begin{equation}
      \label{eq:293}
      \Big\|\big( \mathrm{Id}_n - \Cc_{M(z)} \Gamma \big)^{-1} \Big\|
      \lesssim
      1 + \frac{1}{(\rho(z) + \mathrm{dist}(z, \mathrm{supp}\,\rho))^{2}}
    \end{equation}
    holds uniformly for all $z\in \CC_{+}$.
  \item[(iv)] \textbf{Regularity of $M(z)$ and the density of states}:
    The continuous extension of $M(z)$ to  $\CC_{+}\cup \RR$ exists and satisfies $\|M(z_1) - M(z_2)\| \lesssim | z_1 - z_2|^{1/3}$ uniformly for all $z_1, z_2 \in \CC_{+}\cup \RR$.
    In particular, the density of states $\rho(x)$ defined in \eqref{eq:15} is $1/3$-H\"{o}lder continuous on $\RR$, i.e.,
    \begin{equation}
      \label{eq:295}
      |\rho(x_1) - \rho(x_2)|
      \lesssim
      |x_1 - x_2|^{1/3}
    \end{equation}
    uniformly for all $x_1,x_2 \in \RR$.
    Moreover, the density $\rho(x)$ is real analytic on the open set $\{x \in \RR\, : \, \rho(x)>0\}$.
  \item[(v)] \textbf{Analytic continuation of $M(z)$}: For any $x_0 \in \RR$ satisfying $\rho(x_0)>0$ the solution $M(z)$ can be analytically extended to a neighborhood of $x_0$ in $\CC$.
  \end{itemize}
  The hidden constants in $\lesssim$ and $\sim$ depend on $L, d \in \NN$ and the structure matrices $K_0, L_1, \ldots, L_{d} \in \CC^{n\times n}$.
\end{pr}

\begin{proof}
  We split the proof of Proposition~\ref{pr:1} into four steps corresponding to the statements \textit{(i)}, \textit{(ii)}, \textit{(iii)} and \textit{(iv) - (v)}.
 To make the presentation lighter, we will often suppress the $z$-dependence in the notation.
  
  \textit{Proof of (i)}.
By taking the imaginary part of the Dyson equation \eqref{eq:7} and multiplying it by $M^{*}$ from the left and $M$ from the right we get
  \begin{equation}
    \label{eq:296}
    \Im M
    =
    \Im z M^{*} M +  M^{*}\Gamma[\Im M] M
    \geq
    \Gamma[\Im M]
 \end{equation}
 for all $z\in \CC_{+}$.
 Since $\Im M$ is positive definite, the $L$-flatness \eqref{eq:10} gives
 \begin{equation}
   \label{eq:297}
   \Im M
   \gtrsim
   \sum_{k,l = 1}^{n} z_{kl} \big(\Im M\big)_{kk}   M^* E_{ll}  M
   \gtrsim
   \sum_{k = 1}^{n}  \big(\Im M\big)_{kk}   M^* E_{kk}  M
   ,
 \end{equation}
 where in the last step we used that $z_{kk} = 1$ for $1\leq k \leq n$.
 If in \eqref{eq:296} we multiply by $M$ from the left and $M^*$ from the right, then instead of \eqref{eq:297} we obtain
 \begin{equation}
   \label{eq:298}
      \Im M
   \gtrsim
   \sum_{k,l = 1}^{n} z_{kl} \big(\Im M\big)_{kk}   M E_{ll}  M^*
   \gtrsim
   \sum_{k = 1}^{n}  \big(\Im M\big)_{kk}   M E_{kk}  M^*
   .
 \end{equation}
 By looking at the diagonal entries of \eqref{eq:297} and \eqref{eq:298} we get 
 \begin{equation}
   \label{eq:299}
   w_{j}
   \gtrsim
   \sum_{k,l = 1}^{n} z_{kl} w_{k}   t_{lj}
   \gtrsim
   \sum_{k = 1}^{n}  w_{k}   t_{kj}
 \end{equation}
 for all $1\leq j \leq n$,  where we denoted $w_j := \big(\Im M \big)_{jj}$ and $t_{kj}:= | M_{kj} |^2 + | M_{jk} |^2$.
 The matrix $T := (t_{kj})_{k,j=1}^{n}\in \RR^{n\times n}$ has nonnegative entries, and the vector $\mathrm{w} = (w_j)_{j=1}^{n}\in \RR^{n}$ has strictly positive entries.
 For convenience, we rewrite \eqref{eq:299} as
 \begin{equation}
   \label{eq:300}
   \mathrm{w}^{t}
   \gtrsim
   \mathrm{w}^{t}\,Z \,T
   \geq
   \mathrm{w}^{t}\,T
 \end{equation}
 with vector inequalities understood in the entrywise sense.
 If we denote by $\mathrm{v}\in \RR^{n}$ the right Perron-Frobenius eigenvector of $T$, then after comparing $\mathrm{w}^{t} \mathrm{v}$ and $\mathrm{w}^{t} T \mathrm{v} = \|T \| \mathrm{w}^{t} \mathrm{v}$, and using \eqref{eq:300} we find that the spectral norm of $T$ satisfies the bound $\|T\| \lesssim 1$ uniformly on $z\in \CC_{+}$.
 Thus, we conclude that \eqref{eq:291} holds uniformly on $z\in \CC_{+}$.

 Moreover, if we take the norm on both sides of the matrix Dyson equation \eqref{eq:7}, then \eqref{eq:291} yields the bound
 \begin{equation}
   \label{eq:302}
   \big\| M^{-1}(z) \big\|
   \lesssim
   1 + |z|
 \end{equation}
 uniformly on $z\in \CC_{+}$, and therefore, for any $C>0$ we have
 \begin{equation}
   \label{eq:303}
   \| M \|
   \sim
   \| T \|
   \sim
   1
 \end{equation}
 uniformly on $\{z \in \CC_{+} \, : \, |z|\leq C\}$.

  \textit{Proof of (ii)}.
Since the measure $V(dx)$ is compactly supported, it is easy to see from the Stieltjes transform representation \eqref{eq:9} and the normalization $V(\mathbb{R}) = I_n$ that for a sufficiently large $C>0$ the equivalences
 \begin{equation}
   \label{eq:304}
   \Im M(z)
   \sim
   \frac{\Im z}{|z|^2} I_{n}
   \sim
   \rho(z) I_{n}
 \end{equation}
 hold uniformly on $\{z \in \CC_{+} \, : \, |z|\geq C\}$.
 
 Suppose that $|z|\leq C$.
 Then applying the first inequality in \eqref{eq:300} $2L$ times gives
  \begin{equation}
   \label{eq:305}
   \mathrm{w}^{t}
   \gtrsim
   \mathrm{w}^{t}\,(Z\,T\, Z \, T)^{L}
   \gtrsim
   \mathrm{w}^{t}\,(Z\,T^{\,2})^{L}
   .
 \end{equation}
 where in the second inequality we again used that $z_{kk}= 1$ for $1\leq k \leq n$.
The bounds \eqref{eq:302} and \eqref{eq:303} imply that $M^*M \sim I_n$, from which we have
 \begin{equation}
   \label{eq:306}
   \sum_{j=1}^{n} |M_{jk}|^2 \sim 1
 \end{equation}
 for all $1\leq k \leq n$. 
 Now, using \eqref{eq:306}, we estimate the diagonal entries of $T^{\,2}$ from below
 \begin{equation*}
   \big(T^{\,2} \big)_{kk}
   =
   \sum_{j=1}^{n} t_{jk}^2
   \gtrsim
   \Big(\sum_{j=1}^{n} t_{jk} \Big)^2
   \geq
   \Big(\sum_{j=1}^{n} |M_{jk}|^2 \Big)^2
   \gtrsim
   1
 \end{equation*}
 for $1 \leq k \leq n$, which yields
 \begin{equation}
   \label{eq:308}
   \mathrm{w}^{t}\,(Z\,T^2)^{L}
   \gtrsim
   \mathrm{w}^{t}\,Z^{L}
   .
 \end{equation}
 Since $Z^{L}$ has all entries greater than or equal to $1$, we conclude from \eqref{eq:305} and \eqref{eq:308} that
 \begin{equation*}
   w_{j}
   \gtrsim
   \sum_{k=1}^{n} w_k
   \sim
   \la \Im M(z) \ra
 \end{equation*}
 for $1\leq j \leq n$.
 Now $\big(\Im M(z)\big)_{jj} \sim \rho(z)$ for all $1\leq j \leq n$, and from \eqref{eq:297} and \eqref{eq:302} we get
 \begin{equation}
   \label{eq:311}
   \Im M(z)
   \gtrsim
   \rho(z)   M^*(z)  M(z)
   \geq
   \rho(z) \big\| M^{-1}(z) \big\|^2 I_n
   \gtrsim
   \rho(z) I_n
 \end{equation}
 uniformly on $\{z \in \CC_{+} \, : \, |z| \leq C\}$.
 To estimate $ \Im M(z)$ from above, we notice that since the dimension of the equation \eqref{eq:7} is fixed, we trivially have $\Gamma[\Im M(z)] \lesssim \rho(z) I_n$.
Therefore, by taking again the imaginary part of the MDE \eqref{eq:7} we get 
 \begin{equation}
   \label{eq:313}
   \Im M
   =
   \Im z \,M^*M(z) + M^* \Gamma [\Im M] M
   \lesssim
   (\Im z + \rho(z)) M^*M
   .
 \end{equation}
 On the other hand, from the imaginary part of the Stieltjes transform representation \eqref{eq:9} we see that $\Im M(z) \gtrsim \Im z I_n$ for $|z| \leq C$, which together with \eqref{eq:313} and \eqref{eq:303} gives
 \begin{equation}
   \label{eq:314}
   \Im M(z)
   \lesssim
   \rho(z) \| M(z) \|^2 I_n
   \lesssim
   \rho(z)I_n
 \end{equation}
 uniformly on $\{z \in \CC_{+} \, : \, |z|\leq C\}$.
Combining \eqref{eq:304}, \eqref{eq:311} and \eqref{eq:314} we obtain
 \begin{equation}
   \label{eq:315}
   \Im M(z)
   \sim
   \rho(z) I_n
 \end{equation}
 uniformly on $z\in\CC_{+}$.

 \textit{Proof of (iii)}.
We now establish \eqref{eq:293}. 
Following the proof of the invertibility of the stability operator from \cite[Section 4.2]{AjanErdoKrug19}, we define the \emph{saturated} self-energy operator $\Fc = \Fc(z): \CC^{n\times n} \to \CC^{n \times n}$ given by
 \begin{equation}
   \label{eq:316}
   \Fc
   =
   \Cc^{*}\, \Gamma \, \Cc
   ,
 \end{equation}
 where $\Cc := \Cc_{\sqrt{\Im M}} \, \Cc_{W}$ and
 \begin{equation*}
   W
   :=
   \Big(I_n + \big( \Cc_{\sqrt{\Im M(z)}}^{-1}[\Re M(z)]\big)^{2} \Big)^{1/4}
 \end{equation*}
 is a positive definite matrix.
 Then \eqref{eq:293} can be obtained by first establishing the existence of the uniform spectral gap for $\Fc$, and then applying the Rotation-Inversion Lemma (see Lemmas 4.7, 4.9 and the proof of Proposition~4.4 in \cite{AjanErdoKrug19} for the details).
 The crucial ingredient in this approach is the study of the spectrum of $\Fc$ in \cite[Lemma~4.7]{AjanErdoKrug19}.
 We gather the spectral properties of $\Fc$ is the following lemma, that will be proven at the end of this section.
 \begin{lem}[Spectrum of $\Fc$]
   \label{lem:a1}
   Let $\Fc = \Fc(z)$ be the operator defined in \eqref{eq:316}, and suppose that $\Gamma$ satisfies the $L$-flatness property \textbf{(A)}.
   Then the following holds.
   \begin{itemize}
   \item[(i)] For any $x \in \RR$ the operator $\Fc(x)$ possesses a simple Perron-Frobenius eigenvalue $\|\Fc(x)\|_{2} = 1$, so that
 \begin{equation}
   \label{eq:318}
   \Fc(x) [F]
   =
   F
   ,
 \end{equation}
 where $F \in \CC^{n\times n}$ is the corresponding Perron-Frobenius eigenvector satisfying $\|F\|_{\mathrm{HS}} = 1$ and $\Im F > 0$, and $\|\Fc\|_{2}$ denotes the operator norm of $\Fc$ induced by the Hilbert-Schmidt norm on $\CC^{n\times n}$.
\item[(ii)] There exists $C>0$, sufficiently large, such that
  \begin{equation}
   \label{eq:319}
   \Fc^{L}[R]
   \sim
   \la R \ra I_n
 \end{equation}
 for all positive definite $R \in \CC^{n\times n}$ uniformly on $\{z \in \CC_{+}\, : \, |z| \leq C\}$ .
\item[(iii)] There exists $\kappa>0$, sufficiently small, such that
  \begin{equation}
   \label{eq:320}
   \mathrm{Spec}(\Fc)
   \subset
   [-1+\kappa,1-\kappa]\cup \{1\}
 \end{equation}
 uniformly on $\{z \in \CC_{+} \, : \, |z| \leq C\}$.
   \end{itemize}
 \end{lem}

Given the existence of the spectral gap for $\Fc$ in \eqref{eq:320}, the linear stability bound \eqref{eq:293} follows by repeating the argument presented in \cite[Section~4.2]{AjanErdoKrug19}.
 From \eqref{eq:303} we also find the term $(\rho(z) + \mathrm{dist}(z,\mathrm{supp} \, \rho))^{-2}$ on the right-hand side of the \eqref{eq:293} with the explicit exponent $2$.
 Indeed, using \eqref{eq:303} and the fact that the dimension of the matrix Dyson equation is $N$-independent, we get from  Eq. (4.41) in \cite{AjanErdoKrug19} that
 \begin{equation}
   \label{eq:321}
   \big\| \big(\mathrm{Id} - \Cc_{M(z)}\Gamma \big)^{-1} \big\|
   \lesssim
   \big\| \big(\Cc_{U} - \Fc \big)^{-1} \big\|
 \end{equation}
 uniformly on $\{z\in \CC_{+}\, : \, |z| \leq C\}$, where $U$ is a unitary matrix and $\Fc$ was defined in \eqref{eq:316}.
 Then the Rotation-Inversion Lemma \cite[Lemma~4.9]{AjanErdoKrug19} implies that
 \begin{equation}
   \label{eq:322}
   \big\| \big(\Cc_{U} - \Fc \big)^{-1} \big\|
   \lesssim
   \frac{1}{\max\{1 - \| \Fc\|, |1 - \la F , \Cc_{U}[F] \ra |\}}
 \end{equation}
 uniformly on $\{z\in \CC_{+}\, : \, |z| \leq C\}$, where $F$ is the normalized Perron-Frobenius eigenvector of $\Fc$.
 With \eqref{eq:303}, the lower bound in Eq. (4.45) in \cite{AjanErdoKrug19} becomes $1 - \| \Fc\| \gtrsim \mathrm{dist}(z, \mathrm{supp} \, \rho)^{2}$,  and the last inequality in the proof of Lemma~4.4 in \cite{AjanErdoKrug19} gives  $|1 - \la F , \Cc_{U}[F] \ra | \gtrsim \rho^2$, which together with \eqref{eq:321} and \eqref{eq:322} establish \eqref{eq:293}.

\textit{Proof of (iv) and (v)}.
With the linear stability \eqref{eq:293} established, the regularity of $M(z)$ and the density of states \eqref{eq:295}, as well as the real-analyticity of $\rho(x)$, is obtained by following exactly the lines of the proof of \cite[Proposition 2.2]{AjanErdoKrug19}.
\end{proof}
\begin{proof}[Proof of Lemma~\ref{lem:a1}]
  Part (i) is obtained using exactly the same argument as in part (i) of Lemma~4.7 in \cite{AjanErdoKrug19}.

  In order to prove part (ii), we split \eqref{eq:319} into the upper and lower bounds.
 The upper bound in \eqref{eq:319} is obtain in the same way as the upper bound in Eq. (4.33) of \cite{AjanErdoKrug19}, which can then be iterated $L$ times.
 For the lower bound in \eqref{eq:319}, using the proof by induction, we show that
 \begin{equation}
   \label{eq:325}
   \Fc^{l}[R]
   \gtrsim
   \sum_{k,j=1}^{n} \big(Z^{l}\big)_{kj} \big\la \Cc^{*}[E_{kk}]R \big\ra  \, \Cc^*[E_{jj}]
 \end{equation}
 for all $1 \leq l \leq n$. 
 The case $l = 1$ follows directly from \eqref{eq:10} and the definition of $\Fc$ in \eqref{eq:316}.
 Suppose that \eqref{eq:325} holds for $l<L$. 
 Then after applying $\Fc$ on both sides of \eqref{eq:325} and using \eqref{eq:10} we get
 \begin{equation*}
   \Fc^{l+1}[R]
   \gtrsim
   \sum_{k,j=1}^{n} \big(Z^{l}\big)_{kj} \big\la \Cc^{*}[E_{kk}]R \big\ra  \, \sum_{k',j'=1}^{n} z_{k'j'} \big\la \Cc^*[E_{k'k'}] \Cc^*[E_{jj}] \big\ra \Cc^*[E_{j'j'}]
   .
 \end{equation*}
 Since $\big\la \Cc^*[E_{k'k'}] \Cc^*[E_{jj}] \big\ra \geq 0$ for all $1 \leq k', j \leq n$, we can further estimate
 \begin{equation}
   \label{eq:327}
   \Fc^{l+1}[R]
   \gtrsim
   \sum_{k,j, j' =1}^{n} \big(Z^{l}\big)_{kj} \big\la \Cc^{*}[E_{kk}]R \big\ra  \, z_{jj'} \Big\la \big(\Cc^*[E_{jj}]\big)^2 \Big\ra \, \Cc^*[E_{j'j'}]
   .
 \end{equation}
  In order to estimate the right-hand side of \eqref{eq:327} from below, we notice that exactly as in \cite[Lemma~4.6]{AjanErdoKrug19} we have that 
 \begin{equation}
   \label{eq:328}
   \rho^{1/2}(z) W
   \sim
   I_n
 \end{equation}
 uniformly on $\{z \in \CC_{+}\, : \, |z| \leq C\}$ for any sufficiently large $C>0$.
 Therefore, from \eqref{eq:315}, \eqref{eq:328} and the definition of $\Cc$ we get $\big\la \big(\Cc^*[E_{jj}]\big)^2 \big\ra \gtrsim 1$, which concludes the proof of the induction step.
 Now \eqref{eq:325} with $l=L$ yields
 \begin{equation*}
   \Fc^{L}[R]
   \gtrsim
   \sum_{k,j=1}^{n} \big(Z^{L}\big)_{kj} \big\la \Cc^{*}[E_{kk}]R \big\ra  \, \Cc^*[E_{jj}]
   \gtrsim
   \big\la \Cc^{*}[I_n]R \big\ra  \, \Cc^*[I_n]
   ,
 \end{equation*}
 which after applying \eqref{eq:315} and \eqref{eq:328} to $\Cc^{*}[I_n] = W \Im M(z)\, W$ gives $\Fc^{L}[R] \gtrsim \la R \ra I_n$.
 Combined with the upper bound, this proves \eqref{eq:319}.

 With \eqref{eq:319} established, part (iii) follows from \cite[Lemma~4.8]{AjanErdoKrug19}.
 Indeed, by applying \cite[Lemma~4.8]{AjanErdoKrug19} to the operator $\Fc^{L}$ we find that $\Fc^{L}$ possesses a spectral gap uniformly on $\{z \in \CC_{+} \, : \, |z| \leq C\}$.
 From this and the symmetry of $\Fc$ we conclude that $\Fc$ also has a spectral gap, and thus \eqref{eq:320} holds uniformly on $\{z \in \CC_{+} \, : \, |z| \leq C\}$.
\end{proof}

\section{Proof of Lemma~\ref{lem:b1}}
\label{sec:iid}
 The crucial ingredient of the proof of Lemma~\ref{lem:b1} is the following cumulant expansion formula.
This formula is the complex analog of \cite[Proposition 3.2]{ErdoKrugSchr19} and follows immediately from the real case presented there.
\begin{lem}
  \label{lem:bCumulant}
Let $\xib :=(\xi_1, \dots, \xi_{K}) \in \CC^{K}$  be a complex random vector with finite moments up to order $R+1$ for $R \in \NN$, and let $\varphi: \CC^{K}  \to \CC$ be $R+1$ times differentiable with bounded partial derivatives. Then 
\begin{equation}
  \label{eq:330}
  \EE \Big[\xi_{1} \varphi(\xib)\Big] = \sum_{m=0}^{R-1}\frac{1}{m!}\EE  \Big[ \big(\kappa^{\nabla}_{m+1}[\xi_{1},\xib]  \varphi \big) (\xib)\Big] + \frac{1}{R!}\EE\bigg[\int_0^1 \big( K^{\nabla}_{R+1,s}[\xi_{1},\xib]  \varphi \big)(s \xib)  d s\bigg]
  ,
\end{equation}
where for a complex random variable $\theta$ we set $\kappa^{\nabla}_{1}[\theta,\xib] \varphi := \EE[\theta] \varphi$,  $K_{m+1,s}^{\nabla}$ is a random differential operator defined for $m\geq 1$ through
\begin{equation*}
  K_{m+1,s}^{\nabla}[\theta,\xib]
  :=
  m! \int_{[0,1]^{m-1}}\nabla_{\xib}(\mathbbm{1}_{s\le t_{m-1}}-\EE)\nabla_{\xib}(\mathbbm{1}_{t_{m-1}\le t_{m-2}}-\EE) \cdots
  \nabla_{\xib}(\mathbbm{1}_{t_{1}\le 1}-\EE) \theta  d\tb  
\end{equation*}
 with $d\tb:= dt_1\dots d t_{m-1}$, $\nabla_{\xib} := \sum_{i=1}^K (\xi_{i} \partial_{i} + \overline{\xi}_i \overline{\partial}_{i})$,
and where
\begin{equation*}
\kappa^\nabla_{m+1}[\theta,\xib]:= \EE \bigg[\int_{0}^{1}K^\nabla_{m+1,s}[\theta,\xib] ds \bigg]
\end{equation*}
is a non-random differential operator. 
\end{lem}
The derivation of the above lemma is identical to the real case treated in \cite[Proposition 3.2]{ErdoKrugSchr19}, thus omitted.
We now proceed to establishing Lemma~\ref{lem:b1}.
\begin{proof}[Proof of Lemma~\ref{lem:b1}]
 For $\star \in \{1,2\}$ we have
  \begin{equation}
    \label{eq:333}
    \EE[\Wb \Fcc_{\star}]
    =
    \sum_{i,j=1}^{N} \sum_{l=1}^{N}\EE\Big[ W_{il} F_{\star,lj}\Big] \otimes \Eb_{ij}
    ,
  \end{equation}
  where $F_{1,lj}: = G_{lj}$ and $F_{2,lj}:= G_{lj}\efr(t)$.
  We apply Lemma~\ref{lem:bCumulant} to two cases, namely $\varphi=\varphi(W_{il},W_{li}):=F_{1,lj}(W_{il},W_{li})$ and $\varphi = \varphi(W_{il},W_{li}):=F_{2,lj}(W_{il},W_{li}) $, where we keep the entries of $W_{ab}$ with $(a,b) \in \{1, \dots, N\}^2 \setminus \{(i,l),(l,i)\}$ fixed, i.e., we take the partial expectation $\EE_{il}$ with respect to $W_{il}$ first.
  For $i=l$ we interpret this as $\varphi=\varphi(W_{ii})$.
  
  In order to accommodate the matrix setting of \eqref{eq:333}, where both $W_{il}$ and $F_{\star,lj}(W_{il},W_{li})$ are $n\times n$ matrices, we denote by $\nabla^{ij}_{V}:= \sum_{\alpha,\beta=1}^n v_{\alpha\beta} \partial_{w_{ij}^{\alpha\beta}}$ the directional derivative with respect to the $(i,j)$-th block $W_{ij} = (w_{ij}^{\alpha \beta})_{\alpha,\beta = 1}^{n}$ in the direction $V = (v_{\alpha \beta})_{\alpha,\beta = 1}^{n} \in \CC^{n \times n}$.
  The analyticity of $F_{\star, lj}(W_{il},W_{li})$ for $\star\in\{1,2\}$ implies that $\overline{\partial}_{w_{ik}^{\alpha \beta}} F_{\star, lj} = 0$ for all $1\leq i,k \leq N$ and $1 \leq \alpha, \beta \leq n$.
  
  With this notation, the cumulant expansion \eqref{eq:330} for $R=3$ reads 
  \begin{multline}
    \label{eq:334}
     \EE\Big[ W_{il} F_{\star,lj}\Big]
     =
      \EE \bigg[\int_{0}^{1}\big(\nabla^{il}_{\Wt_{il}} + \nabla^{li}_{\Wt_{li}}\big)(\mathbbm{1}_{s\le 1}-\EE_{\Wbt}) \Wt_{il} ds \, F_{\star,lj} (W_{il},W_{li})
    \\ 
      +
      \sum_{\sigma_{1},\sigma_{2}\in\{il,li\}}
      \int_{[0,1]^{2}}\nabla^{\sigma_{1}}_{\Wt_{\sigma_{1}}}(\mathbbm{1}_{s\le t_{1}}-\EE_{\Wbt}) \nabla^{\sigma_{2}}_{\Wt_{\sigma_{2}}}(\mathbbm{1}_{t_{1}\le 1}-\EE_{\Wbt}) \Wt_{il} dt_1ds \,F_{\star,lj} (W_{il},W_{li})
    \\ 
      +
      \sum_{\sigma_{1},\sigma_{2},\sigma_{3}\in\{il,li\}}\int_{0}^{1} K^{\nabla,(\sigma_{1},\sigma_{2},\sigma_{3})}_{4,s}[W_{il},(W_{\sigma_{1}},W_{\sigma_{2}},W_{\sigma_{3}})]  F_{\star,lj} (sW_{il}, sW_{li})  ds\bigg]
      ,
  \end{multline}
  where we introduced a random differential operator
  \begin{equation*}
    K^{\nabla,\sigmab}_{4,s}[V_{0},\Vb]
    :=
    \int_{0}^{1} \int_{0}^{1} \nabla_{V_{1}}^{\sigma_{1}}(\mathbbm{1}_{s\le t_{2}}-\EE) \nabla_{V_{2}}^{\sigma_{2}}(\mathbbm{1}_{t_{2}\le t_{1}}-\EE) \nabla_{V_{3}}^{\sigma_{3}}(\mathbbm{1}_{t_{1}\le 1}-\EE)  V_{0} dt_1 dt_2
  \end{equation*}
  with random matrices $V_{0}$, $(V_1,V_2,V_3)=:\Vb$ and indices $\sigmab := (\sigma_{1},\sigma_{2},\sigma_{3})$.
  In the above formulas $\Wbt$ is an independent copy of $\Wb$, and $\EE_{\Wbt}$ denotes the partial expectation with respect to $\Wbt$.
  In \eqref{eq:334} we also used that $W_{il}$ is centered, which implies that the term corresponding to $m=0$ in \eqref{eq:330} vanishes.

  We now treat each term on the right-hand side of \eqref{eq:334} individually.
  From the direct computation of the first term we see that for $\sigma_{1} \in \{il,li\}$
  \begin{align*}
    \EE \bigg[
    \int_{0}^{1}\nabla^{\sigma_{1}}_{\Wt_{\sigma_{1}}} (\mathbbm{1}_{s\le 1}-\EE_{\Wbt}) \Wt_{il} ds F_{\star,lj} (W_{il},W_{li})\bigg]
    =
    \EE \bigg[
    \Wt_{il} \nabla^{\sigma_{1}}_{\Wt_{\sigma_{1}}} F_{\star,lj} (W_{il},W_{li})\bigg]
    .
  \end{align*}
  Taking the sum of the above expression for $\sigma_{1} \in \{il, li\}$ and $l \in \{1,\ldots, N\}$, and using the independence of $\Wt_{il}$ and $\Wt_{ab}$ for $ab \notin \{il,li\}$ gives $\EE\big[ \big( \Wbt \nabla_{\Wbt}\Fcc_{\star}(\Wb) \big)_{ij} \big]$.
  Together with \eqref{eq:333} we recover the first term in \eqref{eq:41}.

  Similar computations for the second term give
  \begin{align}
    \nonumber 
    \EE\bigg[\sum_{l=1}^{N} \sum_{\sigma_{1},\sigma_{2}\in\{il,li\}}
    \int_{[0,1]^{2}}\nabla^{\sigma_{1}}_{\Wt_{\sigma_{1}}}(\mathbbm{1}_{s\le t_{1}}-\EE_{\Wbt}) \nabla^{\sigma_{2}}_{\Wt_{\sigma_{2}}}(\mathbbm{1}_{t_{1}\le 1}-\EE_{\Wbt}) \Wt_{il} dt_1ds \,F_{\star,lj} (W_{il},W_{li}) \bigg]
    \\ \label{eq:338}
      =
      \frac{1}{2} \EE\bigg[\sum_{l=1}^{N} \sum_{\sigma_{1},\sigma_{2}\in\{il,li\}}  \Wt_{il} \nabla^{\sigma_{1}}_{\Wt_{\sigma_{1}}} \nabla^{\sigma_{2}}_{\Wt_{\sigma_{2}}}  F_{\star,lj} (\Wb) \bigg]
      =
    \frac{1}{2}\EE\Big[ \Big( \Wbt \big(\nabla_{\Wbt}\big)^{2}\Fcc_{\star}(\Wb) \Big)_{ij} \Big].
  \end{align}
  We now estimate the above expression when $F_{1,lj}(\Wb) = G_{lj}(z)$, i.e., $\Fcc_{1}(\Wb) = \Gb(z)$.
  From the properties of the resolvent \eqref{eq:56} we find that 
  \begin{equation}
  \label{eq:340}
  \nabla^{\sigma_1}_{V_1} \dots \nabla_{V_k}^{\sigma_k} G_{lj}
  =
  (-1)^{k} \sum_{\tau \in S_k}G_{li_{\tau(1)}} V_{\tau(1)} G_{j_{\tau(2)} i_{\tau(2)}} \cdots G_{j_{\tau(k-1)}i_{\tau(k)}}V_{\tau(k)} G_{j_{\tau(k)}j}
\end{equation}
for any $V_{1},\ldots, V_{k}\in \CC^{n\times n}$ and double indices $\sigma_p =i_pj_p$. 
   Using the local law \eqref{eq:32} together with \eqref{eq:93}, \eqref{eq:340} and moment bounds \eqref{eq:4} we obtain for $\sigma_{1},\sigma_{2}\in \{il, li\}$ the estimates
   \begin{equation}
     \label{eq:341}
     \Wt_{il} \nabla^{\sigma_{1}}_{\Wt_{\sigma_{1}}} \nabla^{\sigma_{2}}_{\Wt_{\sigma_{2}}}  G_{lj} (\Wb)
     =
     \left\{
       \begin{array}{ll}
         O_{\prec}\Big( \frac{1}{N^{3/2}} \Big)
         \, , & \quad l=j\, , 
         \\ 
         O_{\prec}\bigg( \frac{1}{\sqrt{N |\Im z|}}\, \frac{1}{N^{3/2}} \bigg) \,,              &
                \quad l\neq j \,.
       \end{array}
       \right.
   \end{equation}
   uniformly on $z\in \Omega$.
   When $z \in \Omega$ the functions $G_{ij}(z)$ are deterministically bounded by $N$, thus the stochastic domination bounds \eqref{eq:341} also hold in expectation.
   Therefore, if in \eqref{eq:338} we take the sum with respect to $l \in \{1,\ldots, N\}$ and use the estimate \eqref{eq:341} in expectation, we obtain the bound
   \begin{equation}
     \label{eq:342}
     \bigg\|\EE\Big[\frac{1}{2} \Wbt  (\nabla_{\Wbt})^2 \Gb \Big]\bigg\|_{\max}
     =
     O_{\prec}\Big(\frac{1}{N |\Im z|^{1/2}} \Big)
     =
     O_{\prec}\Big(\frac{N^{\tau/2}}{N \sqrt{\eta_{0}}} \Big)
     .
   \end{equation}

   To analyze the case $\Fcc_{2}(\Wb) = \Gb \efr(t)$, $F_{2,lj}(\Wb) = G_{lj}(z)\efr(t)$, we derive convenient formulas for the repeated directional derivatives of $\efr(t)$, similar to \eqref{eq:340}.
For this, we define 
\begin{equation*}
\begin{split}
  \Ic([\sigma,V]) &:= \sum_{\tau \in S_k}(-1)^k \frac{\imu t}{\pi} \int\displaylimits_\Omega \frac{\partial \tilde{f}(\zeta)}{\partial \overline{\zeta}}\Tr
  \sum_{l=1}^N G_{li_{\tau(1)}}(\zeta) V_{\tau(1)}G_{j_{\tau(1)}i_{\tau(2)}}(\zeta) \cdots V_{\tau(k)}G_{j_{\tau(k)}l}(\zeta) \dzeta
\\
&=
\frac{(-1)^{k+1}}{k}\frac{\imu t}{\pi} \sum_{\tau \in S_k}\int\displaylimits_\Omega \frac{\partial \tilde{f}(\zeta)}{\partial \overline{\zeta}}\frac{\partial}{\partial \zeta}\Tr G_{j_{\tau(k)}i_{\tau(1)}}(\zeta) V_{\tau(1)} \cdots G_{j_{\tau(k-1)}i_{\tau(k)}}(\zeta)V_{\tau(k)}
\dzeta
\end{split}
\end{equation*}
for  a multiset $[\sigma,V]: = [(\sigma_1,V_1),\dots, (\sigma_k,V_k)]$, where $\sigma_{p} = i_{p}j_{p}$ are double indices, and $V_{1},\ldots, V_{k} \in \CC^{n\times n}$. 
In particular, we have
\begin{equation*}
  \nabla^{\sigma_0}_{V_0}\Ic([{\sigma},{V}])
  =
  \Ic\big( \big[(\sigma_0, V_{0}), (\sigma_{1}, V_{1}), \ldots, (\sigma_{k},V_{k}) \big] \big)\,.  
\end{equation*}
Using this and $\nabla^{\sigma_1}_{V_1} \efr(t) = \efr(t) \Ic\big([(\sigma_{1},V_{1})]\big)$ we see by induction that
\begin{equation}
  \label{eq:346}
  \nabla^{\sigma_1}_{V_1} \dots \nabla^{\sigma_k}_{V_k}  \efr(t)
  =
  \efr(t)\sum_{\pi \in  \Pc([\sigma,V])}\Ic(\pi)\,,   
\end{equation}
where $\Pc([\sigma,V])=\mathcal{P}([(\sigma_1,V_1),\dots, (\sigma_k,V_k)])$ denotes all partitions of $[(\sigma_1,V_1),\dots, (\sigma_k,V_k)]$, and for $\pi \in \mathcal{P}([\sigma,V])$  we set
\begin{equation}
  \label{eq:347}
  \mathcal{I}(\pi) = \prod_{\lambda \in \pi }\mathcal{I}(\lambda)\,.
\end{equation}
From \eqref{eq:341} and $|\efr(t)|\leq 1$ we have that for all $i,j,l\in \{1,\ldots, N\}$ and $\sigma_{1},\sigma_{2} \in\{il,li\}$  the estimate
\begin{equation}
  \label{eq:348}
  \efr(t) \Wt_{il} \nabla^{\sigma_{1}}_{\Wt_{\sigma_{1}}} \nabla^{\sigma_{2}}_{\Wt_{\sigma_{2}}}  G_{lj}
  =
  \left\{
    \begin{array}{ll}
      O_{\prec}\Big( \frac{1}{N^{3/2}} \Big)
      \, , & \quad l=j\, ,
      \\ 
      O_{\prec}\bigg( \frac{N^{\tau/2}}{\sqrt{N \eta_{0}}}\, \frac{1}{N^{3/2}} \bigg)
      \, , & \quad l\neq j \, ,
    \end{array}
    \right.
\end{equation}
   holds uniformly on $z\in \Omega$.
   For the mixed derivatives \eqref{eq:340} and \eqref{eq:346} yield
   \begin{equation}
     \label{eq:349}
     \Wt_{il} \Big(\nabla^{\sigma_{1}}_{\Wt_{\sigma_{1}}}  G_{lj} \Big) \Big( \nabla^{\sigma_{2}}_{\Wt_{\sigma_{2}}} \efr(t)\Big)
     =
     -\Wt_{il}  G_{li_{1}}  \Wt_{i_{1}j_{1}} G_{j_{1}j}  \, \efr(t) \, \Ic\big([(\sigma_{2},\Wt_{\sigma_{2}})]\big)
     .
   \end{equation}
      The local law \eqref{eq:32} implies that for any $i,l \in \{1,\ldots, N\}$, $\sigma_{2}=i_2j_2 \in \{il, li\}$, and  $A \in \CC^{n\times n}$ 
   \begin{multline}
     \label{eq:350}
     \Ic\big([(\sigma_{2},A)]\big) =  \frac{\imu t}{\pi} \int_{\Omega} \frac{\partial \ft(\zeta)}{\partial \overline{\zeta}} \frac{\partial}{\partial \zeta} \Tr \Big[ G_{j_{2}i_{2}}(\zeta) A \Big] \dzeta
     =
     \left\{
       \begin{array}{ll}
         O_{\prec}\Big(N^{\tau}|t|  \|A\|\Big),& \,l=i,
         \\
         O_{\prec}\bigg(\frac{N^{3\tau/2}|t| \|A\|}{\sqrt{N\eta_0}}\bigg)
         ,& \,l\neq i,
       \end{array}
       \right.
   \end{multline}
    where we applied Stokes' theorem to rewrite the above expression using a contour integral as in \eqref{eq:87}.
    Combining \eqref{eq:349}, \eqref{eq:350}, the local law bounds \eqref{eq:32} for $G$  and the bounds for $\Wt$ in \eqref{eq:2}-\eqref{eq:4} and \eqref{eq:93} we get
    \begin{equation}
      \label{eq:351}
      \Wt_{il} \Big(\nabla^{\sigma_{1}}_{\Wt_{\sigma_{1}}}  G_{lj} \Big) \Big( \nabla^{\sigma_{2}}_{\Wt_{\sigma_{2}}} \efr(t)\Big)
      =
      \left\{
        \begin{array}{ll}
          O_{\prec}\bigg(\frac{N^{\tau}|t|}{N^{3/2}}\bigg)
          \, ,& \quad l = i,
          \\
          O_{\prec}\bigg(\frac{1}{N^{3/2}} \, \frac{N^{3\tau/2}|t|}{\sqrt{N\eta_{0}}} \bigg)
          \, , & \quad l \neq i \, .
        \end{array}
        \right.
      \end{equation}

      Finally, using \eqref{eq:346}, the structure of $\Ic[\sigma,V]$, and applying Stoke's theorem again  we find that for all $\sigma_{1},\sigma_{2} \in \{il, li\}$
      \begin{equation*}
        \Ic\Big(\big[(\sigma_{1}, \Wt_{\sigma_{1}}),(\sigma_{2},\Wt_{\sigma_{2}})\big]\Big)
        =
        O_{\prec}\bigg(\frac{N^{\tau} |t|}{N} \bigg)
        .
      \end{equation*}
      Together with \eqref{eq:350}, \eqref{eq:346} and \eqref{eq:347}
      this gives the bound for the second order derivatives
      \begin{align}
        \label{eq:353}
        \nabla^{\sigma_{1}}_{\Wt_{\sigma_{1}}}  \nabla^{\sigma_{2}}_{\Wt_{\sigma_{2}}} \efr(t)
        =
         O_{\prec}\bigg(\frac{N^{2\tau}|t|^2}{N} + \frac{N^{3\tau}|t|^2}{N^2 \eta_{0}} + \frac{N^{\tau}|t|}{N} \bigg)
        =
        O_{\prec}\bigg(\frac{N^{2\tau}(|t|+|t|^2)}{N} \bigg)
        \, ,
      \end{align}
      where we also used that $\tau < 1-\gamma$ to absorb all three terms in \eqref{eq:353} into one.
        Now \eqref{eq:353}, \eqref{eq:32}, \eqref{eq:2}-\eqref{eq:4} and \eqref{eq:93} imply
        \begin{equation}
          \label{eq:354}
        \Wt_{il} G_{lj}\nabla^{\sigma_{1}}_{\Wt_{\sigma_{1}}}  \nabla^{\sigma_{2}}_{\Wt_{\sigma_{2}}} \efr(t)
        =
        \left\{
          \begin{array}{ll}
            O_{\prec}\bigg(\frac{N^{2\tau}(|t|+|t|^2)}{N^{3/2}} \bigg)
            \, , & \quad l = j \, ,
            \\
            O_{\prec}\bigg(\frac{N^{2\tau}(|t|+|t|^2)}{N^{3/2}} \, \frac{N^{\tau/2}}{\sqrt{N \eta_{0}}} \bigg)
            \, , & \quad l \neq j \, .
          \end{array}
          \right.
        \end{equation}
        From the Leibniz rule, \eqref{eq:348}, \eqref{eq:351} and \eqref{eq:354} we establish the bound of the term in \eqref{eq:338} for $F_{2,lj} = G_{lj} \efr(t)$, namely
        \begin{align*}
          \sum_{l=1}^{N} \sum_{\sigma_{1},\sigma_{2}\in\{il,li\}}  \Wt_{il} \nabla^{\sigma_{1}}_{\Wt_{\sigma_{1}}} \nabla^{\sigma_{2}}_{\Wt_{\sigma_{2}}}  \Big( G_{lj} \efr(t) \Big)
          &=
            O_{\prec}\bigg( \frac{N^{2\tau}(1+|t|^2)}{N^{3/2}} + \frac{N^{5\tau/2}(1+|t|^2)}{N\sqrt{\eta_{0}}} \bigg)
            \\
          &=
          O_{\prec}\bigg(\frac{N^{5\tau/2}(1+|t|^2)}{N\sqrt{\eta_{0}}} \bigg)
        \end{align*}
        holding uniformly for $z\in \Omega$ and  $i,j \in \{1,\ldots, N\}$.
        Using the same argument as in the case $\star = 1$ to extend the above bound to its expectation we conclude that
        \begin{equation}
          \label{eq:357}
          \bigg\|\EE\Big[\frac{1}{2} \Wbt  (\nabla_{\Wbt})^2 \big(\Gb \efr(t) \big) \Big]\bigg\|_{\max}
          =
          O_{\prec}\bigg(\frac{N^{5\tau/2}(1+|t|^2)}{N\sqrt{\eta_{0}}} \bigg)
        \end{equation}
        uniformly for $z\in \Omega$.

It remains to estimate the last term in \eqref{eq:334}.
The local law bound $\|\Gb\|_{\max} \prec 1$ implies
\begin{equation*}
  |\nabla^{\sigma_1}_{V_1} \dots \nabla^{\sigma_k}_{V_k}  \efr(t)|
  \prec
  (1+ |t|^k)N^{k\tau} \prod_{i=1}^{k}\|V_{i}\|
  \,, \qquad
  \| \nabla^{\sigma_1}_{V_1} \dots \nabla_{V_k}^{\sigma_k} G_{ab} \|
  \prec
   \prod_{i=1}^{k}\|V_{i}\|
  .
\end{equation*}
If we now use $\|W_{ab}\|\prec N^{-1/2}$ following from \eqref{eq:2}-\eqref{eq:4}, we see that the term in the remainder of \eqref{eq:334} satisfies the bound
\begin{equation*}
  \| K^{\nabla,(\sigma_{1},\sigma_{2},\sigma_{3})}_{4,s}[W_{il},(W_{\sigma_{1}},W_{\sigma_{2}},W_{\sigma_{3}})]  F_{\star,lj} (sW_{il}, sW_{li})\|
  =
  O_{\prec}\Big(\frac{(1+|t|^{3})N^{3\tau}}{N^{2}}\Big)
\end{equation*}
uniformly for $i,l,j\in \{1,\ldots, N\}$, $\sigma_{1},\sigma_{2},\sigma_{3}\in \{il, li\}$, $s\in (0,1)$ and $z\in \Omega$.
After plugging this bound into \eqref{eq:334}, taking the sum for $l\in \{1,\ldots, N\}$ in \eqref{eq:333} and using \eqref{eq:338} we get the following expansion formula
\begin{align*}
  \EE \Big[ \Wb \Fcc_{\star}(\Wb)\Big]
  =
        \EE \Big[ \Wbt \, \nabla_{\Wbt} \Fcc_{\star}(\Wb) \Big]
     +
        \frac{1}{2}  \EE\Big[\Wbt  (\nabla_{\Wbt})^2  \Fcc_{\star}(\Wb)\Big]
        + \errA(z)
        ,
    \end{align*}
with $\|\errA(z)\|_{\rm max} \lesssim (1+|t|^3)N^{3\tau} N^{-1} $. 
We finish the proof by denoting
\begin{equation*}
  \errJ_{\star}(z)
  :=
  \frac{1}{2}  \EE\Big[\Wbt  (\nabla_{\Wbt})^2  \Fcc_{\star}(\Wb)\Big]
        + \errA(z)
      \end{equation*}
      and using the bounds \eqref{eq:342} and \eqref{eq:357}.
\end{proof}

\bibliographystyle{abbrv}
\bibliography{./bib.bib}

\begin{thebibliography}{10}

\bibitem{AjanErdoKrug17}
O.~H. Ajanki, L.~Erd\H{o}s, and T.~Kr\"uger.
\newblock Universality for general {W}igner-type matrices.
\newblock {\em Probab. Theory Related Fields}, 169(3-4):667--727, 2017.

\bibitem{AjanErdoKrug19}
O.~H. Ajanki, L.~Erd\H{o}s, and T.~Kr\"{u}ger.
\newblock Stability of the matrix {D}yson equation and random matrices with
  correlations.
\newblock {\em Probab. Theory Related Fields}, 173(1-2):293--373, 2019.

\bibitem{AljaRenfSter15}
J.~Aljadeff, D.~Renfrew, and M.~Stern.
\newblock Eigenvalues of block structured asymmetric random matrices.
\newblock {\em J. Math. Phys.}, 56(10):103502, 14, 2015.

\bibitem{AltErdoKrug20}
J.~Alt, L.~Erd\H{o}s, and T.~Kr\"{u}ger.
\newblock The {D}yson equation with linear self-energy: spectral bands, edges
  and cusps.
\newblock {\em Doc. Math.}, 25:1421--1539, 2020.

\bibitem{AltErdoKrugNemi_Kronecker}
J.~Alt, L.~Erd\H{o}s, T.~Kr\"{u}ger, and {\relax Yu}.~Nemish.
\newblock Location of the spectrum of {K}ronecker random matrices.
\newblock {\em Ann. Inst. Henri Poincar\'{e} Probab. Stat.}, 55(2):661--696,
  2019.

\bibitem{Ande13}
G.~W. Anderson.
\newblock Convergence of the largest singular value of a polynomial in
  independent {W}igner matrices.
\newblock {\em Ann. Probab.}, 41(3B):2103--2181, 2013.

\bibitem{Ande15}
G.~W. Anderson.
\newblock A local limit law for the empirical spectral distribution of the
  anticommutator of independent {W}igner matrices.
\newblock {\em Ann. Inst. Henri Poincar\'e Probab. Stat.}, 51(3):809--841,
  2015.

\bibitem{BaiYao05}
Z.~Bai and J.~Yao.
\newblock {On the convergence of the spectral empirical process of Wigner
  matrices}.
\newblock {\em Bernoulli}, 11(6):1059 -- 1092, 2005.

\bibitem{BaiSilv04}
Z.~D. Bai and J.~W. Silverstein.
\newblock C{LT} for linear spectral statistics of large-dimensional sample
  covariance matrices.
\newblock {\em Ann. Probab.}, 32(1A):553--605, 2004.

\bibitem{BannCebr23}
M.~Banna and G.~C\'{e}bron.
\newblock Operator-valued matrices with free or exchangeable entries.
\newblock {\em Ann. Inst. Henri Poincar\'{e} Probab. Stat.}, 59(1):503--537,
  2023.

\bibitem{BaoSchnXu22}
Z.~Bao, K.~Schnelli, and Y.~Xu.
\newblock Central limit theorem for mesoscopic eigenvalue statistics of the
  free sum of matrices.
\newblock {\em Int. Math. Res. Not. IMRN}, (7):5320--5382, 2022.

\bibitem{BaueBourNikuYau19}
R.~Bauerschmidt, P.~Bourgade, M.~Nikula, and H.-T. Yau.
\newblock The two-dimensional {C}oulomb plasma: quasi-free approximation and
  central limit theorem.
\newblock {\em Adv. Theor. Math. Phys.}, 23(4):841--1002, 2019.

\bibitem{BekeLodh18}
F.~Bekerman and A.~Lodhia.
\newblock Mesoscopic central limit theorem for general {$\beta$}-ensembles.
\newblock {\em Ann. Inst. Henri Poincar\'{e} Probab. Stat.}, 54(4):1917--1938,
  2018.

\bibitem{BeliCapiDallFevr}
S.~{Belinschi}, M.~{Capitaine}, S.~{Dallaporta}, and M.~{Fevrier}.
\newblock {Fluctuations of the Stieltjes transform of the empirical spectral
  distribution of selfadjoint polynomials in Wigner and deterministic diagonal
  matrices}.
\newblock ArXiv:2107.10031, 2021.

\bibitem{BersReutBook}
J.~Berstel and C.~Reutenauer.
\newblock {\em Noncommutative rational series with applications}, volume 137 of
  {\em Encyclopedia of Mathematics and its Applications}.
\newblock Cambridge University Press, Cambridge, 2011.

\bibitem{BourErdoYau14a}
P.~Bourgade, L.~Erd\H{o}s, and H.-T. Yau.
\newblock Universality of general {$\beta$}-ensembles.
\newblock {\em Duke Math. J.}, 163(6):1127--1190, 2014.

\bibitem{BoutKhor99a}
A.~Boutet~de Monvel and A.~Khorunzhy.
\newblock Asymptotic distribution of smoothed eigenvalue density. {I}.
  {G}aussian random matrices.
\newblock {\em Random Oper. Stochastic Equations}, 7(1):1--22, 1999.

\bibitem{BoutKhor99b}
A.~Boutet~de Monvel and A.~Khorunzhy.
\newblock Asymptotic distribution of smoothed eigenvalue density. {II}.
  {W}igner random matrices.
\newblock {\em Random Oper. Stochastic Equations}, 7(2):149--168, 1999.

\bibitem{CipoErdoHenhKolu}
G.~Cipolloni, L.~Erd\H{o}s, J.~Henheik, and O.~Kolupaiev.
\newblock Gaussian fluctuations in the equipartition principle for wigner
  matrices.
\newblock ArXiv:2301.05181, 2023.

\bibitem{CipoErdoHenhSchr}
G.~Cipolloni, L.~Erd\H{o}s, J.~Henheik, and D.~Schr\"{o}der.
\newblock Optimal lower bound on eigenvector overlaps for non-hermitian random
  matrices.
\newblock ArXiv:2301.03549, 2023.

\bibitem{CipoErdoSchr21}
G.~Cipolloni, L.~Erd\H{o}s, and D.~Schr\"{o}der.
\newblock Eigenstate thermalization hypothesis for {W}igner matrices.
\newblock {\em Comm. Math. Phys.}, 388(2):1005--1048, 2021.

\bibitem{CipoErdoSchr}
G.~Cipolloni, L.~Erd\H{o}s, and D.~Schr\"{o}der.
\newblock Mesoscopic central limit theorem for non-hermitian random matrices.
\newblock ArXiv:2210.12060, 2022.

\bibitem{CipoErdoSchr22}
G.~Cipolloni, L.~Erd\H{o}s, and D.~Schr\"{o}der.
\newblock Optimal multi-resolvent local laws for {W}igner matrices.
\newblock {\em Electron. J. Probab.}, 27:Paper No. 117, 38, 2022.

\bibitem{CipoErdoSchr23}
G.~Cipolloni, L.~Erd\H{o}s, and D.~Schr\"{o}der.
\newblock Functional central limit theorems for {W}igner matrices.
\newblock {\em Ann. Appl. Probab.}, 33(1):447--489, 2023.

\bibitem{CollMingSniaSpei07}
B.~Collins, J.~A. Mingo, P.~\'{S}niady, and R.~Speicher.
\newblock Second order freeness and fluctuations of random matrices. {III}.
  {H}igher order freeness and free cumulants.
\newblock {\em Doc. Math.}, 12:1--70, 2007.

\bibitem{DuitJoha18}
M.~Duits and K.~Johansson.
\newblock On mesoscopic equilibrium for linear statistics in {D}yson's
  {B}rownian motion.
\newblock {\em Mem. Amer. Math. Soc.}, 255(1222):v+118, 2018.

\bibitem{ErdoKnow15b}
L.~Erd\H{o}s and A.~Knowles.
\newblock The {A}ltshuler-{S}hklovskii formulas for random band matrices {II}:
  {T}he general case.
\newblock {\em Ann. Henri Poincar\'{e}}, 16(3):709--799, 2015.

\bibitem{ErdoKrugNemi_Poly}
L.~Erd\H{o}s, T.~Kr\"{u}ger, and {\relax Yu}.~Nemish.
\newblock Local laws for polynomials of {W}igner matrices.
\newblock {\em J. Funct. Anal.}, 278(12):108507, 59, 2020.

\bibitem{ErdoKrugNemi_qDot}
L.~Erd\H{o}s, T.~Kr\"{u}ger, and {\relax Yu}.~Nemish.
\newblock Scattering in quantum dots via noncommutative rational functions.
\newblock {\em Ann. Henri Poincar\'{e}}, 22(12):4205--4269, 2021.

\bibitem{ErdoKrugRenf}
L.~Erd\H{o}s, T.~Kr\"{u}ger, and D.~Renfrew.
\newblock Randomly coupled differential equations with elliptic correlations.
\newblock ArXiv:1908.05178, 2019.

\bibitem{ErdoKrugSchr19}
L.~Erd\H{o}s, T.~Kr\"{u}ger, and D.~Schr\"{o}der.
\newblock Random matrices with slow correlation decay.
\newblock {\em Forum Math. Sigma}, 7:e8, 89, 2019.

\bibitem{ErdoSchlYau09a}
L.~Erd{\H{o}}s, B.~Schlein, and H.-T. Yau.
\newblock Local semicircle law and complete delocalization for {W}igner random
  matrices.
\newblock {\em Comm. Math. Phys.}, 287(2):641--655, 2009.

\bibitem{FyodKhorSimm16}
Y.~V. Fyodorov, B.~A. Khoruzhenko, and N.~J. Simm.
\newblock Fractional {B}rownian motion with {H}urst index {$H=0$} and the
  {G}aussian unitary ensemble.
\newblock {\em Ann. Probab.}, 44(4):2980--3031, 2016.

\bibitem{GargGurvOlivWigd20}
A.~Garg, L.~Gurvits, R.~Oliveira, and A.~Wigderson.
\newblock Operator scaling: theory and applications.
\newblock {\em Found. Comput. Math.}, 20(2):223--290, 2020.

\bibitem{GeszTsek00}
F.~Gesztesy and E.~Tsekanovskii.
\newblock On matrix-valued {H}erglotz functions.
\newblock {\em Math. Nachr.}, 218:61--138, 2000.

\bibitem{HaagThor99}
U.~Haagerup and S.~Thorbj\o{}rnsen.
\newblock Random matrices and {$K$}-theory for exact {$C^\ast$}-algebras.
\newblock {\em Doc. Math.}, 4:341--450, 1999.

\bibitem{HaagThor05}
U.~Haagerup and S.~Thorbj\o{}rnsen.
\newblock A new application of random matrices: {${\rm Ext}(C^*_{\rm
  red}(F_2))$} is not a group.
\newblock {\em Ann. of Math. (2)}, 162(2):711--775, 2005.

\bibitem{HastJuhaSchr92}
H.~M. Hastings, F.~Juhasz, and M.~A. Schreiber.
\newblock Stability of structured random matrices.
\newblock {\em Proceedings: Biological Sciences}, 249(1326):223--225, 1992.

\bibitem{HeKnow17}
Y.~He and A.~Knowles.
\newblock Mesoscopic eigenvalue statistics of {W}igner matrices.
\newblock {\em Ann. Appl. Probab.}, 27(3):1510--1550, 2017.

\bibitem{HeKnowRose18}
Y.~He, A.~Knowles, and R.~Rosenthal.
\newblock Isotropic self-consistent equations for mean-field random matrices.
\newblock {\em Probab. Theory Related Fields}, 171(1-2):203--249, 2018.

\bibitem{HeltMaiSpei18}
J.~W. Helton, T.~Mai, and R.~Speicher.
\newblock Applications of realizations (aka linearizations) to free
  probability.
\newblock {\em J. Funct. Anal.}, 274(1):1--79, 2018.

\bibitem{HeltRaFaSpei07}
J.~W. Helton, R.~Rashidi~Far, and R.~Speicher.
\newblock Operator-valued semicircular elements: solving a quadratic matrix
  equation with positivity constraints.
\newblock {\em Int. Math. Res. Not. IMRN}, (22):Art. ID rnm086, 15, 2007.

\bibitem{HuanLand19}
J.~Huang and B.~Landon.
\newblock Rigidity and a mesoscopic central limit theorem for {D}yson
  {B}rownian motion for general {$\beta$} and potentials.
\newblock {\em Probab. Theory Related Fields}, 175(1-2):209--253, 2019.

\bibitem{JiLee20}
H.~C. Ji and J.~O. Lee.
\newblock Gaussian fluctuations for linear spectral statistics of deformed
  {W}igner matrices.
\newblock {\em Random Matrices Theory Appl.}, 9(3):2050011, 79, 2020.

\bibitem{Joha98}
K.~Johansson.
\newblock {On fluctuations of eigenvalues of random Hermitian matrices}.
\newblock {\em Duke Mathematical Journal}, 91(1):151 -- 204, 1998.

\bibitem{KhorKhorPast96}
A.~M. Khorunzhy, B.~A. Khoruzhenko, and L.~A. Pastur.
\newblock Asymptotic properties of large random matrices with independent
  entries.
\newblock {\em J. Math. Phys.}, 37(10):5033--5060, 1996.

\bibitem{KlepPascVolc17}
I.~Klep, J.~E. Pascoe, and J.~Vol\v{c}i\v{c}.
\newblock Regular and positive noncommutative rational functions.
\newblock {\em J. Lond. Math. Soc. (2)}, 95(2):613--632, 2017.

\bibitem{KrugNemi_mesoCLTpoly}
T.~Kr\"{u}ger and {\relax Yu}.~Nemish.
\newblock {Mesoscopic spectral CLT for quadratic polynomials in random
  matrices}.
\newblock In preparation, 2023.

\bibitem{Lamb18}
G.~Lambert.
\newblock Mesoscopic fluctuations for unitary invariant ensembles.
\newblock {\em Electron. J. Probab.}, 23:Paper No. 7, 33, 2018.

\bibitem{LandLopaSoso23}
B.~Landon, P.~Lopatto, and P.~Sosoe.
\newblock {Single eigenvalue fluctuations of general Wigner-type matrices}.
\newblock {\em Probability Theory and Related Fields}, pages 1--62, 1 2023.

\bibitem{LandSoso20}
B.~Landon and P.~Sosoe.
\newblock Applications of mesoscopic {CLT}s in random matrix theory.
\newblock {\em Ann. Appl. Probab.}, 30(6):2769--2795, 2020.

\bibitem{LeeSchnStetYau16}
J.~O. Lee, K.~Schnelli, B.~Stetler, and H.-T. Yau.
\newblock {Bulk universality for deformed Wigner matrices}.
\newblock {\em The Annals of Probability}, 44(3):2349 -- 2425, 2016.

\bibitem{LiSchnXu21}
Y.~Li, K.~Schnelli, and Y.~Xu.
\newblock Central limit theorem for mesoscopic eigenvalue statistics of
  deformed {W}igner matrices and sample covariance matrices.
\newblock {\em Ann. Inst. Henri Poincar\'{e} Probab. Stat.}, 57(1):506--546,
  2021.

\bibitem{LiXu21}
Y.~Li and Y.~Xu.
\newblock On fluctuations of global and mesoscopic linear statistics of
  generalized {W}igner matrices.
\newblock {\em Bernoulli}, 27(2):1057--1076, 2021.

\bibitem{LodhSimm15}
A.~Lodhia and N.~J. Simm.
\newblock {Mesoscopic linear statistics of Wigner matrices}.
\newblock ArXiv:1503.03533, 2015.

\bibitem{MaiSpeiYin2}
T.~{Mai}, R.~{Speicher}, and S.~{Yin}.
\newblock {The free field: realization via unbounded operators and Atiyah
  property}.
\newblock ArXiv:1905.08187, 2019.

\bibitem{RajaAbbo06}
K.~Rajan and L.~F. Abbott.
\newblock Eigenvalue spectra of random matrices for neural networks.
\newblock {\em Phys. Rev. Lett.}, 97:188104, Nov 2006.

\bibitem{Riab}
V.~Riabov.
\newblock Mesoscopic eigenvalue statistics for wigner-type matrices.
\newblock ArXiv:2301.01712, 2023.

\bibitem{SkelIwasGrigBook}
R.~E. Skelton, T.~Iwasaki, and K.~M. Grigoriadis.
\newblock {\em A unified algebraic approach to linear control design}.
\newblock The Taylor \& Francis Systems and Control Book Series. Taylor \&
  Francis Group, London, 1998.

\bibitem{Spee83}
T.~P. Speed.
\newblock Cumulants and partition lattices.
\newblock {\em Austral. J. Statist.}, 25(2):378--388, 1983.

\bibitem{Thor00}
S.~Thorbj\o{}rnsen.
\newblock Mixed moments of {V}oiculescu's {G}aussian random matrices.
\newblock {\em J. Funct. Anal.}, 176(2):213--246, 2000.

\bibitem{WolkSaigVandBook}
H.~Wolkowicz, R.~Saigal, and L.~Vandenberghe, editors.
\newblock {\em Handbook of semidefinite programming}, volume~27 of {\em
  International Series in Operations Research \& Management Science}.
\newblock Kluwer Academic Publishers, Boston, MA, 2000.
\newblock Theory, algorithms, and applications.

\end{thebibliography}

\end{document}